\documentclass{amsart}

\usepackage[T1]{fontenc}
\usepackage[utf8]{inputenc}
\usepackage[hidelinks]{hyperref}
\usepackage{cleveref}
\usepackage{amsthm}
\usepackage{amssymb}
\usepackage{bm}
\usepackage{quiver}

\usepackage{import}

\usepackage{graphicx}

\usepackage{enumitem} 

\usepackage{tikzit}

\tikzstyle{dotdotdot}=[fill=white, draw=black, shape=rectangle]
\tikzstyle{new style 0}=[fill=black, draw=none, shape=circle]

\tikzstyle{new edge style 0}=[-, densely dashed]
\tikzstyle{new edge style 2}=[-, draw=blue]
\tikzstyle{new edge style 3}=[->]

\tikzstyle{green}=[text={black!30!green}]
\tikzstyle{blue}=[text=blue]
\tikzstyle{red}=[text=red]
\tikzstyle{puncture}=[fill=white, draw=red, shape=circle, minimum size=1pt]
\tikzstyle{blackpuncture}=[fill=white, draw=black, shape=circle, minimum size=1pt]
\tikzstyle{cyan}=[text=cyan]

\tikzstyle{markings}=[-, draw=red, line width=1pt, line cap=round]
\tikzstyle{overbraid}=[-, draw=white, fill=none, line width=6pt]
\tikzstyle{thick}=[-, line width=2pt, draw=blue]
\tikzstyle{dashedline}=[-, dashed]
\tikzstyle{dottedline}=[-, dash pattern=on 0.75pt off 0.75pt, line width=0.75pt]
\tikzstyle{thin red}=[-, line width=0.25pt, draw=black]
\tikzstyle{tangle}=[-, draw=blue, line width=1pt, fill={blue!20}]
\tikzstyle{scc}=[-, draw={black!30!green}, fill={blue!20}, line width=1pt]
\tikzstyle{inner boundary}=[-, fill=white]
\tikzstyle{outer boundary}=[-, fill={red!20}]
\tikzstyle{lowerboundery}=[-, line width=1.5pt, line cap=round, draw=red]
\tikzstyle{upperboundery}=[-, line width=1.5pt, line cap=round, draw=blue]
\tikzstyle{dottedcycle}=[-, draw=blue, dash pattern=on 0.5pt off 1pt on 4pt off 1pt, decoration={markings, mark=at position 0.5 with {\arrow{>}}}, postaction=decorate]
\tikzstyle{cycle}=[-, draw=blue, decoration={markings, mark=at position 0.5 with {\arrow{>}}}, postaction=decorate]
\tikzstyle{path}=[-, draw=cyan, line width=0.25pt]
\tikzstyle{arrowpath}=[-, draw=cyan, line width=0.25pt, decoration={markings, mark=at position 0.5 with {\arrow{>}}}, postaction=decorate]
\tikzstyle{orientedpath}=[-, line width=0.25pt, decoration={markings, mark=at position 0.5 with {\arrow{<}}}, postaction=decorate]
\tikzstyle{inner square}=[-, fill={blue!20}]
\tikzstyle{outer square}=[-, fill={red!20}]
\tikzstyle{blueline}=[-, draw=blue]
\tikzstyle{greenline}=[-, draw=green]
\tikzstyle{bluesquare}=[-, draw=blue, fill={blue!20}]
\usetikzlibrary{decorations.markings}

\usepackage{caption}
\usepackage{xparse}

\usepackage{nicematrix,booktabs}

\usepackage{stmaryrd}

\numberwithin{figure}{section}

\newtheorem{theorem}{Theorem}[section]
\newtheorem*{theorem*}{Theorem}

\newtheorem{lemma}[theorem]{Lemma}
\newtheorem*{lemma*}{Lemma}
\newtheorem{proposition}[theorem]{Proposition}
\newtheorem*{proposition*}{Proposition}

\newtheorem{conjecture}[theorem]{Conjecture}
\newtheorem*{conjecture*}{Conjecture}

\newtheorem{corollary}[theorem]{Corollary}
\newtheorem*{corollary*}{Corollary}

\theoremstyle{remark}
\newtheorem{remark}[theorem]{Remark}
\newtheorem{notation}[theorem]{Notation}

\newtheorem*{question*}{Question}
\theoremstyle{definition}
\newtheorem{definition}[theorem]{Definition}

\newtheorem*{definition*}{Definition}

\newcommand{\Cat}{\mathcal{C}}

\newcommand{\dol}{\overline{\partial}}

\newcommand{\cA}{\mathcal{A}}

\newcommand{\Cinf}{C^\infty}

\newcommand{\cE}{\mathcal{E}}

\newcommand{\uE}{\underline{E}}

\newcommand{\uG}{\underline{G}}

\newcommand{\Mod}{\mathfrak{mod}}

\newcommand{\Modcat}{\mathfrak{Mod}}

\newcommand{\SO}{\mathrm{SO}(3)}
\newcommand{\SU}{\mathrm{SU}(2)}

\newcommand{\sltwo}{\mathfrak{sl}_2}

\newcommand{\lra}{\longrightarrow}

\newcommand{\ra}{\rightarrow}

\newcommand{\Nu}{\mathcal{V}}

\newcommand{\ad}{\mathrm{ad}\,}

\newcommand{\PMod}[1]{\mathrm{PMod}(#1)}
\newcommand{\Modu}[1]{\mathrm{Mod}(#1)}

\newcommand{\tildeMod}[1]{\mathrm{M\tilde{o}d}(#1)}

\newcommand{\Modl}[2]{\mathrm{Mod}^{#1}(#2)}

\newcommand{\tModl}[2]{\mathrm{M\tilde{o}d}^{#1}(#2)}

\newcommand{\Mg}[2]{\mathcal{M}_{{#1},{#2}}}

\newcommand{\tMgb}[2]{\widetilde{\mathcal{M}}_{#1}^{#2}}

\newcommand{\oMg}[2]{\overline{\mathcal{M}}_{{#1},{#2}}}

\newcommand{\hatMg}[2]{\widehat{\mathcal{M}}_{{#1},{#2}}}

\newcommand{\Mgrp}[3]{\overline{\mathcal{M}}_{{#1},{#2}}(#3)}

\newcommand{\Cgrp}[3]{\overline{\mathcal{C}}_{{#1},{#2}}(#3)}

\newcommand{\Mgrb}[3]{\overline{\mathcal{M}}_{#1}^{#2}(#3)}

\newcommand{\Mgrbt}[4]{\overline{\mathcal{M}}_{#1}^{#2}(#3,#4)}

\newcommand{\Surf}{\mathrm{Surf}_\Lambda}
\newcommand{\tSurf}{\mathrm{S\tilde{u}rf}_\Lambda}

\newcommand{\Br}[1]{\mathrm{B}\mu_{#1}}

\newcommand{\End}[1]{\mathrm{End}(#1)}
\newcommand{\GL}[1]{\mathrm{GL}(#1)}

\newcommand{\GLn}[2]{\mathrm{GL}_{#1}(#2)}
\newcommand{\PU}[2]{\mathrm{PU}(#1,#2)}

\newcommand{\Aut}[1]{\mathrm{Aut}(#1)}

\newcommand{\Z}{\mathbb{Z}}

\newcommand{\C}{\mathbb{C}}
\newcommand{\Q}{\mathbb{Q}}

\newcommand{\id}{\mathrm{id}}

\newcommand{\eps}{\epsilon}

\newcommand{\ul}{\underline{\lambda}}
\newcommand{\umu}{\underline{\mu}}
\newcommand{\unu}{\underline{\nu}}

\newcommand{\um}{\underline{m}}

\newcommand{\uk}{\underline{k}}

\newcommand{\ueps}{\underline{\epsilon}}
\newcommand{\uzero}{\underline{0}}

\newcommand{\Lg}{\mathfrak{g}}

\title[Rigidity of \texorpdfstring{$\SU$}{SU(2)} quantum representations at prime levels]{Rigidity of \texorpdfstring{$\mathbf{SU(2)}$}{SU(2)} and \texorpdfstring{$\mathbf{SO(3)}$}{SO(3)} quantum representations of mapping class groups at prime levels}

\author{Pierre Godfard}

\begin{document}

\begin{abstract}
  We prove the rigidity of Witten-Reshetikhin-Turaev \texorpdfstring{$\SU$}{SU(2)} and \texorpdfstring{$\SO$}{SO(3)}
  quantum representations of mapping class groups at all prime levels for closed surfaces of genus at least \texorpdfstring{$7$}{7}.
  The proof relies on Ocneanu rigidity of modular categories and harmonic representatives in Hodge theory.
\end{abstract}

\maketitle


\section{Introduction}


It is expected that quantum representations of mapping class groups are rigid in genus $g\geq 3$, because of
Kazhdan’s property (T), yet to be proved for mapping class groups.
In this paper, we show rigidity
of Witten-Reshetikhin-Turaev $\SU$ and $\SO$ quantum representations at all prime levels for closed surfaces of genus at least \texorpdfstring{$7$}{7}.

We also provide detailed proofs of several connections among rigidity conjectures for mapping class groups $\Modu{S_g^n}$ as the parameters $g$ and $n$ vary.
These connections, while essentially known to experts and extending earlier results of Putman and Wieland on the Ivanov conjecture
\cite[step 2 of Lem. 3.2, Lem. 3.1]{putmanAbelianQuotientsSubgroups2013}, have not, to our knowledge,
been made explicit in the literature.
These relations emerge naturally from a key step in the proof of our main result.


\subsection{The main result on rigidity of \texorpdfstring{$\mathbf{SU(2)}$}{SU(2)} and \texorpdfstring{$\mathbf{SO(3)}$}{SO(3)} quantum representations at prime levels}

The mapping class group $\Modu{S}$ of a compact surface $S$ is the group 
$\pi_0(\mathrm{Homeo}^+(S,\partial S))$
of connected components of the group of orientation-preserving homeomorphisms of $S$
that fix the boundary $\partial S$ pointwise.
The property (T) conjecture for mapping class groups predicts that 
any finite dimensional unitary representations $\rho$ of $\Modu{S_g^n}$ for $g\geq 3$ is rigid,
i.e. $H^1(\Modu{S_g^n},\ad\rho)=0$.
While this question has been well studied for representations $\rho$ with finite image,
(then known as the Ivanov conjecture),
much less is known about unitary representations with infinite image (see \Cref{subsectionintroconjectures} for more context on these conjectures).
Such representations are rarer than those with finite image, and the most fruitful
way to construct them to date is through the theory of quantum representations.

Quantum representations are not constructed individually but rather in collections:
Reshetikhin-Turaev theory \cite{reshetikhinInvariants3manifoldsLink1991,turaevQuantumInvariantsKnots2016} associates to each modular category $\Cat$
a collection of projective representations $\rho_g(\lambda_1,\dotsc,\lambda_n):\Modu{S_g^n}\ra\mathrm{PGL}_{d_g(\lambda_1,\dotsc,\lambda_n)}(\C)$,
indexed by the genus $g$, the number of boundary components $n$, and the choice, for each boundary component $i\in \{1,2,\dotsc,n\}$,
of a simple object $\lambda_i$ of $\Cat$. Moreover, these representations exhibit compatibilities with respect to
the natural morphisms between mapping class groups. The data consisting of this collection of representations together with their compatibilities
forms a \emph{modular functor}, a structure that is equivalent to that of modular fusion category
\cite{bakalovLecturesTensorCategories2000,etingofRigidityNonnegligibleObjects2024}.
These $\rho_g(\lambda_1,\dotsc,\lambda_n)$ are always finite-twists representations,
and are always semisimple (see \cite{godfardSemisimplicityConformalBlocks2025a}).

In this article, we focus on the so-called $\SU$ and $\SO$ quantum representations.
The $\SU$ modular functors are indexed by an even integer $2\ell$ called level.
The $\SU$ modular category at level $2\ell$ has simple objects indexed by the integers $0,1,\dotsc,\ell-2$.
When $\ell$ is odd, it admits a variant, called the $\SO$ modular functor at level $\ell$.
The simple objects are then indexed by the even integers $0,2,\dotsc,\ell-3$.%
\footnote{Note that $\SU$ and $\SO$ \emph{do not} refer to the target groups of the corresponding quantum representations.}

More generally, modular categories can be associated to each choice of a semisimple Lie algebra $\Lg$ over $\C$
and an integer $\ell_c\geq 1$. The $\SU$ and $\SO$ categories are then variants of the case $\Lg=\sltwo$ for $\ell_c=\ell-2$.
The general construction relies either on quantum group representation theory or on the representation theory
of affine Lie algebras (Wess-Zumino-Witten conformal blocks), see \cite[§3.3, §7]{bakalovLecturesTensorCategories2000}.
In the $\SU$ and $\SO$ cases, a more direct and practical construction of the representations via skein module theory
is developed in \cite{blanchetThreemanifoldInvariantsDerived1992}.

The $\SU$ and $\SO$ representations have nice properties: they are unitary,
irreducible when $\ell$ is prime, and often have infinite image (this is the case, for example, for $\rho_g$ when $g\geq 2$ and $\ell\geq 5$ is prime).
The question of their rigidity in genus $g\geq 3$ therefore falls within the scope of the property (T) conjecture.
The main result of this paper is to verify this rigidity for $\rho_g$ when $\ell$ is prime and $g\geq 7$.

\begin{theorem*}[{\ref{mainresult}}]
 Let $\ell\geq 3$ be a prime number. Denote by $(\rho_g(\ul))_{g,\ul}$ the collection of
 $\SU$ quantum representations at level $2\ell$ or $\SO$ at level $\ell$.
 Then, for $g\geq 7$, $\rho_g$ is infinitesimally rigid, that is, $H^1(\Modu{S_g},\ad \rho_g)=0$.
\end{theorem*}

The result for $\ell=5$ is contained in the article \cite{godfardRigidityFibonacciRepresentations2025},
whose method reduces rigidity to finitely many cases in genus $0$ and $1$ that can be handled by hand.
This method is limited by the fact that rigidity in genus $0$ fails for $\ell\geq 7$.
The approach proposed here is different, more involved, and relies on four key ingredients:
the irreducibility of $\rho_{h}(\lambda_1,\dotsc,\lambda_n)$ for all $h$ and all $\lambda_1,\dotsc,\lambda_n$;
a partial homological stability result $H^1(\Modu{S_{h+1}},\ad \rho_{h+1})\hookrightarrow H^1(\Modu{S_h},\ad \rho_h)$ for $h\geq 4$;
Ocneanu rigidity \cite[Th. 2.28]{etingofFusionCategories2005}; and functoriality properties
of harmonic degree $1$ representatives in non-Abelian Hodge theory.
See the proof outline in \Cref{suboutline} for more details.

The restriction to $g\geq 7$ comes from the restriction $h\geq 4$ in the partial stability result
and from technical issues, see \Cref{remark7}. The restriction to the $\SU$ and $\SO$ modular categories and to prime $\ell$
is necessary because the proof uses the irreducibility of the associated quantum representations.
See \Cref{remarkassumptions} for the precise properties of these modular categories that are used.

To conclude this subsection, we discuss the geometric definition of modular functors
\emph{on twisted moduli spaces}, which is crucial in this article.
Fix a modular functor, that is, a collection of representations $(\rho_g(\lambda_1,\dotsc,\lambda_n))_{g,n,\lambda_1,\dotsc,\lambda_n}$
of mapping class groups equipped with compatibilities.
Since the fundamental groups of the Deligne-Mumford-Knudsen moduli spaces of curves
are mapping class groups, applying the Riemann-Hilbert correspondence to the representations $\rho_g(\lambda_1,\dotsc,\lambda_n)$ yields
local systems $\Nu_g(\lambda_1,\dotsc,\lambda_n)$
on stacks $\tMgb{g}{n}$, which are smooth but not proper \cite[6.4.1]{bakalovLecturesTensorCategories2000}.
These local systems are called \emph{conformal blocks}.

For $r$ sufficiently large, each representation $\rho_g(\lambda_1,\dotsc,\lambda_n)$ of $\Modu{S_g^n}$ factors through the quotient
$\Modl{r}{S_g^n}=\Modu{S_g^n}/\langle T_\gamma^r\mid \gamma\rangle$ by the $r$-th powers of Dehn twists (finite-twists representations).
This allows us to extend the local systems $\Nu_g(\lambda_1,\dotsc,\lambda_n)$ to stacks $\Mgrb{g}{n}{r}$ that compactify the $\tMgb{g}{n}$,
called twisted moduli spaces, whose fundamental groups are precisely the $\Modl{r}{S_g^n}$.
In this geometric framework, the compatibilities of the $\rho_g(\lambda_1,\dotsc,\lambda_n)$ with respect to natural maps between mapping class groups
translate into compatibilities of the $\Nu_g(\lambda_1,\dotsc,\lambda_n)$ with respect to natural maps between the moduli spaces $\Mgrb{g}{n}{r}$.
This gives a geometric definition of the notion of modular functor.


\subsection{Outline of the proof}\label{suboutline}


We explain and motivate the steps in the proof of \Cref{mainresult}.
Throughout this section, $\Nu$ denotes an $\SU$ modular functor at level $2\ell$ or an $\SO$ modular functor at level $\ell$, with $\ell\geq 3$ prime.
We adopt a geometric viewpoint on $\Nu$: it is a collection of local systems $\Nu_g(\lambda_1,\dotsc,\lambda_n)$
on the twisted moduli spaces $\Mgrb{g}{n}{r}$.\footnote{The projective anomaly is ignored here; see \Cref{sectionMF} for more details.}
The problem can then be reduced to showing that $H^1(\Mgrb{g}{}{r},\ad\Nu_g)=0$ for $g\geq 7$ (\Cref{propositionfiniteorderdeformation}).

We begin with the partial homological stability result in cohomological degree $1$.
\begin{proposition*}[\ref{propositioninjectivity}]
 For $h\geq 4$, the map
 $$H^1(\Mgrb{h+1}{}{r},\ad \Nu_{h+1})\lra H^1(\Mgrb{h}{}{r},\ad \Nu_h)$$
 is injective.
\end{proposition*}
This map is induced by $\Mgrb{h}{}{r}\leftarrow\Mgrb{h}{1}{r}\xrightarrow{j}\Mgrb{h+1}{}{r}$
and by the fact that $\ad\Nu_h$ is a direct summand of $j^*\ad \Nu_{h+1}$; see \Cref{propositioninjectivity}.
The key ingredient in proving this proposition is the following general fact,
that we prove using non-Abelian Hodge theory.
\begin{proposition*}[\ref{corollaryinjectivity}]:
 For $\delta\geq 1$, $g\geq 2\delta+1$, and any semisimple local system $\uE$ on $\Mgrb{g}{}{r}$,
 the map $H^1(\Mgrb{g}{}{r},\uE)\ra H^1(\Mgrb{g-\delta}{1}{r},j^*\uE)$ is injective,
 where $j$ denotes the map $\Mgrb{g-\delta}{1}{r}\ra\Mgrb{g-\delta}{1}{r}\times\Mgrb{\delta}{1}{r}\ra\Mgrb{g}{}{r}$
 induced by a choice of $*\in \Mgrb{\delta}{1}{r}$.
\end{proposition*}
The map $j$ is independent of the choice of $*$ up to homotopy.
The idea is to apply this fact simultaneously to $g=h+1$, $\delta=2$, $\uE=\ad\Nu_{h+1}$ and to $g=h$, $\delta=1$, $\uE=\ad\Nu_{h}$.
One can then verify that both $H^1(\Mgrb{h+1}{}{r},\ad \Nu_{h+1})$ and $H^1(\Mgrb{h}{}{r},\ad \Nu_h)$ inject
into $\bigoplus_{\lambda\in\Lambda} H^1(\Mgrb{h-1}{1}{r},\ad\Nu_{h-1}(\lambda))$, compatibly with the map
$H^1(\Mgrb{h+1}{}{r},\ad \Nu_{h+1})\ra H^1(\Mgrb{h}{}{r},\ad \Nu_h)$.

This result shows that the finite-dimensional vector spaces
$H^1(\Mgrb{h}{}{r},\ad \Nu_h)$ are isomorphic for sufficiently large $h$.
We denote by $H^1_\mathrm{stab}$ this stable cohomology group.
We first outline a proof that $H^1_\mathrm{stab}=0$, and then explain how to adapt this proof
to obtain that $H^1(\Mgrb{h}{}{r},\ad \Nu_h)=0$ for $h\geq 7$.

The main idea of the proof is to relate $H^1_\mathrm{stab}$ to the deformations of the modular category $\Cat$
associated to $\Nu$, and then to apply Ocneanu rigidity
\cite[Th. 2.28]{etingofFusionCategories2005}, which states that a modular category admits no nontrivial deformations.
By the equivalence between modular categories and modular functors, Ocneanu’s rigidity implies
that a modular functor likewise admits no deformations.
In concrete terms, this means that any collection
$(\tilde{\Nu}_g(\lambda_1,\dotsc,\lambda_n))_{g,n,\lambda_1,\dotsc,\lambda_n}$
of $\C[\eps]$-local systems satisfying the axioms of a modular functor
is isomorphic to the trivial deformation of a modular functor.
Note that this result does not imply rigidity of the individual $\Nu_g(\lambda_1,\dotsc,\lambda_n)$,
since the condition that the collection forms a modular functor imposes strong compatibility
relations among the various deformations.

Thus, to show that $H^1_\mathrm{stab}=0$, it suffices to construct an injective map
$$H^1_\mathrm{stab}\lra\{\text{deformations }\tilde{\Nu}\text{ of }\Nu\}/\text{iso}.$$
Fix $\psi\in H^1_\mathrm{stab}$, and let us construct its image.
We proceed in two steps. The first step is to construct, for each $\Nu_h(\ul)$,
a deformation class $\psi_h(\ul)\in H^1(\Mgrb{h}{n}{r},\ad\Nu_h(\ul))$, such that these classes
satisfy compatibility properties analogous to those of a modular functor.
The second step will be to lift these deformation classes to actual deformations,
while preserving the compatibility conditions.

Fix $h\geq 0$, $n\geq 1$, $\ul\in\Lambda^n$, and assume $\Nu_h(\ul)\neq 0$.
There exists a gluing map
$$i:\Mgrb{h}{n}{r}\times \Mgrb{h}{n}{r}\lra \Mgrb{2h+2n-1}{}{r}$$
which sends a pair of curves in $\Mgrb{h}{n}{r}$ to the nodal curve obtained
by gluing together the pairs of marked points with the same labels.
At the level of fundamental groups, this map corresponds to the map induced between
mapping class groups by gluing the surface $S_h^n$ to its mirror image along their boundaries.
We then obtain a map
$H^1_\mathrm{stab}\ra H^1(\Mgrb{h}{n}{r},\ad\Nu_h(\ul))$
via the following composition.
\[\begin{array}{lcl}
  H^1_\mathrm{stab} &\ra    & H^1(\Mgrb{2h+2n-1}{}{r},\ad \Nu_{2h+2n-1}) \\
                    &\ra    & H^1(\Mgrb{h}{n}{r}\times \Mgrb{h}{n}{r}, i^*\ad \Nu_{2h+2n-1}) \\
                    &\ra    & H^1(\Mgrb{h}{n}{r}\times \Mgrb{h}{n}{r}, \bigoplus_{\umu\in\Lambda^n}\ad \Nu_{h}(\umu)\boxtimes \ad \Nu_{h}(\umu)) \\
                    &\ra    & \bigoplus_{\umu} H^1(\Mgrb{h}{n}{r}, \ad \Nu_{h}(\umu))\otimes H^0(\Mgrb{h}{n}{r}, \ad \Nu_{h}(\umu)) \\
                    &\ra    & H^1(\Mgrb{h}{n}{r}, \ad \Nu_{h}(\ul)).
\end{array}\]
Here, the gluing axiom of a modular functor is used to decompose $i^*\ad \Nu_{2h+2n-1}$.
We apply the Künneth formula and, crucially, use the fact that 
$H^0(\Mgrb{h}{n}{r}, \ad \Nu_{h}(\ul))=\C$
(which follows from the irreducibility of $\rho_h(\ul)$).
The class $\psi_h(\ul)$ is then defined as the image of $\psi$ under 
$H^1_\mathrm{stab}\ra H^1(\Mgrb{h}{n}{r},\ad\Nu_h(\ul))$.
When $n=0$, $\psi_h$ is defined as the image of $\psi_h(0)$ under the canonical isomorphism 
$H^1(\Mgrb{h}{1}{r},\ad\Nu_h(0))\simeq H^1(\Mgrb{h}{}{r},\ad\Nu_h)$
(see \Cref{propositionvacuum}).
It is then possible—though somewhat technical—to verify that the classes defined in this way satisfy
the expected compatibility conditions (see \Cref{definitioncompatiblecollection} for a precise statement).
This verification relies essentially on the irreducibility of the $\rho_h(\ul)$.

To lift the classes $\psi_h(\ul)\in H^1(\Mgrb{h}{n}{r},\ad\Nu_h(\ul))$ to genuine deformations
while preserving compatibility, we apply the following general result concerning harmonic representatives
of degree-$1$ classes in (non-Abelian) Hodge theory.
\begin{proposition*}[\ref{theoremuniversalsectionofH1}]
  For any proper smooth Kähler Deligne–Mumford stack $X$ and any semisimple local system $\uE$ on $X$,
  there exists a canonical section $H^1(X,\uE)\ra Z^1_{dR}(X,\uE)$.
  These representatives are compatible with direct sums, tensor products,
  and pullbacks along algebraic morphisms $Y\ra X$.
\end{proposition*}

Thus each $\psi_h(\ul)$ can be lifted to a differential $1$-form in
$Z^1_{dR}(\Mgrb{h}{n}{r},\ad\Nu_h(\ul))$, and hence to a deformation of $\Nu_h(\ul)$
(see \Cref{sublifting}).
The functoriality of the maps $H^1(X,\uE)\ra Z^1_{dR}(X,\uE)$
then implies that the deformations of the $\Nu_h(\ul)$ constructed in this way
fit together to form a deformation of the modular functor $\Nu$.
This completes the construction of the image of $\psi$ under the desired injective map.
We can now conclude as follows: Ocneanu’s rigidity implies that for every $h$ and $\ul$,
$\psi_h(\ul)=0$.
Since, by stability, $\psi=\psi_h$ for sufficiently large $h$, we deduce that $\psi=0$.

Note that homological stability is used here only to compare 
$\mathrm{colim}_h\:H^1(\Mgrb{h}{}{r},\ad\Nu_h)$ and 
$H^1(\Mgrb{h}{}{r},\ad\Nu_h)$ for sufficiently large $h$.
Without appealing to partial homological stability, the argument above already shows that 
$\mathrm{colim}_h\:H^1(\Mgrb{h}{}{r},\ad\Nu_h)=0$.

The issue with the preceding approach is that we have no control over the stability bound in $h$ 
for $H^1(\Mgrb{h}{}{r},\ad\Nu)$.
We now explain how to adapt the proof scheme to show that 
$H^1(\Mgrb{h}{}{r},\ad\Nu)=0$ for all $h\geq 7$.
Fix $g\geq 7$ and let $\psi\in H^1(\Mgrb{g}{}{r},\ad\Nu)$.
The general idea is to repeat a variant of the above construction of the classes $\psi_h(\ul)$
for suitable values of $h$, $n$, and $\ul$ for which the construction makes sense.
Let 
$I\subset\{(h,n,\ul)\mid h,n\ge 0,\; \ul\in\Lambda^n\}$ 
be the set indexing the $\psi_h(\ul)$ that we define, 
and consider the restriction of the modular functor $\Nu$ to $I$,
which forms what we call an \emph{$I$-truncated} modular functor.
To ensure that the $\psi_h(\ul)$ satisfy the desired compatibility conditions,
we are forced to impose certain restrictions on $I$.
In particular, our choice of $I$ requires that if $(h,n,\ul)\in I$, then $h\leq (g+1)/2$,
and it ensures that for $h\leq (g+1)/2$, we have $(h,0,())\in I$.
By lifting the $\psi_h(\ul)$ to their harmonic representatives,
we obtain a deformation of the $I$-truncation of $\Nu$.
A detailed examination of the equivalence between modular categories and modular functors
shows that Ocneanu’s rigidity also applies to $I$-truncated modular functors.
Hence the resulting deformation is trivial, and $\psi_h(\ul)=0$ for all $(h,n,\ul)\in I$.

However, as noted above, $(g,0,())\notin I$, and we cannot directly conclude that $\psi=0$.
\emph{This is precisely where the partial homological stability result becomes crucial:}
it shows that the map $\psi\mapsto \psi_4$ is injective.
The argument then concludes as follows: since $(4,0,())\in I$, we have $\psi_4=0$, 
and by injectivity, it follows that $\psi=0$.

\begin{remark}\label{remark7}
  The lower bound $7$ arises essentially from the technical step in constructing 
  the classes $\psi_h(\ul)$ from $\psi$.
  These classes are defined by pullback along embeddings between surfaces,
  and the condition $h\leq (g+1)/2$ appears in order to guarantee the existence 
  of a canonical choice of embedding.
  Together with the condition $h\geq 4$ coming from partial homological stability,
  this forces us to restrict to $g\geq 7$.
\end{remark}


\subsection{Motivation: rigidity conjectures on mapping class groups}\label{subsectionintroconjectures}


There are some notable analogies
 between these discrete groups
and two other important families of groups:
lattices in Lie groups (for instance, $\mathrm{SL}_n(\Z)$), and the automorphism groups of free groups
$\mathrm{Out}(F_n)$.
For example, they all satisfy the Tits alternative, residual finiteness and forms of homological stability
(see \cite[sections 6.3, 1.3 and 2.2]{vogtmannAutomorphismsFreeGroups2002}, \cite{grossmanResidualFinitenessCertain1974} and \cite[Appendix]{drutuGeometricGroupTheory2018}).

A strong rigidity property preventing almost-invariant vectors,
known as property (T), is known to hold for lattices in Lie groups of rank at least $3$,
by Kazhdan’s original paper \cite{kazhdanConnectionDualSpace1967} on property (T),
and for $\mathrm{Aut}(F_n)$ when $n\ge 5$, by
recent computer-assisted proofs by Kaluba–Nowak–Ozawa \cite{kaluba$textAutmathbbF_5$2019} 
and Kaluba–Kielak–Nowak \cite{kalubaProperty$mathrmAutF_n$2021}.
Unlike in the other cases, property (T) remains open for mapping class groups in genus $g\geq 3$
(it fails for genus $\leq 2$, see \cite[Cor. 2.3]{aramayonaRigidityPhenomenaMapping2016}).
In that context, the question was raised by Ivanov
in \cite{ivanovFifteenProblemsMapping2006}, and
the same paper also discusses a consequence of the conjectured property (T),
known as the \emph{Ivanov conjecture}:
for any surface $S$ of genus $\geq 3$ and any finite-index subgroup $\Gamma\subset \Modu{S}$,
the abelianization $\Gamma^\mathrm{ab}$ is finite.
This conjecture is relatively well-studied,
both directly (see \cite{ershovFinitenessPropertiesJohnson2018})
and through another essentially equivalent conjecture formulated by
Putman and Wieland in \cite{putmanAbelianQuotientsSubgroups2013}
(see \cite{landesmanCanonicalRepresentationsSurface2024,
klukowskiTangleFreePermutations2024,markovicUnramifiedCorrespondencesVirtual2022}).

The property (T) and Ivanov conjectures can be expressed uniformly as follows, for $S=S_g^n$ compact of genus $g$ with $n$ boundary components.
\begin{equation}
 (\mathrm{X})_g^n:\;\forall \rho,\; H^1(\Modu{S_g^n},\rho)=0.
\end{equation}
When $\rho$ ranges over all unitary representations of $\Modu{S_g^n}$ on a Hilbert space, this condition
is property (T), which we denote by $(\mathrm{T})_g^n$. When $\rho$ instead ranges over \emph{finite image} representations of $\Modu{S_g^n}$,
we recover Ivanov's conjecture for $\Modu{S_g^n}$, denoted $(\mathrm{I})_g^n$.

We now introduce two further conjectures of the same flavor.
The first, called the ``finite-dimensional property (T)'' and denoted $(\mathrm{fT})_g^n$,
asserts $(\mathrm{X})_g^n$ for all \emph{finite-dimensional unitary}
representations of $\Modu{S_g^n}$.
The second, denoted $(\mathrm{fD})_g^n$, asserts $(\mathrm{X})_g^n$ for \emph{semisimple} finite-dimensional representations of $\Modu{S_g^n}$
\emph{that send each Dehn twist to an element of finite order}.
We refer to such $\rho$ as finite-twists representations.
This last conjecture is motivated by the fact that it implies $(\mathrm{fT})_g^n$ for $g\ge 3$,
and that its scope includes all adjoint representations of quantum representations arising
from any modular fusion category
(see \cite[Proposition 2.17]{godfardHodgeStructuresConformal2025} and \cite{godfardSemisimplicityConformalBlocks2025a}).

\begin{remark}
  In \cite[Conjecture 6.1.2]{littMotivesMappingClass2024}, Daniel Litt conjectures that for $g\geq 3$ and any \emph{irreducible}
  representation $\rho$ of $\Modu{S_{g,n}}$, $H^1(\Modu{S_{g,n}},\ad\rho)=0$. The restriction to irreducible
  representations is necessary as some non-trivial extensions exist (see \cite[Remark 6.1.4]{littMotivesMappingClass2024}).
  We expect that no such extensions arise when restricting to finite-twists representations, as stated in conjecture $(\mathrm{fD})_g^n$.
  There is some evidence supporting Litt’s conjecture.
  For instance, the typical non-unitary infinite image representation is the symplectic representation
  of $\Modu{S_g}$ on $H^1(S_g,\Z)$, which is known to be rigid for $g\geq 3$ by a result of Looijenga
  \cite{looijengaStableCohomologyMapping1995}. More recently, this result was extended by Zhong to higher Prym representations
  of non-closed surfaces of genus $g\geq 41$ \cite{zhongPrymRepresentationsTwisted2025}.
\end{remark}

The conjectures $(\mathrm{fD})_g^n$, $(\mathrm{fT})_g^n$ and $(\mathrm{I})_g^n$ are each \emph{equivalent to infinitesimal rigidity}
of the representations $\rho$ of the relevant class.
Indeed, infinitesimal rigidity is expressed by the vanishing $H^1(\Modu{S_g^n},\ad\rho)=0$,
where $\ad\rho=\rho^*\otimes\rho$ denotes the adjoint representation, and each of the classes of representations considered is stable under taking adjoints.
Conversely, one obtains $(\mathrm{X})_g^n$ from infinitesimal rigidity using the fact that $\rho$ appears as a direct summand of $\ad(\rho\oplus \mathrm{1})$, where $\mathrm{1}$ denotes the trivial representation.

In genus $g\geq 3$, every finite-dimensional unitary representation
is in fact a finite-twists representation (see the discussion preceding \Cref{conjecturefD}). Thus, for $g\geq 3$,
these conjectures are organized as follows.
\[\begin{tikzcd}
	& {(\mathrm{T})_g^n} \\
	{(\mathrm{fD})_g^n} & {(\mathrm{fT})_g^n} & {(\mathrm{I})_g^n.}
	\arrow[between={0.1}{0.9}, Rightarrow, from=1-2, to=2-2]
	\arrow[between={0.1}{0.9}, Rightarrow, from=2-1, to=2-2]
	\arrow[between={0.1}{0.9}, Rightarrow, from=2-2, to=2-3]
\end{tikzcd}\]

There are also relations linking these conjectures across different values of $g$ and $n$.
We provide proofs of such relations in \Cref{sectionpartialstability},
as a byproduct of one of the key tools used in proving the main result.
These generalize the implications $(\mathrm{I})_{g,n+2}\Rightarrow (\mathrm{I})_{g+1,n}$
and $(\mathrm{I})_{g,n+1}\Rightarrow (\mathrm{I})_{g,n}$ proved by Putman and Wieland
in \cite{putmanAbelianQuotientsSubgroups2013}
and were essentially known to experts, but not written down in this generality to our knowledge
(see \Cref{remarkPWpaper} for more details;
note that these implications are not the main object of Putman and Wieland's article).

\begin{theorem*}[{\ref{corollaryconjecturesreducetothreethree}}]
  Let $(\mathrm{X})$ denote one of $(\mathrm{I})$ (Ivanov conjecture \Cref{conjectureIvanov}),
  $(\mathrm{fT})$ (finite dimensional property (T) conjecture \ref{conjecturefT})
  or $(\mathrm{fD})$ (conjecture \ref{conjecturefD} on finite Dehn twist representations). Then
  \begin{itemize}
    \item $\forall g\geq 3,\forall n\geq 0$, $(\mathrm{X})_{g,n+1}\Rightarrow (\mathrm{X})_{g+1,n}$;
    \item $\forall g\geq 3,\forall n\geq 0$, $(\mathrm{X})_{g,n+1}\Rightarrow (\mathrm{X})_{g,n}$;
    \item $\forall g\geq 3,\forall n\geq 3$, $(\mathrm{X})_{g,n}\Rightarrow (\mathrm{X})_{g,n+1}$.
  \end{itemize}
  In particular, if $(\mathrm{X})_{3,3}$ holds, then $(\mathrm{X})_{g,n}$ holds for all $g\geq 3$, $n\geq 0$.
\end{theorem*}

To prove Ivanov's conjecture and the finite-dimensional property (T) conjecture,
it would thus suffice to have a proof of property (T) for $\Modu{S_3^3}$, for example by computer.


\subsection{Organization of the paper}


In \Cref{sectionMF}, we introduce the equivalent notions of topological and geometric modular functors.
For the geometric definition, we introduce twisted compactifications of moduli spaces of curves.
The last subsection collects properties of $\SU$ and $\SO$ modular functors.
Note that we use a trimmed down definition of modular functor adapted to these cases for smoother exposition.

In \Cref{sectionmainresult}, we state the main result. \Cref{sectionproof} is devoted to its proof.
In \Cref{subbuildingdeformationofmodularfunctor}, we explain the construction of the classes
$\psi_h(\ul)\in H^1(\Mgrb{h}{n}{r},\ad\Nu_h(\ul))$, $(h,n,\ul)\in I$ from a class $\psi\in H^1(\Mgrb{g}{}{r},\ad\Nu_g(\ul))$, $g\geq 7$,
and check that these classes satisfy compatibility conditions akin to those of a modular functor.
These classes are lifted to a deformation of the $I$-truncation of $\Nu$ in \Cref{sublifting}.
To that end, the subsection discusses harmonic representatives in degree $1$ in the context of non-Abelian Hodge theory.
In \Cref{subocneanu}, we apply Ocneanu rigidity to that deformation. \Cref{subpartialhomologicalstability}
contains the partial homological stability result and concludes the proof of \Cref{mainresult}.

The purpose of \Cref{sectionpartialstability} is two-fold: to prove the main Lemma behind the partial homological stability (\Cref{lemmainjectivity}),
and to comment on an application of that Lemma to rigidity conjectures on mapping class groups (\Cref{corollaryconjecturesreducetothreethree}).
The focus is thus not so much on quantum representations and the section can be read independently of the rest of the paper.
The very last subsection revisits rigidity of Fibonacci representations \cite{godfardHodgeStructuresConformal2025} (case $\ell=5$),
giving a swift proof in that case. 


\subsection{Acknowledgements}


The author is especially thankful to Julien Marché for introducing him to the question of rigidity of quantum representations,
and for multiple helpful conversations on this topic.
The author also thanks Ramanujan Santharoubane and Nicolas Tholozan for helpful discussions and for their support.
The author thanks Andrew Putman for useful comments on a draft of this paper.


\section{Modular functors}\label{sectionMF}


Instead of using a construction of the $\SO$ and $\SU$ modular functors, we choose to work axiomatically
in order to highlight the $3$ relevant properties of these modular functors at prime levels that are used in the proof of \Cref{mainresult}.
For better exposition, the definition of modular functor given here is stripped down, see \Cref{remarkinvolution,remarksymmetryofgluing} below,
and we omit some definitions related to the projective anomaly. The anomaly is only necessary to properly define the $\SO$ and $\SU$ modular functors.
When studying deformations of their quantum representations, one works with adjoints and the projective anomaly disappears,
see \Cref{propositionfiniteorderdeformation}.
The reader may thus safely ignore the anomaly.

The proof of \Cref{mainresult} relies heavily on the geometric definition of modular functors \emph{via twisted moduli spaces of curves},
as the properness of these Deligne-Mumford stacks allows the use of (pure) Hodge theory. This will be crucial at $2$ distinct
points in the proof, see \Cref{sublifting,sectionpartialstability}.

A classic reference for modular functors is \cite[chp. 5 and 6]{bakalovLecturesTensorCategories2000}.
The exposition in \Cref{subtopologicalMF,sublevels,subgeometricMF} follows \cite[§2]{godfardHodgeStructuresConformal2025},
while that of \Cref{subSOandSU} follows \cite[§2]{godfardRigidityFibonacciRepresentations2025}.


\subsection{Topological definition of modular functor}\label{subtopologicalMF}


\begin{definition}\label{definitioncolourset}
  A set of colors is a finite set $\Lambda$ with a preferred element $0\in \Lambda$.
\end{definition}

\begin{remark}\label{remarkinvolution}
  For general modular functors, $\Lambda$ is also equipped with an involution.
  Here, as we will work with $\SO$ and $\SU$ modular functors, the involution is trivial and hence suppressed from the definitions.
\end{remark}

\begin{definition}\label{definitionsurfaces}
    Let $g,n\geq 0$. We define $S_g^n$ to be the compact surface of genus $g$ with $n$ boundary components,
    and $S_{g,n}$ to be the surface of genus $g$ with $n$ punctures and empty boundary.
\end{definition}

We now define the source category of (non-anomalous) modular functors.

\begin{definition}\label{definitioncolouredcategory}
    Let $\Lambda$ be a set of colors. The groupoid $\Surf$ of surfaces colored with $\Lambda$ is such that
    \begin{description}
        \item[(1)] its objects are compact oriented surfaces $S$ together with, for each component $B$ of $\partial S$,
        the choice of an orientation preserving identification $\varphi_B: B\simeq S^1$ and of a color $\lambda_B\in \Lambda$ ;
        \item[(2)] its morphisms from $\Sigma_1=(S_1,\varphi^1,\underline{\lambda}^1)$
        to $\Sigma_2=(S_2,\varphi^2,\underline{\lambda}^2)$
        are isotopy classes of homeomorphisms $f:S_1\ra S_2$ preserving orientation such that for each component $B_1\subset\partial S_1$
        and its image $f(B_1)=B_2\subset \partial S_2$, we have $\lambda_{B_1}=\lambda_{B_2}$ and $\varphi^2_{B_2}\circ f=\varphi^1_{B_1}$.
        \item[(3)] the composition is the composition of homeomorphisms.
    \end{description}
    This groupoid has a natural monoidal structure induced by the disjoint union $\sqcup$.
\end{definition}

To define anomalous modular functors, one needs to work with an extension $\tSurf$ of the groupoid $\Surf$.
Its objects $(S,L,\varphi,\underline{\lambda})$ are further equipped with a split Lagrangian $L$ in the first homology group
$H_1(\widehat{S};\Q)$ of the closure $\widehat{S}$ of $S$.
The map $\tSurf\ra\Surf$ sends $(S,L,\varphi,\underline{\lambda})$ to $(S,\varphi,\underline{\lambda})$.
We refer to \cite[§2.5]{godfardHodgeStructuresConformal2025} for the definition of on morphisms in $\tSurf$ and of the composition.

This extension is central in the following sense. For $(S,L,\varphi,\underline{\lambda})$ in $\tSurf$ with $S$ connected,
the automorphism group of $(S,\varphi,\underline{\lambda})$ in $\Surf$ is a mapping class group\footnote{Boundary components with same color may be permuted.},
while that of $(S,L,\varphi,\underline{\lambda})$ in $\tSurf$ is a central extension of that mapping class group by
$\Z$.

\begin{notation}
  We will often shorten $(S,L,\varphi,\underline{\lambda})$ or $(S,\varphi,\underline{\lambda})$ to $(S,\ul)$.
  Also, we will use the notation "$\ul,\mu$" for $\ul\in\Lambda^n$ and $\mu\in\Lambda$ as shorthand for "$\lambda_1,\dotsc,\lambda_n,\mu$".
\end{notation}

We now describe the additional gluing structure on $\Surf$.
Let $S$ be a compact surface and $\partial_+S\sqcup\partial_-S\subset\partial S$ be two components of its boundary.
Let $\varphi_{\partial_{\pm}S}:\partial_{\pm}S\simeq S^1$ be identifications of these components with $S^1$.

Define $S_{\pm}$ to be the surface obtained from $S$ by gluing $\partial_+S$ to $\partial_-S$
along $\varphi_{\partial_-S}^{-1}\circ(z\mapsto z^{-1})\circ\varphi_{\partial_+S}$.
Then $S_{\pm}$ is called the gluing of $S$ along $\partial_{\pm}S$. Notice that if we exchange the roles of $\partial_+S$ and $\partial_-S$
in the gluing process we get a surface $S_{\mp}$ that is canonically isomorphic to $S_{\pm}$.
See \Cref{figuregluing}.

This gluing structure can be extended to $\tSurf$, i.e. one can canonically obtain a Lagrangian $L_\pm$ on $S_{\pm}$ from a Lagrangian $L$ on $S$,
see \cite[§2.5]{godfardHodgeStructuresConformal2025}.

\begin{figure}
  \def\svgwidth{0.6\linewidth}
\begingroup%
  \makeatletter%
  \providecommand\color[2][]{%
    \errmessage{(Inkscape) Color is used for the text in Inkscape, but the package 'color.sty' is not loaded}%
    \renewcommand\color[2][]{}%
  }%
  \providecommand\transparent[1]{%
    \errmessage{(Inkscape) Transparency is used (non-zero) for the text in Inkscape, but the package 'transparent.sty' is not loaded}%
    \renewcommand\transparent[1]{}%
  }%
  \providecommand\rotatebox[2]{#2}%
  \newcommand*\fsize{\dimexpr\f@size pt\relax}%
  \newcommand*\lineheight[1]{\fontsize{\fsize}{#1\fsize}\selectfont}%
  \ifx\svgwidth\undefined%
    \setlength{\unitlength}{253.71787244bp}%
    \ifx\svgscale\undefined%
      \relax%
    \else%
      \setlength{\unitlength}{\unitlength * \real{\svgscale}}%
    \fi%
  \else%
    \setlength{\unitlength}{\svgwidth}%
  \fi%
  \global\let\svgwidth\undefined%
  \global\let\svgscale\undefined%
  \makeatother%
  \begin{picture}(1,0.78063141)%
    \lineheight{1}%
    \setlength\tabcolsep{0pt}%
    \put(0,0){\includegraphics[width=\unitlength,page=1]{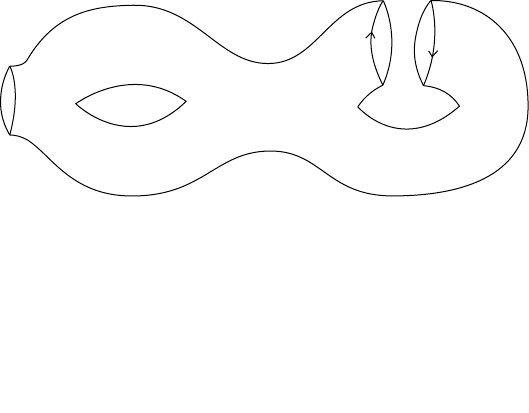}}%
    \put(0.83310804,0.68330276){\color[rgb]{0,0,0}\makebox(0,0)[lt]{\lineheight{1.25}\smash{\begin{tabular}[t]{l}$\partial_-S$\end{tabular}}}}%
    \put(0.59633921,0.67482277){\color[rgb]{0,0,0}\makebox(0,0)[lt]{\lineheight{1.25}\smash{\begin{tabular}[t]{l}$\partial_+S$\end{tabular}}}}%
    \put(0.23714832,0.66825613){\color[rgb]{0,0,0}\makebox(0,0)[lt]{\lineheight{1.25}\smash{\begin{tabular}[t]{l}$S$\end{tabular}}}}%
    \put(0.23714832,0.2549931){\color[rgb]{0,0,0}\makebox(0,0)[lt]{\lineheight{1.25}\smash{\begin{tabular}[t]{l}$S_\pm$\end{tabular}}}}%
    \put(0,0){\includegraphics[width=\unitlength,page=2]{gluing.pdf}}%
  \end{picture}%
\endgroup%

  \caption{Gluing $S_{\pm}$ of $S$ along $\partial_{\pm}S$. The $\bullet$ represents the image of $1$ in the boundary parametrization.}
  \label[figure]{figuregluing}
\end{figure}

\begin{definition}[Modular Functor]\label{definitiontopologicalmodularfunctor}
    Let $\Lambda$ be a set of colors.
    Then a modular functor is the data of a monoidal functor
    $$\Nu:\tSurf\lra \text{finite dimensional }\C-\text{vector spaces}$$
    where the monoidal structure on the target is understood to be the tensor product.
    This data is augmented by the following isomorphisms.
    \begin{description}
        \item[(G)] For any surface $S$, split Lagrangian $L$ in $H_1(\widehat{S};\Q)$, and pair of boundary components $\partial_{\pm}S$ lying on distinct connected components of $S$,
        let $S_{\pm}$ be the gluing of $S$ along $\partial_{\pm}S$. For any coloring $\underline{\lambda}$
        of the components of $\partial S_{\pm}$, an isomorphism as below is given.
        \begin{equation}\label{equationgluing}
            \Nu(S_{\pm},L_\pm,\underline{\lambda})\simeq \bigoplus_{\mu\in\Lambda}\Nu(S,L,\mu,\mu,\underline{\lambda}).
        \end{equation}
    \end{description}
    This rule, also sometimes called fusion or factorization rule, is the most important property of modular functors.
    The isomorphisms of \textbf{(G)} are assumed to be functorial and compatible with disjoint unions.
    Gluing isomorphisms along distinct pairs of boundary components must commute with each other. Moreover, we ask for the gluing to be symmetric in the sense
    that the gluing isomorphisms are unchanged when exchanging $\partial_-S$ and $\partial_+S$ and identifying $S_\pm$ with $S_\mp$.

    The functor is also assumed to verify $2$ more axioms, called normalization and non-degeneracy.
    \begin{description}
        \item[(N)] $\Nu(S_0^1,\lambda)$ is canonically isomorphic to $\C$ if $\lambda=0$ and is $0$ otherwise;
        \item[(nonD)] For each $\lambda$, $\Nu(S_0^2,\lambda,\lambda)\neq 0$.
    \end{description}

    A modular functor is said to be non-anomalous if $\Nu$ factorizes through $\Surf$.
\end{definition}

\begin{remark}
  Given a modular functor $\Nu$, for any colored connected surface $\Sigma=(S,\varphi,\underline{\lambda})$,
  one gets a projective representations $\rho_\Sigma$ of the mapping class group $\Modu{S,\partial S}$.
  For a non-anomalous modular functor, this representation is canonically linearized. 
  In particular, for any $g\geq 0$, by setting $\Sigma=(S_g,\varnothing,\varnothing)$,
  one gets, up to conjugacy, a projective representation $\rho_g$ of $\Modu{S_g}$.
  These are the representations of interest in this article.
\end{remark}

\begin{remark}\label{remarksymmetryofgluing}
  In the more general definition of modular functor, gluing and symmetry of gluing are more subtle, see, for example, \cite[2.5]{godfardHodgeStructuresConformal2025}.
\end{remark}

\begin{remark}
  By applying the axiom \textbf{(G)} to the gluing of one boundary component of $S_0^2$ to a boundary component of another copy of $S_0^2$, one gets the following.
  \begin{equation*}
    \Nu(S_0^2,\lambda,\mu) = \begin{cases}
      \C \text{ canonically } & \text{if }\lambda=\mu\text{,} \\
      0 & \text{otherwise.}
    \end{cases}
  \end{equation*}
\end{remark}

\begin{remark}\label{remarkvacuum}
  By the normalization axiom, for $S$ obtained from $S'$ by removing the interior of a disk we have a canonical isomorphism
  $\Nu(S,0)\simeq \Nu(S')$ called a vacuum isomorphism.
\end{remark}

\begin{remark}\label{remarklevel}
  Let $S_{\pm}$ be a colored surface constructed as a gluing of a surface $S$ along $\partial_{\pm}S$.
  Let $\gamma$ denote the simple closed curve that is the image of $\partial_{\pm}S$ in $S_{\pm}$.
  Then the Dehn twist $T_\gamma$ acts block-diagonally on the decomposition \textbf{(G)}.

  Moreover, one can easily see that its action on the block $\Nu(S,\mu,\mu,\underline{\lambda})$
  is by a scalar $t_\mu\in \C^\times$ that depends only on $\mu$ and not on the surface $S$.
\end{remark}

\begin{remark}\label{remarkcentralcharge}
  Let $\Nu$ be a modular functor.
  For any connected colored surface $\Sigma=(S,\varphi,\ul)$ and object $\tilde{\Sigma}$ of $\tSurf$ extending it,
  the generator of $\Z$ in the canonical central extension $1\ra \Z\ra \Aut{\tilde{\Sigma}}\ra\Aut{\Sigma}\ra 1$ acts by
  $c\in\C^\times$ on $\Nu(\tilde{\Sigma})$.
  This scalar $c$ can be shown using the gluing axiom to be independent of $\tilde{\Sigma}$ and $\Sigma$, and is called the central charge of $\Nu$.
  A modular functor is non-anomaly if and only if its central charge is $1$.
\end{remark}


\subsection{Levels for modular functors}\label{sublevels}


It will be crucial in this work that the source groupoid of the modular functors be fundamental groupoids
of proper DM stacks. While the connected components of $\tSurf$ are fundamental groupoids of DM stacks, these stacks are not proper.
Hence we will be using a general result about modular functors which says that a modular functor $\Nu$
always factors through some quotient of $\tSurf$ whose connected components are fundamental groupoids of proper DM stacks
(twisted moduli spaces of curves, see \Cref{subtwistedmodulispaces} below).

\begin{definition}
  For any $\Sigma=(S,\varphi,\ul)$ in $\Surf$ and $\gamma\subset S$, denote by $T_\gamma$ the Dehn left twist along $\gamma$
  seen as an element of $\Aut{\Sigma}$.
  For $r\geq 1$, denote by $\Surf(r)$ the quotient of $\Surf$ by $T_\gamma^r$, for all $\gamma$ as above.
  
  For any object $\tilde{\Sigma}=(S,L,\varphi,\ul)$ in $\tSurf$ extending $\Sigma$ in $\Surf$,
  and simple closed curve $\gamma\subset S$ with homology class $[\gamma]$ in $L$, the Dehn twist $T_\gamma$ in $\Aut{\Sigma}$
  has a preferred lift $(T_\gamma, \id_L)$ in $\Aut{\tilde{\Sigma}}$, that we will also denote $T_\gamma$ and call a pure Dehn twist.\footnote{See
  \cite[§2.5]{godfardHodgeStructuresConformal2025} for the definition of morphisms in $\tSurf$.}
  For $r,s\geq 1$, denote by $\tSurf(r,s)$ the quotient of the groupoid $\tSurf$ by $T_\gamma^r$
  for all $T_\gamma$ pure Dehn twist, and by $\kappa^s$ for all $\kappa$ automorphism in $\tSurf$ in the kernel of $\tSurf\ra\Surf$.
\end{definition}

The following is essentially the Andersen-Moore-Vafa theorem. We refer to \cite[§2.4]{godfardHodgeStructuresConformal2025} for a proof and references.

\begin{theorem}\label{theoremlevel}
  For any modular functor $\Nu$, there exists $r,s\geq 1$ such that $\Nu$ factors through $\tSurf(r,s)$.
\end{theorem}

\begin{definition}\label{definitionlevel}
  A level for a modular functor $\Nu$ is a choice of a pair $(r,s)$ as in \Cref{theoremlevel}.
\end{definition}


\subsection{Twisted moduli spaces of curves}\label{subtwistedmodulispaces}


Two of the main arguments in this paper use Hodge theory and rely on the equivalent geometric definition of modular functor
via twisted moduli spaces of curves.
In the Bakalov-Kirillov geometric definition \cite[6.4.1]{bakalovLecturesTensorCategories2000}, modular functors are defined as collections of local systems
on some bundle over moduli stacks of smooth curves curve $\Mg{g}{n}$.
These stacks are not proper, hence will not work for us as we will use pure Hodge theory.
A natural idea is then to use the Deligne-Mumford compactifications $\oMg{g}{n}$ of these moduli spaces.
While the local systems considered do not extend to these compactifications in general, they can be extended to some alternative twisted versions
$\Mgrbt{g}{n}{r}{s}$ of the Deligne-Mumford compactifications.
This section and the next one are a condensed version of \cite[§2.4]{godfardHodgeStructuresConformal2025},
following \cite[§2.1, §2.2]{godfardSemisimplicityConformalBlocks2025a}.
See the former for more details.

\begin{notation}
  For $g,n\geq 0$ such that $2g-2+n>0$, we will denote by $\Mg{g}{n}$ the moduli stack of smooth connected curves of genus $g$ with $n$ distinct marked 
  points over $\C$, and by $\oMg{g}{n}$ the Deligne-Mumford compactification classifying stable nodal connected curves with $n$ distinct marked points
  on their smooth locus.
\end{notation}

We will work with variations of the Deligne-Mumford compactifications.
These depend on a integer $r\geq 1$ and classify $r$-twisted curves in the sense of Kontsevich, Abramovic and Vistoli
\cite{abramovichCompactifyingSpaceStable2002}. For a reference on their definition, see Chiodo's article \cite[1.3]{chiodoStableTwistedCurves2008}.

\begin{notation}
  For $r\geq 1$ and $g,n\geq 0$ such that $2g-2+n>0$, denote by $\Mgrp{g}{n}{r}$
  the moduli space of stable nodal $r$-twisted connected curves of genus $g$ with $n$ distinct smooth marked points \cite[Th. 4.4]{chiodoStableTwistedCurves2008}.
\end{notation}

Note that $\Mgrp{g}{n}{1}=\oMg{g}{n}$.

\begin{remark}\label{remarkrootstack}
  One can give an alternative description of these stacks using the root stack construction.
  Let $D_i\subset \oMg{g}{n}$ for $i=1,\dotsc,k$ be the components
  of the boundary divisor. Then $\Mgrp{g}{n}{r}$ is obtained from $\oMg{g}{n}$ by taking $r$-th root stack independently locally
  on each $D_i$ (see \cite[2.3, 4.5]{chiodoStableTwistedCurves2008})
  \begin{equation*}
    \Mgrp{g}{n}{r}=\oMg{g}{n}\left[\sum_i \frac{D_i}{r}\right].
  \end{equation*}
  With this description, applying the Van-Kampen theorem, one can see that the fundamental group of $\Mgrp{g}{n}{r}$ is equivalent
  to the quotient $\PMod{S_{g,n}}/\langle T_\delta^r\mid \delta\rangle$ of the pure mapping class group of the $n$ times \emph{punctured}
  genus $g$ surface by all $r$-th powers of Dehn twists.
\end{remark}

\begin{notation}
  For $r\geq 1$ and $g,n\geq 0$ such that $2g-2+n>0$, denote by $\Mgrb{g}{n}{r}$
  the moduli space of stable nodal $r$-twisted curves of genus $g$ with $n$ distinct smooth order $r$ stacky points
  \emph{and a section at each stacky point}.
  We will use the convention $\Mgrb{0}{2}{r}=\Br{r}$, where $\Br{r}$ is a specific family of curves, see \cite[Rk. 2.26]{godfardHodgeStructuresConformal2025}.
\end{notation}

\begin{remark}\label{remarkboundary}
  The stack $\Mgrb{g}{n}{r}$ is a $\mu_r^n$-gerbe over $\Mgrp{g}{n}{r}$, which can be described as follows.
  Let $\Sigma_i\subset \Cgrp{g}{n}{r}$ be the $i$-th stacky marked point in the universal curve $\Cgrp{g}{n}{r}\ra \Mgrp{g}{n}{r}$.
  Then $\Sigma_i$ is a $\mu_r$-gerbe over $\Mgrp{g}{n}{r}$ and $\Mgrb{g}{n}{r}$ is the product over $\Mgrp{g}{n}{r}$ of the $\Sigma_i$.
  With this description, we see that the fundamental group of $\Mgrb{g}{n}{r}$ is isomorphic
  to the quotient $\Modu{S_g^n}/\langle T_\delta^r\mid \delta\rangle$ of the mapping class group of the
  genus $g$ surface with $n$ \emph{boundary components} by all $r$-th powers of Dehn twists.
\end{remark}

To deal with projective anomalies of modular functors,
we need yet another variation: the stacks $\Mgrbt{g}{n}{r}{s}$. These spaces depend also on another integer $s\geq 1$, and
$\Mgrbt{g}{n}{r}{s}$ is a $\mu_s$-gerbe over $\Mgrb{g}{n}{r}$. At the level of fundamental groups, they correspond
to the stable central extension of mapping class groups by $\Z$, quotiented by $s\Z$
and all $r$-th powers of pure Dehn twists.
See \cite[§2.5]{godfardHodgeStructuresConformal2025} for definitions and details on this.

Gluing and forgetful maps between the Deligne-Mumford compactifications $\oMg{g}{n}$ have analogs for the $\Mgrb{g}{n}{r}$
and the $\Mgrbt{g}{n}{r}{s}$. See \Cref{definitiongeometricmodularfunctor} below. These can either be defined from the interpretation as
moduli spaces of twisted curves, or directly from those on the $\oMg{g}{n}$ using the root stack and gerbe construction
mentioned in \Cref{remarkrootstack,remarkboundary}.


\subsection{Geometric definition of modular functors}\label{subgeometricMF}


We reproduce here the definition of geometric modular functor in \cite[§2.4]{godfardHodgeStructuresConformal2025}.
This definition is equivalent to that given by Bakalov-Kirillov in \cite[6.4.1, 6.7.6]{bakalovLecturesTensorCategories2000}.

\begin{definition}[Modular Functor]\label{definitiongeometricmodularfunctor}
  Let $\Lambda$ be a set of colors.
  Let $r,s\geq 1$ be integers.
  Then a geometric modular functor with level $(r,s)$ is the data,
  for each $g,n\geq 0$, $(g,n)\neq (0,0),(0,1),(1,0)$, and $\ul\in\Lambda^n$,
  of a local system $\Nu_g(\ul)$ over $\Mgrbt{g}{n}{r}{s}$, together with some isomorphisms
  described below.
  \begin{description}
      \item[(G-sep)] For each gluing map
      \begin{equation*}
        q:\Mgrbt{g_1}{n_1+1}{r}{s}\times\Mgrbt{g_2}{n_2+1}{r}{s}\lra \Mgrbt{g_1+g_2}{n_1+n_2}{r}{s}
      \end{equation*}
      and each $\ul$, an isomorphism
      \begin{equation*}
        q^*\Nu_{g_1+g_2}(\lambda_1,\dotsc,\lambda_n)\simeq \bigoplus_\mu\Nu_{g_1}(\lambda_1,\dotsc,\lambda_{n_1},\mu)\otimes
        \Nu_{g_2}(\lambda_{n_1+1},\dotsc,\lambda_n,\mu);
      \end{equation*}
      \item[(G-nonsep)] For each gluing map
      \begin{equation*}
        p:\Mgrbt{g-1}{n+2}{r}{s}\lra \Mgrbt{g}{n}{r}{s}
      \end{equation*}
      and each $\ul$, an isomorphism
      \begin{equation*}
        p^*\Nu_g(\lambda_1,\dotsc,\lambda_n)\simeq \bigoplus_\mu\Nu_{g-1}(\lambda_1,\dotsc,\lambda_n,\mu,\mu);
      \end{equation*}
      \item[(N)] For each forgetful map $f:\Mgrbt{g}{n+1}{r}{s}\ra\Mgrbt{g}{n}{r}{s}$, and each $\ul$, an isomorphism
      \begin{equation*}
        f^*\Nu_g(\lambda_1,\dotsc,\lambda_n)\simeq \Nu_g(\lambda_1,\dotsc,\lambda_n,0)
      \end{equation*}
      and a canonical isomorphism $\Nu_0(0,0)\simeq \underline{\C}$ (trivial local system);
      \item[(Perm)] For each $\ul\in\Lambda^n$ and permutation $\sigma\in S_n$, an isomorphism
      \begin{equation*}
        \Nu_g(\lambda_1,\dotsc,\lambda_n)\simeq \sigma^*\Nu_g(\lambda_{\sigma(1)},\dotsc,\lambda_{\sigma(n)}).
      \end{equation*}
  \end{description}
  The isomorphisms of \textbf{(G-sep)}, \textbf{(G-nonsep)}, \textbf{(N)} and \textbf{(Perm)}
  are to be compatible with each other and repeated applications.
  
  The functor is also assumed to verify the non-degeneracy axiom
  \begin{description}
      \item[(nonD)] For each $\lambda$, $\Nu_0(\lambda,\lambda)\neq 0$.
  \end{description}
\end{definition}

\begin{remark}
  Despite our convention $\Mgrbt{0}{2}{r}{s}=B\mu_r\times B\mu_s$, because we interpret it as a family of twisted curves (see \cite[Rk. 2.26]{godfardHodgeStructuresConformal2025}),
  we can make sense of gluing maps involving it.
\end{remark}


\subsection{Comments on the relation between the \texorpdfstring{$\mathbf{2}$}{2} definitions of modular functor}

The equivalence of the $2$ definitions is essentially proved in \cite[6.4.2]{bakalovLecturesTensorCategories2000},
with the difference that Bakalov and Kirillov do not use twisted compactifications.
Here, we take the time to give the general idea of how this equivalence works, as we will need to juggle between the $2$ definitions.

We first explain the non-anomalous case.
Let $\Nu:\Surf\ra \C-\text{vector spaces}$ be a modular functor in the topological definition. Then the monoidality of $\Nu$ provides,
for any colored surfaces $\Sigma_1$, $\Sigma_2$, a functorial isomorphism $\Nu(\Sigma_1\sqcup\Sigma_2)\simeq \Nu(\Sigma_1)\otimes\Nu(\Sigma_2)$.
Hence a modular functor is essentially determined by its restriction to the full sub-groupoid $\Surf^\mathrm{conn}\subset \Surf$ of connected surfaces.

Notice that if one rewrites \Cref{definitiontopologicalmodularfunctor} replacing $\Surf$ by $\Surf^\mathrm{conn}$, then some modifications to
the gluing axiom is necessary. Indeed, the gluing axiom for $S_\pm$ connected presents $2$ cases: either $S$ is connected, or $S$ has $2$ connected components,
say $S_+$ containing $\partial_+S$, and $S_-$ containing $\partial_-S$. In the former case, the axiom still takes the form of \Cref{equationgluing},
but in the latter, we have to fill in for the lack of monoidality and the axiom takes the form of an isomorphism
\begin{equation}\label{equationgluingdisconnected}
            \Nu(S_{\pm},\underline{\lambda}_+,\underline{\lambda}_-)\simeq \bigoplus_{\mu\in\Lambda}\Nu(S_+,\mu,\underline{\lambda}_+)\otimes\Nu(S_-,\mu,\underline{\lambda}_-)
\end{equation}
where $\ul_+$ (respectively $\ul_-$) are the colorings of the boundary components of $S_{\pm}$ that are also in $S_+$ (respectively $S_-$).

Let $\mathrm{Surf}_{g}^{n}$ be the groupoid whose objects are compact genus $g$ surfaces with $n$ ordered parametrized boundary components,
and whose morphisms are isotopy classes of diffeomorphisms preserving the ordering and parametrizations of the boundary components.
Notice that $\Surf^\mathrm{conn}$ is equivalent to $\bigsqcup_{g,n}\left(\mathrm{Surf}_{g}^{n}\times\Lambda^n\right)/S_n$.

We can again simplify the definition of $\Nu$ by further restricting the underlying functor from $\Surf^\mathrm{conn}$
to $\bigsqcup_{g,n}\mathrm{Surf}_{g}^{n}\times\Lambda^n$.
The price to pay is then an additional permutation axiom which gives canonical isomorphisms
$$\sigma_*:\Nu(S,\lambda_1,\dotsc,\lambda_n)\simeq \Nu(\sigma(S),\lambda_{\sigma(1)},\dotsc,\lambda_{\sigma(n)})$$
for $S$ a connected surface with $n$ ordered boundary components (each parametrized), and $\sigma$ a permutation.
Above, $\sigma(S)$ denotes $S$ with its boundary components reordered according to $\sigma$.

Finally, we can restrict $\Nu$ to $\bigsqcup_{(g,n)\neq (0,0), (0,1), (1,0)}\mathrm{Surf}_{g}^{n}\times\Lambda^n$.
And again, the price to pay is some changes to the axioms: the normalization axiom \textbf{(N)} can no longer make sense as is
because $g=0$, $n=1$ is not allowed, and is replaced by the statement in \Cref{remarkvacuum}. With this formulation, 
it is usually called the vacuum axiom.

We now explain why this modified (equivalent) definition of topological modular functor is equivalent
to the definition of geometric modular functor.
The main point is that the fundamental groupoid of $\Mgrb{g}{n}{r}$ is equivalent to
the quotient $\mathrm{Surf}_{g}^{n}(r)$ of $\mathrm{Surf}_{g}^{n}$ by $r$-th powers of Dehn twists, see \Cref{remarkboundary}.
Hence, a functor from $\mathrm{Surf}_{g}^{n}(r)\times\Lambda^n$ to finite dimensional $\C$-vector spaces
is the same as a collection of $\C$-local systems on $\Mgrb{g}{n}{r}$ parametrized by $\Lambda^n$.
The axioms listed above for a modular functor with source $\bigsqcup_{(g,n)\neq (0,0), (0,1), (1,0)}\mathrm{Surf}_{g}^{n}\times\Lambda^n$
then translate to those of \Cref{definitiongeometricmodularfunctor}.

In the anomalous case, the above derivations work as is when adding Lagrangians, and noting that the fundamental groupoid of $\Mgrbt{g}{n}{r}{s}$
is equivalent to the quotient $\mathrm{S\tilde{u}rf}_{g}^{n}(r,s)$ of the extension $\mathrm{S\tilde{u}rf}_{g}^{n}$ of $\mathrm{Surf}_{g}^{n}$
by $r$-th powers of pure Dehn twists
and the $s$-th powers of automorphisms in the kernel of $\mathrm{S\tilde{u}rf}_{g}^{n}\ra\mathrm{Surf}_{g}^{n}$.


\subsection{\texorpdfstring{$\mathbf{SO(3)}$ and $\mathbf{SU(2)}$}{SO(3) and SU(2)} modular functors at prime levels}\label{subSOandSU}


For $\ell\geq 3$ be an integer, define $\Lambda_\ell:=\{0,1,\dotsc,\ell-2\}$, and when $\ell$ is odd, define $\Lambda_\ell^\mathrm{ev}:=\{0,2,\dotsc,\ell-3\}$.
For $\ell\geq 3$, the $\SU$ modular functor of level $2\ell$ is a modular functor with set of colors $\Lambda_\ell$, which we will denote $\Nu_\ell$.
For $\ell\geq 3$ odd, the $\SO$ modular functor of level $\ell$ is a modular functor with set of colors $\Lambda_\ell^\mathrm{ev}$,
which we will denote $\Nu_\ell^\mathrm{ev}$.
These modular functors are anomalous.

\begin{remark}
  The terminology "level" in "$\SU$ modular functor of level $2\ell$" is related to but different from our definition of "a level" for a modular functor
  (\Cref{definitionlevel}). The relation is that $(2\ell,8\ell)$ is a level for $\Nu_\ell$,
  and, for $\ell$ odd, $(\ell,4\ell)$ is a level for $\Nu_\ell^\mathrm{ev}$. Note that this notion of level does not
  match the usual one in conformal field theory: the $\SU$ modular functor at level $2\ell$
  is related to the $\mathfrak{sl}_2$ modular functor at conformal level $\ell-2$.
\end{remark}

We refer to the work of C. Blanchet, N. Habegger, G. Masbaum, and P. Vogel \cite{blanchetTopologicalQuantumField1995}
for a construction of the $\SU$ and $\SO$ modular functors. Note that in their paper, the modular functors are defined
over cyclotomic fields. Here, we choose any embedding of the cyclotomic field into $\C$.
This choice is irrelevant to the question of rigidity: proving rigidity for one is proving it for all of them, as it is a Galois invariant property.
Only the following properties of these functors will be used in this paper.

\begin{proposition}\label{propositionSOmodularfunctor}
    Let $\ell\geq 3$. Then the modular functor $\Nu=\Nu_\ell$ satisfies
    \begin{description}
      \item[(I)] $\{\lambda\mid \Nu(S_1^1,\lambda)\neq 0\}=\{\lambda\mid \Nu(S_2^1,\lambda)\neq 0\}$.
    \end{description}
    for $\Lambda=\Lambda_\ell$. The same properties hold, when $\ell$ is odd, for $\Nu=\Nu_\ell^\mathrm{ev}$
    and $\Lambda=\Lambda_\ell^\mathrm{ev}$.\footnote{In fact, in both cases, the set in \textbf{(I)} is that of even numbers in $\Lambda$. For
    the $\SO$ case, this is all of $\Lambda$.}
    Moreover, if $\ell$ is an (odd) prime number, both $\Nu_\ell$ and $\Nu_\ell^\mathrm{ev}$ satisfy
    $$\text{\emph{\textbf{(II)}} For every }\lambda,\mu\in\Lambda\text{, if }\lambda\neq\mu\text{, then }t_\lambda\neq t_\mu.$$
\end{proposition}

\begin{proof}
    Property \textbf{(I)} is deduced from \Cref{colourconditions} below and repeated use of the gluing axiom, see \cite{blanchetTopologicalQuantumField1995}.
    Property \textbf{(II)} is easily deduced from the fact that $t_\lambda=(-1)^\lambda\zeta_\ell^{\lambda(\lambda+2)}$, where $\zeta_\ell$
    is a primitive $\ell$-th root of unity in the $\SO$ case and a primitive $2\ell$-th root of unity in the $\SU$ case.
    See \cite[section 3]{blanchetThreemanifoldInvariantsDerived1992} or \cite[Lemma 2.5]{marcheIntroductionQuantumRepresentations2021}
    for the proof of this formula.
\end{proof}

\begin{proposition}\label{colourconditions}
    Let $a,b,c\in\Lambda$ for $\Lambda=\Lambda_\ell$ (respectively $\Lambda=\Lambda_\ell^\mathrm{ev}$).
    Then $\Nu(S_0^3,a,b,c)\neq 0$, for $\Nu=\Nu_\ell$ (respectively $\Nu=\Nu_\ell^\mathrm{even}$), if and only if
    \begin{itemize}
        \item $a+b+c$ is even;
        \item $a,b,c$ verify triangular inequalities, ie. $|a-b|\leq c \leq a+b$;
        \item $a+b+c<2\ell-2$.
    \end{itemize}
    Moreover, in these cases, $\Nu(S_0^3,a,b,c)$ is of dimension $1$.
\end{proposition}

The most crucial property of $\Nu$ we will need for \Cref{mainresult} is irreducibility, which is true for $\Nu=\Nu_\ell$ and $\Nu=\Nu_\ell^\mathrm{ev}$
when $\ell$ is prime but not in general
(see \cite[Th. 1.6.(ii), (iii)]{blanchetTopologicalQuantumField1995}).
The following Proposition follows from a result of \cite{koberdaIrreducibilityQuantumRepresentations2018},
where they apply the method of \cite{robertsIrreducibilityQuantumRepresentations2001} to a cleverly chosen special case of a Proposition of
\cite{blanchetTopologicalQuantumField1995}. See \cite[App. A]{godfardRigidityFibonacciRepresentations2025} for details and precise citations. 

\begin{proposition}\label{propositionirreducibility}
    Let $\ell\geq 3$ be a prime number. Then for any colored surface $(S,\underline{\lambda})$,
    the representation of $\tildeMod{S}$ on $\Nu_\ell(S,\underline{\lambda})$ is irreducible. The same result holds for $\Nu_\ell^\mathrm{ev}$.
\end{proposition}


\section{Main result}\label{sectionmainresult}


\begin{theorem}\label{mainresult}
  Let $\ell\geq 3$ be a prime number. Denote by $\Nu$ the $\SU$ modular functor $\Nu_\ell$ of level $2\ell$,
  or the $\SO$ modular functor $\Nu_\ell^\mathrm{ev}$ of level $\ell$. Then for any $g\geq 7$,
  the representation of $\Modu{S_g}$ on $\Nu(S_g)$ is (cohomologically) rigid, i.e. $H^1(\Modu{S_g},\ad \Nu(S_g))=0$.
\end{theorem}

\begin{remark}\label{remarkassumptions}
  In fact, we prove that any modular functor satisfying properties \textbf{(I)}, \textbf{(II)} of \Cref{propositionSOmodularfunctor}
  and irreducibility (\Cref{propositionirreducibility}) satisfies the conclusion of the theorem.
\end{remark}

Note that in this context, cohomological rigidity never depends on the target group of the representation
(see \cite[Prop. 4.2]{godfardRigidityFibonacciRepresentations2025}),
or, when $g\geq 3$, on whether the source group is $\tildeMod{S_g}$, $\Modu{S_g}$, $\tModl{\ell}{S_g}$, or $\Modl{\ell}{S_g}$.

\begin{proposition}[{\cite[prop 4.1]{godfardRigidityFibonacciRepresentations2025} and \Cref{propositionrirreleventinH1}}]\label{propositionfiniteorderdeformation}
  Let $g\geq 3$, $n\geq 0$, $r\geq 1$, and $d\geq 0$, then for any representation $\rho:\tModl{r}{S_g^n}\ra \GLn{d}{\C}$:
  \begin{align*}
    &H^1(\Modl{r}{S_g^n},\ad \rho)=H^1(\tModl{r}{S_g^n},\ad \rho)\\
    &\phantom{H^1}=H^1(\tildeMod{S_g^n},\ad \rho)=H^1(\Modu{S_g^n}^r,\ad \rho).
  \end{align*}
\end{proposition}


\section{\texorpdfstring{Proof of \Cref{mainresult}}{Proof of Theorem \ref{mainresult}}}\label{sectionproof}


This section contains most arguments in the proof of the main result. Only the proof of partial homological stability for $H^1$
with coefficients in $\ad\Nu$ (\Cref{propositioninjectivity}) is deferred to \Cref{sectionpartialstability}.
For an outline of the proof, see \Cref{suboutline}.
This outline does not follow the order of this section and we advise the reader to consult it before reading the proof.

In this section, unless otherwise specified, $\Nu$ is a modular functor with set of colors $\Lambda$ according to the equivalent
\Cref{definitiongeometricmodularfunctor,definitiongeometricmodularfunctor},
satisfying properties \textbf{(I)}, \textbf{(II)} of \Cref{propositionSOmodularfunctor} and irreducibility (\Cref{propositionirreducibility}).
We will choose a level $(r,s)$ for $\Nu$ (see \Cref{definitionlevel}).

We will denote by $\ad\Nu$ the modular functor with set of colors $\Lambda^2$
defined by $\ad\Nu_g((\lambda_1,\mu_1),\dotsc,(\lambda_n,\mu_n))=\Nu_g(\lambda_1,\dotsc,\lambda_n)\otimes\Nu_g(\mu_1,\dotsc,\mu_n)^*$,
where $(-)^*$ denotes the dual vector space. The gluing isomorphisms and other extra data for $\ad\Nu$ is induced by that of $\Nu$.
Note that $\ad\Nu$ is non-anomalous. We will abbreviate $\ad\Nu_g((\lambda_1,\lambda_1),\dotsc,(\lambda_n,\lambda_n))$ to
$\ad\Nu_g(\ul)$.\footnote{$\ad\Nu_g(\ul)$ is then exactly the adjoint local system of $\Nu_g(\ul)$. Hence the notation $\ad\Nu$.}

We shall prove \Cref{mainresult} for $\Nu$. Note that by combining \Cref{propositionfiniteorderdeformation}
and \Cref{remarkboundary}, we need only prove that for $g\geq 7$, $H^1(\Mgrb{g}{}{r},\ad\Nu_g)=0$.


\subsection{Building a collection of low genus deformation classes satisfying modular functor axioms from a deformation in high genus}\label{subbuildingdeformationofmodularfunctor}


Let us fix $g\geq 7$ and a cohomology class $\psi$ in $H^1(\Mgrb{g}{}{r},\ad\Nu_g)$, i.e. an equivalence class of deformation
of the flat bundle $\Nu_g$ on $\Mgrbt{g}{}{r}{s}$.

Fix $g'\geq 4$ such that $g\geq 2g'-1$ (for example, $g'=4$). 
\begin{definition}\label{definitionembeddable}
  We will call a triplet $(h,n,\ul)$ with $h\geq 0$, $n\geq 0$ and $\ul\in\Lambda^n$
  \emph{embeddable in $S_{g'}$ for $\Nu$} if there exists a (possibly disconnected) colored surface $\Sigma$ such that
  \begin{enumerate}
    \item $S_{g'}$ can be obtained from $(S_h^n,\ul)\sqcup\Sigma$ by a gluing $G:S_h^n\sqcup\Sigma\ra S_{g'}$ of pairs of boundary components with matching colors,
    such that no $2$ boundary components of $S_h^n$ are glued together;
    \item connected components of $\Sigma$ are not homeomorphic to $S_0^0$, $S_0^1$ or $S_1^0$;
    \item if $\Nu(S_h^n,\ul)\neq 0$, then $\Nu(\Sigma)\neq 0$.
  \end{enumerate}
  The choice of such $\Sigma$ and of the gluing $G$ in (1) will be called an \emph{embedding of $(S_h^n,\ul)$ in $S_{g'}$}.
\end{definition}

\begin{notation}
  We will denote by $I_{g'}^\Nu$ the set of $(h,n,\ul)$ that are embeddable in $S_{g'}$ for $\Nu$.
\end{notation}

For example, if $1\leq n\leq 5$, then $(0,n,\ul)$ is in $I_{g'}^\Nu$ for any $\ul$.
Indeed, $(S_0^n,\ul)\sqcup (S_0^n,\ul)$ can be glued into $S_{n-1}$. As $g'\geq 4$, we have $n-1\leq g'$.
If $n-1=g'$ we are done, otherwise $(S_0^n,\ul)\sqcup \left((S_0^{n+1},(\ul,0))\sqcup (S_{g'-(n-1)}^1,0)\right)$ can be glued into $S_{g'}$.
By a similar argument, for $1\leq n\leq 3$ and $\ul\in\Lambda^n$, $(1,n,\ul)$ is in $I_{g'}^\Nu$.

\begin{remark}
  The definition above is equivalent to $\Nu_h(\ul)$ appearing as a tensor factor in a direct summand after successive applications
  of the gluing axiom to $\Nu_{g'}$, such that all other tensor factors in that direct summand are non-zero.
\end{remark}

\begin{definition}[{\cite[Def. 5.10]{godfardHodgeStructuresConformal2025}}]\label{definitiontruncationset}
  Let $\Lambda$ be a set of colors. A truncation set is a subset $I$ of $\{(g,n,\ul)\mid g,n\geq 0,\:\ul\in\Lambda^n\}$
  satisfying the following $3$ conditions.
  \begin{description}
    \item[(P)] for any $n\geq 0$ and $\sigma\in S_n$ permutation, if $(g,n,\ul)\in I$, then $(g,n,\sigma(\ul))\in I$;
    \item[(0)] $\forall 1\leq n\leq 5,\:\forall \ul\in\Lambda^n,\: (0,n,\ul)\in I$;
    \item[(1)] $\forall 1\leq n\leq 3,\:\forall \ul\in\Lambda^n,\: (1,n,\ul)\in I$.
  \end{description}
\end{definition}

By the discussion above, \emph{$I_{g'}^\Nu$ is a truncation set.}

\begin{definition}[{\cite[Def. 5.11 and 5.18]{godfardHodgeStructuresConformal2025}}]\label{definitionItruncatedmodularfunctor}
  Let $\Lambda$ be a set of colors and $I$ a truncation set for $\Lambda$. By restricting the geometric definition of modular functor
  (\Cref{definitiongeometricmodularfunctor}) to those $(h,n,\ul)$ in $I$, one obtains the definition
  of geometric $I$-truncated modular functor. More precisely, A geometric $I$-truncated modular functor is a collection
  of $\C$-local systems $(\Nu_g(\ul))_{g,n,\ul\in\Lambda^n}$ over the over $\Mgrbt{h}{n}{r}{s}$ for $(g,n)\neq (0,0),(0,1),(1,0)$ and $(g,n,\ul)\in I$,
  satisfying the conditions of \Cref{definitiongeometricmodularfunctor} whenever they make sense.
  In axioms \textbf{(G-sep)} and \textbf{(G-nonsep)}, one must restrict the direct sums $\oplus_\mu$ to those summands corresponding to triplets in $I$.
\end{definition}

\begin{remark}
  The restriction of $\Nu$ to $I_{g'}^\Nu$ is an $I_{g'}^\Nu$-truncated modular functor.
  To check this, one needs to verify that for every gluing isomorphism involving $(h,n,\ul)$ in $I_{g'}^\Nu$,
  all triplets appearing in the non-zero direct summands are also in $I_{g'}^\Nu$. And this condition is ensured by point (3) of \Cref{definitionembeddable}.
  See \cite[Def. 5.13]{godfardHodgeStructuresConformal2025} for more on this.
\end{remark}

\begin{definition}\label{definitiondeformationofmodularfunctor}
  A deformation of a modular functor $\Nu$ is a modular functor of $\C[\epsilon]$ projective local systems such that taking the quotient by $\epsilon$
  recovers $\Nu$. Similarly for $I$-truncated modular functors.
\end{definition}

Here $\C[\epsilon]$ is short for the ring of dual numbers $\C[X]/(X^2)$ with $\epsilon=X$.
The aim of the rest of this subsection is to explain how to construct, from $\psi$ in $H^1(\Mgrb{g}{}{r},\ad\Nu_g)$, a collection of deformation classes
$\psi_h(\ul)\in H^1(\Mgrb{h}{n}{r};\ad\Nu_h(\ul))$ for $(h,n,\ul)$ in $I_{g'}^\Nu$ which are compatible with the axioms
of a modular functor. The definition below makes this notion precise.

\begin{definition}\label{definitioncompatiblecollection}
  Let $I$ be a truncation set.
  A \emph{collection of compatible deformation classes} of $\Nu$ an $I$-truncated modular functor is a collection of classes
  $\psi_h(\ul)\in H^1(\Mgrb{h}{n}{r};\ad\Nu_h(\ul))$ for $(h,n,\ul)$ in $I$ which satisfy the following axioms
  whenever they make sense.
  \begin{description}
      \item[(G-sep)] For each gluing map
      \begin{equation*}
        q:\Mgrb{h_1}{n_1+1}{r}\times\Mgrb{h_2}{n_2+1}{r}\lra \Mgrb{h_1+h_2}{n_1+n_2}{r}
      \end{equation*}
      and gluing isomorphism
      \begin{equation*}
        q^*\Nu_{h_1+h_2}(\lambda_1,\dotsc,\lambda_n)\simeq \bigoplus_\mu\Nu_{h_1}(\lambda_1,\dotsc,\lambda_{n_1},\mu)\otimes
        \Nu_{h_2}(\lambda_{n_1+1},\dotsc,\lambda_n,\mu),
      \end{equation*}
      we have a map
      \begin{multline*}
        \bigoplus_\mu H^1(\Mgrb{h_1}{n_1+1}{r};\ad \Nu_{h_1}(\lambda_1,\dotsc,\mu))\oplus H^1(\Mgrb{h_2}{n_2+1}{r};\ad \Nu_{h_2}(\lambda_{n_1+1},\dotsc,\mu))\\
        \lra H^1(\Mgrb{h_1}{n_1+1}{r}\times\Mgrb{h_2}{n_2+1}{r};q^*\ad \Nu_{h_1+h_2}(\lambda_1,\dotsc,\lambda_n))
      \end{multline*}
      where the direct sum is restricted to the $\mu$ with non-zero corresponding summand in the gluing isomorphism.
      This map is induced by the Künneth formula and the maps $\C\ra \ad \Nu_{h_1}(\lambda_1,\dotsc,\mu)$ and $\C\ra \ad \Nu_{h_2}(\lambda_{n_1+1},\dotsc,\mu)$.
      We then require that this map send $\oplus_\mu \psi_{h_1}(\lambda_1,\dotsc,\mu)\oplus\psi_{h_2}(\lambda_{n_1+1},\dotsc,\mu)$ to
      $q^*\psi_{h_1+h_2}(\lambda_1,\dotsc,\lambda_n)$.
      \item[(G-nonsep)] For each gluing map
      \begin{equation*}
        p:\Mgrb{h-1}{n+2}{r}\lra \Mgrb{h}{n}{r}
      \end{equation*}
      and gluing isomorphism
      \begin{equation*}
        p^*\Nu_h(\lambda_1,\dotsc,\lambda_n)\simeq \bigoplus_\mu\Nu_{h-1}(\lambda_1,\dotsc,\lambda_n,\mu,\mu),
      \end{equation*}
      we have an embedding
      \begin{align*}
        \bigoplus_\mu& H^1(\Mgrb{h-1}{n+2}{r};\ad \Nu_{h-1}(\lambda_1,\dotsc,\lambda_n,\mu,\mu))\\
        &\lra H^1(\Mgrb{h-1}{n+2}{r};p^*\ad \Nu_{h}(\lambda_1,\dotsc,\lambda_n))
      \end{align*}
      where the direct sum is restricted to the $\mu$ with non-zero corresponding summand in the gluing isomorphism.
      We then require that this embedding send $\oplus_\mu \psi_{h-1}(\lambda_1,\dotsc,\lambda_n,\mu,\mu)$ to
      $p^*\psi_{h}(\lambda_1,\dotsc,\lambda_n)$.
      \item[(N)] For each forgetful map $f:\Mgrb{h}{n+1}{r}\ra\Mgrb{h}{n}{r}$, and forgetful isomorphism
      \begin{equation*}
        f^*\Nu_h(\lambda_1,\dotsc,\lambda_n)\simeq \Nu_h(\lambda_1,\dotsc,\lambda_n,0),
      \end{equation*}
      we require that $f^*\psi_h(\lambda_1,\dotsc,\lambda_n)=\psi_h(\lambda_1,\dotsc,\lambda_n,0)$.
      \item[(Perm)] For each $\ul\in\Lambda^n$ and permutation isomorphism
      \begin{equation*}
        \Nu_h(\lambda_1,\dotsc,\lambda_n)\simeq \sigma^*\Nu_h(\lambda_{\sigma(1)},\dotsc,\lambda_{\sigma(n)}),
      \end{equation*}
      we require that $\sigma^*\psi_h(\lambda_{\sigma(1)},\dotsc,\lambda_{\sigma(n)})=\psi_h(\lambda_1,\dotsc,\lambda_n)$.
  \end{description}
\end{definition}

\begin{remark}
  As $H^1(\Mgrb{h}{n}{r},\ad\Nu_h(\ul))$ classifies deformations of $\Nu_h(\ul)$, a deformation of the modular functor $\Nu$ induces
  a collection of compatible deformation classes on $\Nu$. That any collection of compatible deformation classes comes from a
  deformation of $\Nu$ is non-trivial. This is the subject of \Cref{sublifting} and the proof relies on (non-Abelian) Hodge theory.
\end{remark}

\begin{remark}\label{remarkvacuumreducestogluing}
  If $I=I_{g'}^\Nu$ in \Cref{definitioncompatiblecollection}, the conditions \textbf{(N)} follow from the conditions \textbf{(G-sep)}.
  Indeed, notice that if $S_h$ is embeddable in $S_g'$, then $h=g'$. Hence is both $(h,n,\ul)$
  and $(h,n+1,(\ul,0))$ are in $I_{g'}^\Nu$, then $n>0$, and one may obtain $(S_h^{n+1},(\ul,0))$ by gluing
  $(S_h^n,\ul)$ to $(S_0^3,(\lambda_1,\lambda_1,0))$. The gluing isomorphism thus obtained is equivalent to the
  forgetful isomorphism, and the condition \textbf{(N)} on $\psi_h(\lambda_1,\dotsc,\lambda_n)$ and $\psi_h(\lambda_1,\dotsc,\lambda_n,0)$
  reduces to that imposed by \textbf{(G-sep)}.
\end{remark}

For a general truncated modular functor $\Nu$, the condition imposed by the axioms \textbf{(G-sep)} and \textbf{(G-nonsep)}
are very stringent. For example, in the case of \textbf{(G-nonsep)} $p^*\ad\Nu_h(\ul)$ decomposes as
$$\bigoplus_{\mu_1,\mu_2}\Nu_{h-1}(\lambda_1,\dotsc,\lambda_n,\mu_1,\mu_1)\otimes\Nu_{h-1}(\lambda_1,\dotsc,\lambda_n,\mu_2,\mu_2)^*$$
and hence its $H^1$ also decomposes accordingly.
However the axiom imposes that $p^*\psi_h(\ul)$ must be $0$ in the $H^1$ of the summands $\mu_1\neq\mu_2$.
In the case of \textbf{(G-sep)}, an additional restriction is imposed by the axiom
$\ad\Nu_{h_1}(\lambda_1,\dotsc,\mu)\otimes \ad\Nu_{h_2}(\lambda_{n_1+1},\dotsc,\mu)$ is a direct summand of $q^*\ad\Nu_{h_1+h_2}(\ul)$,
and by the Künneth formula, its $H^1$ decomposes as
\begin{multline*}
  H^1(\Mgrb{h_1}{n_1+1}{r},\ad\Nu_{h_1}(\lambda_1,\dotsc,\mu))\otimes H^0(\Mgrb{h_2}{n_2+1}{r},\ad\Nu_{h_2}(\lambda_{n_1+1},\dotsc,\mu))\\
  \oplus H^0(\Mgrb{h_1}{n_1+1}{r},\ad\Nu_{h_1}(\lambda_1,\dotsc,\mu))\otimes H^1(\Mgrb{h_2}{n_2+1}{r},\ad\Nu_{h_2}(\lambda_{n_1+1},\dotsc,\mu)).
\end{multline*}
The axiom then imposes that $q^*\psi_{h_1+h_2}(\ul)$ be in the image of
\begin{align*}
  &H^1(\Mgrb{h_1}{n_1+1}{r},\ad\Nu_{h_1}(\lambda_1,\dotsc,\mu))\otimes \C\\
  &\oplus \C\otimes H^1(\Mgrb{h_2}{n_2+1}{r},\ad\Nu_{h_2}(\lambda_{n_1+1},\dotsc,\mu))
\end{align*}
where the map is given by the scalar endomorphism maps $\C\ra \ad\Nu_{h_1}(\lambda_{1},\dotsc,\mu)$
and $\C\ra\ad\Nu_{h_2}(\lambda_{n_1+1},\dotsc,\mu)$.

However, when $\Nu$ satisfies property \textbf{(II)} of \Cref{propositionSOmodularfunctor} and irreducibility (\Cref{propositionirreducibility}),
the conditions described above disappear, thanks to the following Lemmas.

\begin{lemma}[{\cite[Cor. 6.2]{godfardRigidityFibonacciRepresentations2025}}]\label{lemmacenterkills}
  Let $h\geq 0$, $n\geq 0$ such that $2h-2+n>0$ and $\ul,\umu\in\Lambda^n$ where $\Lambda$ is the set of colors of a modular functor $\Nu$.
  Let $(r,s)$ be a level for $\Nu$.
  Assume that $\Nu$ satisfies condition \textbf{(II)} of \Cref{propositionSOmodularfunctor}.
  Then if $\ul\neq \umu,$ the cohomology group $H^1(\Mgrb{h}{n}{r},\Nu_h(\ul)\otimes\Nu_h(\umu)^*)$ is $0$.
\end{lemma}

\begin{lemma}\label{lemmaKunneth}
  Let $\Nu$ be a modular functor with set of colors $\Lambda$, and $(r,s)$ a level for it.
  For $i=1,2$, let $h_i\geq 0$, $n_i\geq 0$ such that $2h_i-2+n_i>0$, and $\ul^i\in\Lambda^{n_i}$.
  Assume that $\Nu$ satisfies irreducibility (\Cref{propositionirreducibility}), that $\Nu_{h_1}(\ul^1)\neq 0$ and $\Nu_{h_2}(\ul^2)\neq 0$.
  Then
  \begin{multline*}
    H^1(\Mgrb{h_1}{n_1}{r}\times\Mgrb{h_2}{n_2}{r},\ad\Nu_{h_1}(\ul^1)\boxtimes\ad\Nu_{h_2}(\ul^2))\\
    \simeq H^1(\Mgrb{h_1}{n_1}{r},\ad\Nu_{h_1}(\ul^1))\oplus H^1(\Mgrb{h_2}{n_2}{r},\ad\Nu_{h_2}(\ul^2)).
  \end{multline*}
\end{lemma}
\begin{proof}
  The result follows from the discussion above and the fact that the irreducibility assumption implies that, whenever $\Nu_h(\ul)\neq 0$,
  all inclusion $\C\subset \ad\Nu_h(\ul)$ induce isomorphisms on $H^0$.
\end{proof}

Applying \Cref{lemmacenterkills,lemmaKunneth} to the discussion above, we get the following.

\begin{corollary}[{\cite[Cor. 6.5]{godfardRigidityFibonacciRepresentations2025}}]\label{corollarycenterkillsandKunneth}
  Assume $\Nu$ is a modular functor satisfying \textbf{(II)} of \Cref{propositionSOmodularfunctor} and irreducibility (\Cref{propositionirreducibility}).
  Then the maps in conditions \textbf{(G-sep)} and \textbf{(G-nonsep)} of \Cref{definitioncompatiblecollection} are isomorphisms.
  More precisely, in the context of \textbf{(G-sep)}, there is an isomorphism
  \begin{multline*}
    H^1(\Mgrb{h_1}{n_1+1}{r}\times\Mgrbt{h_2}{n_2+1}{r}{s};q^*\ad \Nu_{h_1+h_2}(\lambda_1,\dotsc,\lambda_n))\\
   \simeq\bigoplus_\mu H^1(\Mgrb{h_1}{n_1+1}{r};\ad \Nu_{h_1}(\lambda_1,\dotsc,\mu))\oplus H^1(\Mgrb{h_2}{n_2+1}{r};\ad \Nu_{h_2}(\lambda_{n_1+1},\dotsc,\mu))
  \end{multline*}
  where the direct sum is restricted to the $\mu$ with non-zero corresponding summand in the gluing isomorphism.
  In the context of \textbf{(G-nonsep)}, there is an isomorphism
  \begin{align*}
    H^1&(\Mgrb{h-1}{n+2}{r};p^*\ad \Nu_{h}(\lambda_1,\dotsc,\lambda_n))\\
    &\simeq\bigoplus_\mu H^1(\Mgrb{h-1}{n+2}{r};\ad \Nu_{h-1}(\lambda_1,\dotsc,\lambda_n,\mu,\mu))
  \end{align*}
  where the direct sum is restricted to the $\mu$ with non-zero corresponding summand in the gluing isomorphism.
\end{corollary}

This Corollary is the reason why we restrict ourselves to modular functors satisfying \textbf{(II)} and irreducibility:
given a gluing map $q$ as in \Cref{definitioncompatiblecollection},
it allows to define the deformation classes $\psi_{h_1}(\lambda_1,\dotsc,\mu)$ and $\psi_{h_2}(\lambda_{n_1+1},\dotsc,\mu)$ for all $\mu$
(with non-zero summand in the gluing isomorphism) from any class $\psi_{h_1+h_2}(\ul)\in H^1(\Mgrb{h_1+h_2}{n_1+n_2}{r},\ad\Nu_{h_1+h_2}(\ul))$.
Similarly given a map $p$ as in \Cref{definitioncompatiblecollection}, it allows to define $\psi_h(\lambda_1,\dotsc,\lambda_n,\mu,\mu)$
for $\Nu_h(\lambda_1,\dotsc,\lambda_n,\mu,\mu)\neq 0$ from $\psi_h(\lambda_1,\dotsc,\lambda_n)$.

Starting with the class $\psi\in H^1(\Mgrb{g}{}{r},\ad\Nu_g)$,
we shall construct in this manner, through successive restrictions,
a collection of compatible deformation classes $\psi_h(\ul)$ of the $I_{g'}^\Nu$-truncation of $\Nu$.
The need to restrict to $I_{g'}^\Nu$ rather than, for example, to the larger $I_g^\Nu$, comes from the fact that a local system $\Nu_h(\ul)$
may appear in different ways through sequences of successive restrictions along gluing isomorphisms,
i.e. there may be several possible embeddings $(\Sigma, G)$ of $(S_h^n,\ul)$ into $S_g$ (\Cref{definitionembeddable}).
However, when considering surfaces embeddable in $S_{g'}$ with $g\geq 2g'-1$,
and building the deformations $\psi_{h}(\ul)$ from a class $\psi_{g'}\in H^1(\Mgrb{g'}{}{r},\ad\Nu_{g'})$
induced by $\psi\in H^1(\Mgrb{g}{}{r},\ad\Nu_g)$, we can show that $\psi_{h}(\ul)$ is independent of the sequences of gluing isomorphisms used to define it.

The rest of this subsection is devoted to the construction of the $\psi_{h}(\ul)$ (\Cref{definitiondeformationclasses1}),
to the proof that they do not depend on the choices made (\Cref{corollaryindependenceofembedding}), and
the proof that they satisfy the conditions of a collection of compatible deformation classes (\Cref{propositiondeformationisMF}).

Let us first define the deformation $\psi_{g'}$. For that, we need that $H^1$ is preserved by vacuum isomorphisms.

\begin{proposition}[{\cite[Lem. 7.1]{godfardRigidityFibonacciRepresentations2025}}]\label{propositionvacuum}
  For any $h,n\geq 0$ with $2h-2+n>0$, $\ul\in\Lambda^n$, and $m\geq 0$, vacuum isomorphisms induce an isomorphism
  $$H^1(\Mgrb{h}{n+m}{r}, \ad\Nu_h(\ul,0,\dotsc,0))\simeq H^1(\Mgrb{h}{n}{r}, \ad\Nu_h(\ul)).$$
\end{proposition}

\begin{definition}\label{definitionpsigprime}
  Let $p:\Mgrb{g'}{2(g-g')}{r}\ra\Mgrb{g}{}{r}$
  be the map obtained by gluing marked points $2i$ and $2i-1$ for $1\leq i\leq g-g'$.
  Then this map together with \Cref{propositionvacuum} induce a cohomology map
  $$H^1(\Mgrb{g}{}{r},\ad\Nu_g)\lra H^1(\Mgrb{g'}{2(g-g')}{r},\ad\Nu_{g'}(0,\dotsc,0))\simeq H^1(\Mgrb{g'}{}{r},\ad\Nu_{g'}).$$
  The arrow above is induced by the projection $p^*\ad\Nu_{g}\ra \ad\Nu_{g'}(0,\dotsc,0)$ on a direct summand in the gluing isomorphism.
  We define $\psi_{g'}$ to be the image of $\psi$ under this map.
\end{definition}

\begin{remark}
  Note that any map $\Mgrb{g'}{2(g-g')}{r}\ra\Mgrb{g}{}{r}$ given by the gluing of pairs of marked points will give the same
  $\psi_{g'}$. Indeed $H^1(\Mgrb{g'}{2(g-g')}{r},\ad\Nu_{g'}(0,\dotsc,0))$ is invariant under permutation of the marked points,
  as it is isomorphic to $H^1(\Mgrb{g'}{}{r},\ad\Nu_{g'})$.
\end{remark}

\begin{definition}\label{definitiondeformationclasses1}
  Let $(h,n,\ul)$ be in $I_{g'}^\Nu$ and $e=(\Sigma,G)$ be an embedding (\Cref{definitionembeddable}).
  Let $\Sigma=\bigsqcup_i (S_{k_i}^{m_i},\umu_i)$ be the decomposition into connected components.
  Let us denote by $\Mgrb{\uk}{\um}{r}$ the product of the $\Mgrb{k_i}{m_i}{r}$.
  The gluing $G$ induces a map
  $$i_e:\Mgrb{h}{n}{r}\times\Mgrb{\uk}{\um}{r}\lra \Mgrb{g'}{}{r}.$$
  Denote by $\Nu_{\uk}(\umu)$ the local system over $\Mgrb{\uk}{\um}{r}$
  that is the external tensor product of the $\Nu_{k_i}(m_i)$. We then have maps
  \[\begin{array}{lcl}
          H^1(\Mgrb{g'}{}{r},\ad\Nu_{g'}) & \lra   & H^1(\Mgrb{h}{n}{r}\times\Mgrb{\uk}{\um}{r},\ad\Nu_h(\ul)\boxtimes\ad\Nu_{\uk}(\umu)) \\
                                          & \lra   & H^1(\Mgrb{h}{n}{r},\ad\Nu_h(\ul))\otimes H^0(\Mgrb{\uk}{\um}{r},\ad\Nu_{\uk}(\umu)) \\
                                          & \lra   & H^1(\Mgrb{h}{n}{r},\ad\Nu_h(\ul))
  \end{array}\]
  whose composition we will denote $i_e^*$.
  The second arrow above is the projection on a factor in the Künneth isomorphism,
  and the third arrow is $0$ if $\Nu_h(\ul)=0$, and is otherwise induced
  by the canonical isomorphism $\C\simeq H^0(\Mgrb{\uk}{\um}{r},\ad\Nu_{\uk}(\umu))$
  provided by irreducibility and $\Nu_{k_i}(\mu_i)\neq 0$ (the latter is an assumption in the definition of $e$ when $\Nu_h(\ul)\neq 0$).
  We then define $\psi_h^e(\ul)\in H^1(\Mgrb{h}{n}{r},\ad\Nu_h(\ul))$ to be $i_e^*\psi_{g'}$.
\end{definition}

\begin{proposition}\label{propositionembeddinginSg}
  If $(h,n,\ul)$ is in $I_{g'}^\Nu$, then there exists an embedding $(M,G)$ of $(S_h^n,\ul)$ in $S_g$
  such that $M$ is connected. 
\end{proposition}

We will call such an embedding a \emph{connected embedding}.

\begin{proof}
  Let $(\Sigma,G')$ be an embedding of $(S_h^n,\ul)$ in $S_{g'}$. Replacing $\Sigma$ if necessary,
  we may assume that $G'$ does not glue any pairs of boundary components of $\Sigma$ together.
  Now every connected component of $\Sigma$ is either an annulus or has negative Euler characteristic.
  Let $a$ be the number of annuli, and $b$ the number of other connected components.
  Then
  \begin{align*}
    2g'-2 &=    - \chi(S_{g'})= -\chi(S_h^n) - \chi(\Sigma) \\
          &\geq -\chi(S_h^n) + b = 2h-2+n+ b
          \geq -2 +n + b.
  \end{align*}
  So that $2g'\geq n+b$. Now, as each boundary component of $\Sigma$ is glued to one of $S_h^n$ by $G'$,
  we have $n\geq 2a+b$. Hence $2g'\geq 2(a+b)$. As $g\geq 2g'-1$, we have $g-g'\geq a+b-1$.
  Now, to construct a possible $(M,G)$, one may add $g-g'$ boundary components colored with $0$ to $\Sigma$ to obtain a colored surface $\tilde{\Sigma}$,
  and glue these extra boundary components into pairs such that the result $M$ is connected.
  This is always possible as any $a+b$ points can be made into a connected graph using $a+b-1$ edges.
\end{proof}

\begin{definition}\label{definitiondeformationclasses2}
  Let $(h,n,\ul)$ be in $I_{g'}^\Nu$ and $f=(M,G)$ be a connected embedding of $(S_h^{n},\ul)$ into $S_g$.
  Then for $M=(S_k^m,\umu)$, the gluing $G$ induces a map
  $$i_f:\Mgrb{h}{n}{r}\times\Mgrb{k}{m}{r}\ra \Mgrb{g}{}{r}.$$
  This map in turn induces a morphism\footnote{As in \Cref{definitiondeformationclasses1}, if $\Nu_h(\ul)=0$, this is the $0$ morphism.}
  \[\begin{array}{lcl}
    H^1(\Mgrb{g}{}{r},\ad\Nu_{g}) & \lra   & H^1(\Mgrb{h}{n}{r}\times\Mgrb{k}{m}{r},\ad\Nu_h(\ul)\boxtimes\ad\Nu_{k}(\umu)) \\
                                  & \lra   & H^1(\Mgrb{h}{n}{r},\ad\Nu_h(\ul))\otimes H^0(\Mgrb{k}{m}{r},\ad\Nu_{k}(\umu)) \\
                                  & \lra   & H^1(\Mgrb{h}{n}{r},\ad\Nu_h(\ul))
  \end{array}\]
  that we will denote $i_f^*$.
  We then define $\psi_h^f(\ul)\in H^1(\Mgrb{h}{n}{r},\ad\Nu_h(\ul))$ to be $i_f^*\psi$.
\end{definition}

\begin{proposition}\label{propositionindependenceofembedding}
  Let $(h,n,\ul)$ be in $I_{g'}^\Nu$.
  Then for any connected embedding $f$ of $(S_h^{n},\ul)$ in $S_g$
  and any embedding $e$ of $(S_h^n,\ul)$ in $S_{g'}$, we have $\psi_h^e(\ul)=\psi_h^f(\ul)$.
\end{proposition}

\begin{proof}
  We may assume that $\Nu_h(\ul)\neq 0$ as otherwise there is nothing to prove.
  Let $e=(\Sigma,G')$ and $f=(M,G)$ be the embeddings. 
  We may and will assume that $G'$ does not glue boundary components of $\Sigma$ together.
  Indeed, if $\Sigma'$ is the surface obtained by gluing the pairs of boundary components of $\Sigma$ glued by $G'$,
  then $G'$ induces an embedding $e'=(\Sigma',G'')$ of $(S_h^n,\ul)$ in $S_{g'}$. It then follows from the definitions that
  $\psi_h^e(\ul)=\psi_h^{e'}(\ul)$. Similarly, we may assume without loss of generality that $G$ does not glue boundary components of $M$ together.

  Now, for any other connected embedding $f'=(M',G'')$ of $(S_h^{n},\ul)$ in $S_g$, such that $G''$ does not glue boundary components of $M'$ together,
  $M$ and $M'$ have the same Euler characteristic,
  the same number of boundary components, and the same coloring of these boundary components.
  So $f'$ and $f$ are the same gluing.
  Hence we may assume that $f=(M,G)$ is constructed as in the proof of \Cref{propositionembeddinginSg}.

  The procedure in that proof relating $e$ and $f$ via the surface $\tilde{\Sigma}$ can be summed up in the following diagram
  \[\begin{tikzcd}
    {(S_h^{n},\ul)\sqcup M} & {S_g} \\
    {(S_h^{n},\ul)\sqcup \tilde\Sigma} & {S_{g'}^{2(g-g')}} \\
    {(S_h^n,\ul)\sqcup \Sigma} & {S_{g'}}
    \arrow["f", from=1-1, to=1-2]
    \arrow[from=2-1, to=1-1]
    \arrow[from=2-1, to=2-2]
    \arrow[from=2-1, to=3-1]
    \arrow[from=2-2, to=1-2]
    \arrow[from=2-2, to=3-2]
    \arrow["e"', from=3-1, to=3-2]
  \end{tikzcd}\]
  where the horizontal arrows and the vertical arrows pointing up are gluings, and the vertical arrows pointing down are obtained
  by capping some boundary components with disks.

  Let $\Sigma=\bigsqcup_i (S_{k_i}^{m_i},\umu_i)$ and $\tilde{\Sigma}=\bigsqcup_i (S_{k_i}^{m_i+\epsilon_i},(\umu_i,0,\dotsc,0))$
  be decompositions into connected components. Let $(S_k^m,\unu)=M$.
  Let us denote by $\Mgrb{\uk}{\um}{r}$ the product of the $\Mgrb{k_i}{m_i}{r}$ and by $\Mgrb{\uk}{\um+\ueps}{r}$
  the product of the $\Mgrb{k_i}{m_i+\epsilon_i}{r}$. 
  Then the diagram above induces a diagram
  \[\begin{tikzcd}
    {\Mgrb{h}{n}{r}\times\Mgrb{k}{m}{r}} & {\Mgrb{g}{}{r}} \\
    {\Mgrb{h}{n}{r}\times \Mgrb{\uk}{\um+\ueps}{r}} & {\Mgrb{g'}{2(g-g')}{r}} \\
    {\Mgrb{h}{n}{r}\times\Mgrb{\uk}{\um}{r}} & {\Mgrb{g'}{}{r}}
    \arrow["{i_f}", from=1-1, to=1-2]
    \arrow[from=2-1, to=1-1]
    \arrow[from=2-1, to=2-2]
    \arrow[from=2-1, to=3-1]
    \arrow[from=2-2, to=1-2]
    \arrow[from=2-2, to=3-2]
    \arrow["{i_e}"', from=3-1, to=3-2]
  \end{tikzcd}\]
  of moduli spaces. Denote by $\Nu_{\uk}(\umu)$ the local system over $\Mgrbt{\uk}{\um}{r}{s}$
  that is the exterior tensor product of the $\Nu_{k_i}(\mu_i)$. Similarly for $\Nu_{\uk}(\umu,\uzero)$.
  This in turn implies the commutation of the diagram of deformation spaces
  \[\begin{tikzcd}
    \begin{array}{c} H^1(\Mgrb{h}{n}{r},\ad\Nu_h(\ul))\\ \otimes H^0(\Mgrb{k}{m}{r},\ad\Nu_{k}(\unu)) \end{array} & {H^1(\Mgrb{g}{}{r},\ad\Nu_g)} \\
    \begin{array}{c} H^1(\Mgrb{h}{n}{r},\ad\Nu_h(\ul))\\ \otimes H^0(\Mgrb{\uk}{\um+\ueps}{r},\ad\Nu_{\uk}(\umu,\uzero)) \end{array} & {H^1(\Mgrb{g'}{2(g-g')}{r},\ad\Nu_{g'}(0,\dotsc,0))} \\
    \begin{array}{c} H^1(\Mgrb{h}{n}{r},\ad\Nu_h(\ul))\\ \otimes H^0(\Mgrb{\uk}{\um}{r},\ad\Nu_{\uk}(\umu)) \end{array} & {H^1(\Mgrb{g'}{}{r},\ad\Nu_{g'}).}
    \arrow["\simeq", from=1-1, to=2-1]
    \arrow["{i_f^*}"', from=1-2, to=1-1]
    \arrow[from=1-2, to=2-2]
    \arrow[from=2-2, to=2-1]
    \arrow["\simeq"', from=3-1, to=2-1]
    \arrow["\simeq"', from=3-2, to=2-2]
    \arrow["{i_e^*}", from=3-2, to=3-1]
  \end{tikzcd}\]
  Note that all $H^0$ groups in the above diagram are in fact canonically isomorphic to $\C$ as $\Nu$ satisfies irreducibility (\Cref{propositionirreducibility}),
  and are written down for clarity.
  Notice that by definition, the map from the top right corner to the bottom
  right corner sends $\psi$ to $\psi_{g'}$. Hence $\psi_h^e(\ul)=\psi_h^f(\ul)$.
\end{proof}

\begin{corollary}\label{corollaryindependenceofembedding}
  Let $(h,n,\ul)$ be in $I_{g'}^\Nu$, then $\psi_h^e(\ul)$ is independent of the embedding $e$ of $(S_h^n,\ul)$ in $S_{g'}$.
\end{corollary}

\begin{proof}
  By \Cref{propositionembeddinginSg}, there exists a connected embedding $f$ of $(S_h^n,\ul)$ in $S_{g}$. Then for any two embeddings $e$ and $e'$
  of $(S_h^n,\ul)$ in $S_{g'}$, by \Cref{propositionindependenceofembedding}, $\psi_h^e(\ul)=\psi_h^f(\ul)=\psi_h^{e'}(\ul)$.
\end{proof}

\begin{notation}
  For $(h,n,\ul)$ in $I_{g'}^\Nu$, we will denote by $\psi_h(\ul)\in H^1(\Mgrb{h}{n}{r},\ad\Nu_h(\ul))$
  the class $\psi_h^e(\ul)$, which is independent of the embedding $e$ by \Cref{corollaryindependenceofembedding}.
\end{notation}

\begin{proposition}\label{propositiondeformationisMF}
  The $\psi_h(\ul)\in H^1(\Mgrb{h}{n}{r},\ad\Nu_h(\ul))$ for $(h,n,\ul)$ in $I_{g'}^\Nu$ form a collection of compatible deformation
  classes for the $I_{g'}^\Nu$-truncation of $\Nu$.
\end{proposition}

\begin{proof}
  We need to check the conditions of \Cref{definitioncompatiblecollection}.
  Condition \textbf{(Perm)} follows from the definition of the $\psi_h(\ul)$. As mentioned in \Cref{remarkvacuumreducestogluing},
  condition \textbf{(N)} follows from \textbf{(G-sep)}. Hence we need only check conditions \textbf{(G-nonsep)} and \textbf{(G-sep)}.
  Note that thanks to \Cref{corollarycenterkillsandKunneth}, the inclusion maps used in these conditions are isomorphisms.
  
  Let $(h,n,\ul)$ be as in condition \textbf{(G-nonsep)} of \Cref{definitioncompatiblecollection},
  and $\mu\in \Lambda$. We may assume that $\Nu_h(\ul)\neq 0$ as otherwise there is nothing to prove.
  Let $e=(\Sigma,G)$ be an embedding of $(S_h^n,\ul)$ in $S_{g'}$.
  Then one gets an embedding $e'=(\Sigma\sqcup (S_0^2,\mu,\mu),G')$ of $(S_{h-1}^{n+2},(\ul,\mu,\mu))$ in $S_{g'}$
  by first gluing $(S_0^2,\mu,\mu)$ to $(S_{h-1}^{n+2},(\ul,\mu,\mu))$ and then applying $e$. As $e'$ factors through $e$,
  using the notation of \Cref{definitiondeformationclasses1} from the factorization
  $$\Mgrb{h-1}{n+2}{r}\times\Mgrb{\uk}{\um}{r}\times\Mgrb{0}{2}{r}\lra\Mgrb{h}{n}{r}\times\Mgrb{\uk}{\um}{r}\lra \Mgrb{g'}{}{r}$$
  we get a commutative diagram
  \[\begin{tikzcd}
    {H^1(\Mgrb{g'}{}{r},\ad\Nu_{g'})} & {H^1(\Mgrb{h}{n}{r},\ad\Nu_h(\ul))} \\
    {H^1(\Mgrb{h-1}{n+2}{r};\ad \Nu_{h-1}(\lambda_1,\dotsc,\lambda_n,\mu,\mu))} & {\bigoplus_\nu H^1(\Mgrb{h-1}{n+2}{r};\ad \Nu_{h-1}(\lambda_1,\dotsc,\lambda_n,\nu,\nu)).}
    \arrow["{i_e^*}", from=1-1, to=1-2]
    \arrow["{i_{e'}^*}"', from=1-1, to=2-1]
    \arrow["{p^*}", from=1-2, to=2-2]
    \arrow[from=2-2, to=2-1]
  \end{tikzcd}\]
  Hence the condition \textbf{(G-nonsep)} is satisfied.

  Let $h_1$, $n_1$, $h_2$, $n_2$ and $\ul$ be as in condition \textbf{(G-sep)} of \Cref{definitioncompatiblecollection} and $\mu\in\Lambda$
  with non-zero corresponding direct summand in the gluing isomorphism (if the summand is $0$, there is nothing to prove).
  Thanks to \Cref{corollarycenterkillsandKunneth} and symmetry of gluing,
  we need only check that the projection of $\psi_{h_1+h_2}(\ul)$ in
  $H^1(\Mgrb{h_1}{n_1+1}{r},\ad\Nu_{h_1}(\lambda_1,\dotsc,\lambda_{n_1},\mu))$
  is $\psi_{h_1}(\lambda_1,\dotsc,\lambda_{n_1},\mu)$. Let $e=(\Sigma,G)$ be an embedding of $(S_{h_1+h_2}^{n_1+n_2},(\ul))$
  in $S_{g'}$. Then by applying the gluing and then $e$,
  we get an embedding $e_1=(\Sigma\sqcup (S_{h_2}^{n_2+1},(\lambda_{n_1+1},\dotsc,\lambda_{n},\mu)),G')$ of
  $(S_{h_1}^{n_1+1},(\lambda_1,\dotsc,\lambda_{n_1},\mu))$ in $S_{g'}$.
  This yields a factorization
  $$\Mgrb{h_1}{n_1+1}{r}\times\Mgrb{h_2}{n_2+1}{r}\times\Mgrb{\uk}{\um}{r}\lra\Mgrb{h_1+h_2}{n_1+n_2}{r}\times\Mgrb{\uk}{\um}{r}\lra \Mgrb{g'}{}{r}$$
  and a commutative diagram
  \[\begin{tikzcd}
    {H^1(\Mgrb{g'}{}{r},\ad\Nu_{g'})} & {H^1(\Mgrb{h_1+h_2}{n_1+n_2}{r},\ad\Nu_{h_1+h_2}(\ul))} \\
    {        H^1(\Mgrb{h_1}{n_1+1}{r};\ad \Nu_{h_1}(\lambda_1,\dotsc,\mu))} & \begin{array}{c}         \bigoplus_\nu H^1(\Mgrb{h_1}{n_1+1}{r};\ad \Nu_{h_1}(\lambda_1,\dotsc,\nu))\\\oplus H^1(\Mgrb{h_2}{n_2+1}{r};\ad \Nu_{h_2}(\lambda_{n_1+1},\dotsc,\nu)). \end{array}
    \arrow["{i_e^*}", from=1-1, to=1-2]
    \arrow["{i_{e_1}^*}"', from=1-1, to=2-1]
    \arrow["{q^*}", from=1-2, to=2-2]
    \arrow[from=2-2, to=2-1]
  \end{tikzcd}\]
  Hence the condition \textbf{(G-sep)} is satisfied.
\end{proof}


\subsection{Lifting the deformation classes to a deformation of the modular functor using Hodge theory}\label{sublifting}


The aim of this section is to prove the following.

\begin{theorem}\label{theoremlifting}
  Let $I$ be a truncation set and
  let $\Nu$ be an $I$-truncated modular functor. Then any collection of compatible deformation classes for $\Nu$
  lifts to a deformation of $\Nu$.
\end{theorem}

\begin{remark}
  The Theorem also applies to (non-truncated) modular functor by choosing $I=\{(h,n,\ul)\mid h,n\geq 0,\;\ul\in\Lambda^n\}$.
\end{remark}

Let $\uE$ be a local system on a smooth proper DM stack $X$ and $(E,\nabla)$ the corresponding $C^\infty$ bundle with flat connection.
We will denote by $\ad\uE$ and $(\End{E},\nabla)$ the corresponding adjoints.
The cohomology $H^*(X,\ad\uE)$ is then computed by the $\Cinf$ de Rham complex $(\cA^\bullet(\End{E}),\nabla)$.
Denote by $Z^1_{dR}(X,\ad\uE)$ the set of closed $1$-forms in this complex.
There is a natural map
\[\begin{array}{rcl}
  Z^1_{dR}(X,\ad\uE) &\lra        & \{\text{projective }\C[\eps]-\text{local systems extending }\uE\}\\
  \eta             &\longmapsto & \text{the sheaf of } (\nabla+\eps \eta)-\text{flat sections of }E\otimes_\C\C[\eps].
\end{array}\]
The map clearly sends $\eta$ to a sheaf of $\C[\eps]$-modules. To check that this sheaf is locally free, one needs to find
systems of local solutions to $(\nabla+\eps\eta)(a+\epsilon b)=0$. If $\nabla a=0$, then the equation simplifies to
$\nabla b=-\eta a$, which can be solved completely as $\eta a$ is closed and hence locally exact. So, locally, a system of solutions
for $\nabla$ can be extended to one for $\nabla+\eps\eta$.

\begin{remark}\label{remarkreductiontolifttoZ1}
  To prove \Cref{theoremlifting}, thanks to the map above, one needs only lift any collection of compatible deformation classes
  $\psi_h(\ul)\in H^1(\Mgrb{h}{n}{r},\ad \Nu_h(\ul))$
  to a collection of closed $1$-forms $\eta_h(\ul)\in Z^1(\Mgrb{h}{n}{r},\ad \Nu_h(\ul))$ satisfying the
  compatibility conditions of \Cref{definitioncompatiblecollection} with each $H^1$ replaced by the corresponding $Z^1$.
\end{remark}

This lift can be achieved using a much more general result from non-Abelian Hodge theory, which we now explain.
For $X$ a smooth proper \emph{Kähler} DM stack, from the $\partial\dol$-Lemma follows the well known Hodge decomposition
$$H^1(X,\C)\simeq H^0(X,\Omega_X^1)\oplus \overline{H^0(X,\Omega_X^1)}$$
where $\Omega_X$ is the sheaf of holomorphic $1$-forms. Notice that $H^0(X,\Omega_X^1)\oplus \overline{H^0(X,\Omega_X^1)}$
is exactly the subspace of smooth $1$-forms in $\cA^1(\C)$ that are killed by both the holomorphic derivation $\partial$ and
antiholomorphic derivation $\dol$.

As the Hodge decomposition extends seamlessly to $H^1(X,\uE)$ for any unitary local system $\uE$ on $X$,
one obtains a canonical section $H^1(X,\uE)\ra Z^1_{dR}(X,\uE)$ whose image is exactly the $1$-forms killed by both $\partial$ and $\dol$.
It is canonical in the sense that it is compatible with taking direct sums and tensor products of unitary local systems,
and with taking pullbacks of local systems along algebraic maps $Y\ra X$ between Kähler smooth proper DM stacks.
These compatibilities follow from those of $\partial$ and $\dol$.

Using non-Abelian Hodge theory, this picture can be extended further to the case of semisimple local systems.
In this theory, to any semisimple bundle with flat connection $(E,D)$ on a smooth proper Kähler algebraic DM stack
is associated a canonical decomposition $D=D'+D''$ of the connection such that a $D'D''$-Lemma holds on the corresponding
de Rham complex $(\cA^\bullet(E),D)$ (see \cite[§1]{simpsonHiggsBundlesLocal1992} for the case of Kähler manifolds).

\begin{theorem}\label{theoremuniversalsectionofH1}
  Let $\cE$ be a semisimple local system on $X$ a smooth proper Kähler algebraic DM stack.
  Denote by $(E,D)$ the associated bundle with flat connection and $D=D'+D''$ the decomposition from non-Abelian Hodge theory.
  Then the subspace of $D'$ and $D''$-closed $1$-forms in $\cA^1(\uE)$ maps isomorphically to
  $H^1(X,\uE)$ and provides a section $s_{\uE}:H^1(X,\uE)\ra Z^1_{dR}(X,\uE)$.
  This section is compatible with taking direct sums and tensor products of semisimple local systems,
  and with pullbacks along algebraic maps $Y\ra X$.
\end{theorem}

Let us spell out here what the compatibilities with respect to tensor products and pullbacks mean.
Let $\uE_1$, $\uE_2$ be semisimple local systems on $X$ smooth proper Kähler algebraic DM stack.
Then the diagram
\[\begin{tikzcd}
	{H^1(X,\uE_1\otimes\uE_2)} & {H^1(X,\uE_1)\otimes H^0(X,\uE_2)\oplus H^0(X,\uE_1)\otimes H^1(X,\uE_2)} \\
	{Z^1_{dR}(X,\uE_1\otimes\uE_2)} & {Z^1_{dR}(X,\uE_1)\otimes H^0(X,\uE_2)\oplus H^0(X,\uE_1)\otimes Z^1_{dR}(X,\uE_2)}
	\arrow["{s_{\uE_1\otimes\uE_2}}"', from=1-1, to=2-1]
	\arrow[from=1-2, to=1-1]
	\arrow["{s_{\uE_1}\otimes\id_{\uE_2}\oplus \id_{\uE_1}\otimes s_{\uE_2}}", from=1-2, to=2-2]
	\arrow[from=2-2, to=2-1]
\end{tikzcd}\]
must commute.

For $f:Y\ra X$ an algebraic map between smooth proper Kähler algebraic DM stacks and $\uE$ a semisimple local system on $X$,
the following diagram commutes.
\[\begin{tikzcd}
	{H^1(X,\uE)} & {H^1(Y,f^*\uE)} \\
	{Z^1_{dR}(X,\uE)} & {Z^1_{dR}(Y,f^*\uE).}
	\arrow["{f^*}", from=1-1, to=1-2]
	\arrow["{s_{\uE}}"', from=1-1, to=2-1]
	\arrow["{s_{f^*\uE}}", from=1-2, to=2-2]
	\arrow["{f^*}", from=2-1, to=2-2]
\end{tikzcd}\]

\begin{proof}[Proof of \Cref{theoremuniversalsectionofH1}]
  We will use the results on Kähler identities in \cite[§2]{simpsonHiggsBundlesLocal1992}. In that reference, the results are proved
  for compact Kähler manifolds, but also hold for Kähler compact analytic DM stacks: see \cite[Th. 9.7]{simpsonLocalSystemsProper2011}
  for a proof of the existence and uniqueness of the decomposition $D=D'+D''$ ;
  the Kähler identities for $D$, $D'$ and $D''$ then follow, as in \cite[pp. 14-15]{simpsonHiggsBundlesLocal1992},
  from the fact that the stack is Kähler.

  The space of harmonic $1$-forms for $D$ then project bijectively onto $H^1(X,\uE)$.
  Thanks to the coincidence up to scalar of the Laplacians of $D$, $D'$ and $D''$, any harmonic $1$-form is $D'$ and $D''$-closed.
  Hence the space of $D'$ and $D''$-closed $1$-forms surjects onto $H^1(X,\uE)$.
  Let $\alpha$ be such a $1$-form with trivial cohomology class.
  Then $\alpha$ is $D=D'+D''$ and $D'$-closed and $D$ exact. By the $DD'$-Lemma
  ($DD'$ variant with same proof as the $D'D''$-Lemma \cite[Lem. 2.1]{simpsonHiggsBundlesLocal1992}),
  $\alpha$ is in the image of $DD'$. But $\alpha$ is a $1$-form, so $\alpha=0$.
  Hence the space of $D'$ and $D''$-closed $1$-forms maps bijectively to $H^1(X,\uE)$.
  The compatibilities follow from those of the decomposition $D=D'+D''$ (see \cite[p. 18]{simpsonHiggsBundlesLocal1992}).
\end{proof}

\begin{remark}
  Note that only at the end of the above proof we used that we are working with degree $1$ cohomology.
  In fact, \Cref{theoremuniversalsectionofH1} is a special case of a formality result for the dg-category
  of semisimple local systems on a smooth proper Kähler algebraic DM stack
  (\cite[Cor. 2.2, §3]{simpsonHiggsBundlesLocal1992},
  see also \cite[§3]{godfardSemisimplicityConformalBlocks2025a} for an exposition).
\end{remark}

\begin{proof}[Proof of \Cref{theoremlifting}]
  Let $(\psi_h(\ul))_{h,\ul}$ be a collection of compatible deformation classes.
  The stacks $\Mgrb{h}{n}{r}$ are Kähler. Indeed, their coarse space is projective, hence supports a Kähler form
  (they may have singularities, but that is not an issue), and \cite[Lem. 1]{wuCompactKahlerOrbifold2023}
  explains how to build a Kähler form on a DM stack from one on its coarse space (see also \cite[Rk. 2.4]{eyssidieuxOrbifoldKahlerGroups2021}).
  Hence we may use \Cref{theoremuniversalsectionofH1}
  to lift the collection to a collection of endomorphism $1$-forms $(\eta_h(\ul))_{h,\ul}$.
  Now, by compatibility of the sections $s_{\uE}$, the $(\eta_h(\ul))_{h,\ul}$
  satisfy the compatibility conditions of \Cref{definitioncompatiblecollection} with each $H^1$ replaced by the corresponding $Z^1$.
  As mentioned in \Cref{remarkreductiontolifttoZ1}, this provides a deformation of $\Nu$ with associated deformation classes $(\psi_h(\ul))_{h,\ul}$.
\end{proof}


\subsection{Application of Ocneanu rigidity to the deformation}\label{subocneanu}


In this section, we apply one of the main ingredients in the proof of \Cref{mainresult}: the Ocneanu rigidity Theorem for modular functors.
We first state Ocneanu rigidity for fusion categories. See \cite[4.1, 8.10, 8.13]{etingofTensorCategories2015} for more on fusion, ribbon fusion and modular fusion categories.
A textbook account of Ocneanu rigidity in this form is available in \cite[chp. 9.1]{etingofTensorCategories2015}.

\begin{theorem}[{Ocneanu rigidity \cite[Th. 2.28]{etingofFusionCategories2005}}]\label{Ocneanurigidity}
  A fusion category does not admit nontrivial infinitesimal deformations (i.e. its associator does not).
\end{theorem}

\begin{corollary}[{\cite[Rk. 2.33]{etingofFusionCategories2005}}]
  Any infinitesimal deformation of a modular category in which only the associator varies is trivial.\footnote{The
  assumption on the associator is in fact superfluous. We make it because it will hold in our situation and simplifies the proof,
  see \cite[Cor. 7.5]{godfardHodgeStructuresConformal2025} for example.}
\end{corollary}

Ocneanu rigidity then lifts to truncated modular functors thanks to the following variant of the Bakalov-Kirillov
equivalence between modular categories and modular functors \cite[5.4.1, 6.7.13]{bakalovLecturesTensorCategories2000}.

\begin{theorem}[{\cite[Th. 5.16]{godfardHodgeStructuresConformal2025}}]\label{forgetsoverCandCeps}
  Let $\Lambda$ be a set of colors, $c\in \C^\times$ and $I$ a truncation set.
  Let $\Mod_{c}^I$ be the groupoid of $I$-truncated modular functors with central charge $c$ (\cite[Def. 5.11]{godfardHodgeStructuresConformal2025})
  and $\Modcat_\Lambda$ the groupoid of modular fusion categories with set of colors $\Lambda$ (\cite[Th. 5.16]{godfardHodgeStructuresConformal2025}).
  Then there exists a fully-faithful functor $\Mod_{c}^I\ra\Modcat$.
  Similarly the groupoid of infinitesimal deformations of $I$-truncated modular functors admits a fully-faithful functor
  to the groupoid of infinitesimal deformations of modular fusion categories.
\end{theorem}

For the statement on infinitesimal deformations, set $R=\C[\eps]$ in \cite[Th. 5.16]{godfardHodgeStructuresConformal2025}.
As a consequence of \Cref{Ocneanurigidity,forgetsoverCandCeps}, one gets.

\begin{theorem}\label{OcneanurigidityItruncatedmodularfunctors}
  Let $\Lambda$ be a set of colors, $I$ a truncation set and $\Nu$ an $I$-truncated modular functor.
  Then any deformation of $\Nu$ which is trivial on $\Mgrbt{0}{2}{r}{s}$ and $\Mgrbt{0}{3}{r}{s}$ is trivial.
\end{theorem}

The condition on the bundles over $\Mgrbt{0}{2}{r}{s}$ and $\Mgrbt{0}{3}{r}{s}$ ensures that in the associated deformation
of the modular category, only the associator varies.
Combining this with \Cref{theoremlifting} and \Cref{propositiondeformationisMF}, one gets the following.
Here $\Nu$, $g$ and $g'$ are as in \Cref{subbuildingdeformationofmodularfunctor}.

\begin{proposition}\label{propositionvanishingofresctrictiontogprime}
  The map
  $$H^1(\Mgrb{g}{}{r};\ad\Nu_g)\lra H^1(\Mgrb{g'}{}{r};\ad\Nu_g)$$
  defined in \Cref{definitionpsigprime} is zero.
\end{proposition}
\begin{proof}
  Let $\psi\in H^1(\Mgrb{g}{}{r};\ad\Nu_g)$. In \Cref{subbuildingdeformationofmodularfunctor},
  a compatible collection of deformation classes $(\psi_h(\ul))_{h,n,\ul\in\Lambda}$ for the $I_{g'}^\Nu$-truncation of $\Nu$
  was constructed. By \Cref{theoremlifting}, this collection of deformation classes comes from a genuine
  infinitesimal deformation of the $I_{g'}^\Nu$-truncation of $\Nu$.
  As moduli spaces for $(g,n)=(0,2),(0,3)$ are zero dimensional, the corresponding deformation classes are $0$.
  Hence, by linearity of the lift in \Cref{theoremuniversalsectionofH1}, the corresponding deformations are trivial.
  Applying \Cref{OcneanurigidityItruncatedmodularfunctors},
  we see that the deformation of the $I_{g'}^\Nu$-truncation of $\Nu$ is trivial and hence that all the deformation classes $\psi_h(\ul)$ are $0$.
  In particular, $\psi_{g'}=0$.
  But $\psi_{g'}$ is exactly the image of $\psi$ by the above map.
\end{proof}


\subsection{Partial homological stability in degree \texorpdfstring{$\mathbf{1}$}{1} and conclusion of the proof of
\texorpdfstring{\Cref{mainresult}}{\ref{mainresult}}}\label{subpartialhomologicalstability}


The aim of this section is to use Hodge theoretic techniques to prove that for $\Nu$
as in \Cref{subbuildingdeformationofmodularfunctor},
the map of \Cref{definitionpsigprime} is in injective.

Note that we already know by \Cref{propositionvanishingofresctrictiontogprime} that the same map is zero.
Hence under these assumptions, $H^1(\Mgrb{g}{}{r};\ad\Nu_g)$ will vanish.

\begin{proposition}\label{propositioninjectivity}
  Let $h\geq 4$ and $\Nu$ a modular functor satisfying property \textbf{(I)} of \Cref{propositionSOmodularfunctor}
  and irreducibility (\Cref{propositionirreducibility}). The map $\Mgrb{h}{1}{r}\times\Mgrb{1}{1}{r}\ra \Mgrb{h+1}{}{r}$
  induces a chain of maps
  \[\begin{array}{rl}
    H^1(\Mgrb{h+1}{}{r},\ad\Nu_{h+1})&\ra H^1(\Mgrb{h}{1}{r}\times\Mgrb{1}{1}{r},\ad\Nu_{h}(0)\boxtimes\ad\Nu_1(0)) \\
                                     &\simeq H^1(\Mgrb{h}{1}{r},\ad\Nu_{h}(0))\oplus H^1(\Mgrb{1}{1}{r},\ad\Nu_1(0)) \\
                                     &\ra H^1(\Mgrb{h}{1}{r},\ad\Nu_{h}(0)) \\
                                     &\simeq H^1(\Mgrb{h}{}{r},\ad\Nu_{h}).
  \end{array}\]
  Here \Cref{corollarycenterkillsandKunneth} is used to define the first isomorphism and the last isomorphism follows from \Cref{propositionvacuum}.
  Their composition $H^1(\Mgrb{h+1}{}{r},\ad\Nu_{h+1})\ra H^1(\Mgrb{h}{}{r},\ad\Nu_{h})$ is \emph{injective}.
\end{proposition}

We will deduce \Cref{propositionindependenceofembedding} from the following general result which is of independent interest,
and whose proof we defer to \Cref{sectionpartialstability}.

\begin{theorem*}[{\ref{corollaryinjectivity}, case (1) with $n=0$}]
  Let $r\geq 1$ and $h,k\geq 1$ with $h\geq 2k+1$. Any choice of point in $\Mgrb{k}{1}{r}$
  induces a map $j:\Mgrb{h-k}{1}{r}\ra\Mgrb{h-k}{1}{r}\times\Mgrb{k}{1}{r}\ra \Mgrb{h}{}{r}$ (whose homotopy type is independent
  of that choice). Then for any semisimple local system $\uE$ on $\Mgrb{h}{}{r}$, the map
  $$H^1(\Mgrb{h}{}{r},\uE)\lra H^1(\Mgrb{h-k}{1}{r},j^*\uE)$$
  is injective.
\end{theorem*}

\begin{proof}[Proof of \Cref{propositioninjectivity}]
  The idea is to apply the Theorem with $k=2$ to $H^1(\Mgrb{h+1}{}{r},\ad\Nu_{h+1})$
  and with $k=1$ to $H^1(\Mgrb{h}{}{r},\ad\Nu_{h})$. This is possible
  because $h\geq 4$ and the local systems $\ad\Nu_{h+1}$ and $\ad\Nu_{h}$ are semisimple as adjoints
  of simple local systems of characteristic $0$ (alternatively, one can apply \cite[Th. 1.1]{godfardSemisimplicityConformalBlocks2025a}
  to $\ad\Nu$). Using \Cref{corollarycenterkillsandKunneth} one gets the following commutative diagram.
  \[\begin{tikzcd}
    {H^1(\Mgrb{h+1}{}{r},\ad\Nu_{h+1})} & \begin{array}{c} \bigoplus_\mu H^1(\Mgrb{h-1}{1}{r},\ad\Nu_{h-1}(\mu))\\ \otimes H^0(\Mgrb{2}{1}{r},\ad\Nu_2(\mu)) \end{array} \\
    \\
    \begin{array}{c} \bigoplus_\nu H^1(\Mgrb{h}{1}{r},\ad\Nu_{h}(\nu))\\\otimes H^0(\Mgrb{1}{1}{r},\ad\Nu_1(\nu)) \end{array} & \begin{array}{c} \bigoplus_{\nu,\mu} H^1(\Mgrb{h-1}{1}{r},\ad\Nu_{h-1}(\mu))\\ \otimes H^0(\Mgrb{1}{1}{r},\ad\Nu_1(\mu,\nu))\\ \otimes H^0(\Mgrb{1}{1}{r},\ad\Nu_1(\nu)) \end{array} \\
    \\
    {H^1(\Mgrb{h}{}{r},\ad\Nu_{h})} & \begin{array}{c} \bigoplus_\mu H^1(\Mgrb{h-1}{1}{r},\ad\Nu_{h-1}(\mu))\\ \otimes H^0(\Mgrb{1}{1}{r},\ad\Nu_1(\mu)) \end{array}
    \arrow["{u_{h+1}}", from=1-1, to=1-2]
    \arrow[from=1-1, to=3-1]
    \arrow[from=1-2, to=3-2]
    \arrow[from=3-1, to=3-2]
    \arrow["\begin{array}{c} \text{projecting on }\nu=0\text{ factor and}\\\text{applying }H^0(\ad\Nu_1(0))\simeq\C \end{array}"{description}, from=3-1, to=5-1]
    \arrow["\begin{array}{c} \text{projecting on }\nu=0\text{ factor and}\\\text{applying }H^0(\ad\Nu_1(0))\simeq\C \end{array}"{description}, from=3-2, to=5-2]
    \arrow[from=5-1, to=5-2]
  \end{tikzcd}\]
  Now notice that the map $H^1(\Mgrb{h+1}{}{r},\ad\Nu_{h+1}))\lra H^1(\Mgrb{h-1}{1}{r},j^*\ad\Nu_{h+1})$ factors through the map $u_{h+1}$ of the diagram.
  Hence as the former is injective by \Cref{corollaryinjectivity}, $u_{h+1}$ is injective.
  Now, by property \textbf{(II)} of \Cref{propositionSOmodularfunctor}, for any $\mu$,
  $\ad\Nu_2(\mu)$ vanishes if and only if $\ad\Nu_1(\mu)$ does. And if they don't, by the irreducibility assumption,
  both $H^0(\Mgrb{2}{1}{r},\ad\Nu_2(\mu))$ and $H^0(\Mgrb{1}{1}{r},\ad\Nu_1(\mu))$ are one dimensional and the map
  $H^0(\Mgrb{2}{1}{r},\ad\Nu_2(\mu))\ra H^0(\Mgrb{1}{1}{r},\ad\Nu_1(\mu))$ as in the diagram is an isomorphism.
  Hence the map
  \begin{multline*}
    \bigoplus_\mu H^1(\Mgrb{h-1}{1}{r},\ad\Nu_{h-1}(\mu)) \otimes H^0(\Mgrb{2}{1}{r},\ad\Nu_2(\mu)) \\
    \lra \bigoplus_\mu H^1(\Mgrb{h-1}{1}{r},\ad\Nu_{h-1}(\mu)) \otimes H^0(\Mgrb{1}{1}{r},\ad\Nu_1(\mu))
  \end{multline*}
  is injective. Hence the map from the top left to the bottom right in the diagram is injective.
  By commutativity of the diagram,
  $$H^1(\Mgrb{h+1}{}{r},\ad\Nu_{h+1})\ra H^1(\Mgrb{h}{}{r},\ad\Nu_{h})$$
  is also injective.
\end{proof}

\begin{proposition}\label{propositioncoincidenceofstabilizationmaps}
  Let $g\geq g'$, then $g-g'$ applications of the map in \Cref{propositioninjectivity}
  induce a map
  $$H^1(\Mgrb{g}{}{r};\ad\Nu_g)\lra H^1(\Mgrb{g'}{}{r};\ad\Nu_{g'}).$$
  Then this map coincides with the one of \Cref{definitionpsigprime}.
\end{proposition}

\begin{proof}
  By an easy check, we may assume $g'=g-1$. For any choice of point in $\Mgrb{0}{3}{r}$ we have the commutative diagram below,
  where all maps are gluing or forgetful maps, except for the map $\Mgrb{g-1}{1}{r}\ra\Mgrb{g-1}{1}{r}\times\Mgrb{0}{3}{r}$
  which is induced by the choice of $*\ra\Mgrb{0}{3}{r}$.
  \[\begin{tikzcd}
    {\Mgrb{g-1}{1}{r}\times\Mgrb{0}{3}{r}} & {\Mgrb{g-1}{1}{r}\times\Mgrb{1}{1}{r}} \\
    {\Mgrb{g-1}{1}{r}} & {\Mgrb{g}{}{r}} \\
    {\Mgrb{g-1}{}{r}} & {\Mgrb{g-1}{2}{r}.}
    \arrow[from=1-1, to=1-2]
    \arrow[from=1-1, to=3-2]
    \arrow[from=1-2, to=2-2]
    \arrow[from=2-1, to=1-1]
    \arrow[from=2-1, to=3-1]
    \arrow[from=3-2, to=2-2]
    \arrow[from=3-2, to=3-1]
  \end{tikzcd}\]
  This diagram in turn induces the following commutative diagram.
  \[\begin{tikzcd}
    \begin{array}{c} H^1(\Mgrb{g-1}{1}{r},\ad\Nu_{g-1}(0))\\\oplus H^1(\Mgrb{0}{3}{r},\ad\Nu_{0}(0,0,0)) \end{array} & \begin{array}{c} H^1(\Mgrb{g-1}{1}{r},\ad\Nu_{g-1}(0))\\ \oplus H^1(\Mgrb{1}{1}{r},\ad\Nu_{0}(0)) \end{array} \\
    {H^1(\Mgrb{g-1}{1}{r},\ad\Nu_{g-1}(0))} & {H^1(\Mgrb{g}{}{r},\ad\Nu_{g})} \\
    {H^1(\Mgrb{g-1}{}{r},\ad\Nu_{g-1})} & {H^1(\Mgrb{g-1}{2}{r},\ad\Nu_{g-1}(0,0)).}
    \arrow[from=1-1, to=2-1]
    \arrow[from=1-2, to=1-1]
    \arrow[from=2-2, to=1-2]
    \arrow[from=2-2, to=3-2]
    \arrow["\simeq", from=3-1, to=2-1]
    \arrow["\simeq"', from=3-1, to=3-2]
    \arrow[from=3-2, to=1-1]
  \end{tikzcd}\]
  Now the map of \Cref{definitionpsigprime} is the map in the diagram through the bottom right corner,
  while that of \Cref{propositioninjectivity} is the map through the top right and the top left corners.
  By commutativity of the diagram, they are equal.
\end{proof}

\begin{proof}[Conclusion of the proof of \Cref{mainresult}]
  Let $g\geq 7$. Choose $g'\geq 4$ with $g\geq 2g'-1$.
  Let $\Nu$ be a modular functor satisfying the assumptions of \Cref{remarkassumptions}.
  Then, by \Cref{propositionvanishingofresctrictiontogprime}, the map
  $$H^1(\Mgrb{g}{}{r};\ad\Nu_g)\lra H^1(\Mgrb{g'}{}{r};\ad\Nu_{g'})$$
  of \Cref{definitionpsigprime} is zero. Applying \Cref{propositioncoincidenceofstabilizationmaps,propositioninjectivity},
  the same map is injective. Hence $H^1(\Mgrb{g}{}{r};\ad\Nu_g)=0$.
  By \Cref{propositionfiniteorderdeformation} and \Cref{remarkboundary}, $H^1(\Modu{S_g},\ad \Nu(S_g))=0$.
\end{proof}


\section{Partial homological stability for first cohomology groups via Hodge theory}\label{sectionpartialstability}


In this section, we prove \Cref{corollaryinjectivity}, which states that some maps between twisted moduli spaces induce injections
at the level of $H^1$ groups with coefficients in semisimple local system. This result was crucial to the proof of \Cref{propositioninjectivity}
on partial homological stability for $H^1(\Mgrb{h}{}{r},\ad\Nu_h)$.
The main tool in this section is harmonic representatives in non-Abelian Hodge theory (see \Cref{theoremuniversalsectionofH1}).

\Cref{corollaryinjectivity} is not a statement about quantum representations in particular, and in \Cref{subdigressionIvanov}
we take the time to explain how it implies some remarks about the Ivanov conjecture and $2$ other rigidity conjectures on mapping class groups.
For example, each of these conjectures for $g\geq 3$ and $n\geq 0$ reduces to $(g,n)=(3,3)$.
These remarks were essentially known to experts, but not written down to our knowledge.


\subsection{A Lemma on subvarieties generating the fundamental group}\label{sublemmainjectivity}


In this subsection, we explain our main tool, \Cref{lemmainjectivity}, for partial homology stability in cohomological degree $1$.

\begin{proposition}[non-Abelian Hurwitz theorem]\label{propositionnonabHurwitz}
  Let $X$ be a connected topological space or more generally a connected topological DM stack. Let $k$ be a field and let $\uE\ra X$
  be a finite dimensional local system of $k$-vector spaces. Let $x\in X$, and denote by $\rho$ the monodromy representation
  of $\uE$ at $x$. The the map $H^1(X,\uE)\ra H^1(\pi_1(X,x),\rho)$ induced by $X\ra B\pi_1(X,x)$ is an isomorphism.
\end{proposition}

Although this is a well known result, we provide a proof for completeness.

\begin{proof}
  The Serre spectral sequence for the fiber sequence $\tilde{X}\ra X\ra B\pi_1(X,x)$ with coefficients in $\uE$ is
  $$E_2^{pq}=H^p(\pi_1(X,x),H^q(\tilde{X},k)\otimes E_x)\Rightarrow H^{p+q}(X,\uE).$$
  Then $E_2^{01}=H^0(\pi_1(X,x),H^1(\tilde{X},k)\otimes E_x)=0$ as $H^1(\tilde{X},k)=0$.
  Moreover, looking at differentials, one sees that $E_2^{10}=E_\infty^{10}$. Hence $H^1(\pi_1(X,x),\rho)\simeq E_2^{10}\simeq H^1(X,\uE)$.
\end{proof}

We will need the comparison map above at the level of $1$-cocycles in the case where $X$ is a $\Cinf$ DM stack.
In that case, we will represent classes in $H^1(X,\uE)$ by elements of the set $Z^1_{dR}(X,\uE)$ of de Rham $1$-cocycles.
For $H^1(\pi_1(X,x),\rho)$ with $\rho:\pi_1(X,x)\ra \GL{E_x}$, we will use the usual group cohomology complex, with
\[\begin{array}{rcl}
  Z^1(\pi_1(X,x),\rho) &=& \{\phi:\pi_1(X,x)\ra E_x\mid \forall a,b,\:\phi(ab)=\phi(a)+a\cdot\phi(b)\} \\
  B^1(\pi_1(X,x),\rho) &=& \{a\mapsto a\cdot v-v\mid v\in E_x\}.
\end{array}\]
The map $H^1(X,\uE)\ra H^1(\pi_1(X,x),\rho)$ then lifts to the map
\[\begin{array}{rcl}
  Z^1_{dR}(X,\uE) &\lra        & Z^1(\pi_1(X,x),\rho) \\
  \alpha          &\longmapsto & \int_{\tilde{\gamma}}\tilde{\alpha}
\end{array}\]
where $\tilde{\gamma}$ and $\tilde{\alpha}$ are lifts to the universal cover $\tilde{X}$ and
the identification of the pullback of $\uE$ to $\tilde{X}$ with $E_x\times \tilde{X}$ is used.
Here we see $\tilde{\gamma}$ on the stack $\tilde{X}$ as a collection of paths on local étale charts glued together.

\begin{lemma}\label{lemmainjectivity}
  Let $X$ be a connected Kähler proper smooth algebraic DM stack over $\C$, $i:Y\ra X$ a closed connected subvariety
  and $\hat{Y}\ra Y$ a map from a Kähler proper smooth algebraic DM stack which restrict to a dense open embedding
  over some open subset $U\subset \hat{Y}$. Denote by $f$ the map $\hat{Y}\ra X$.

  Then for any semisimple complex local system $\uE$ over $X$, the maps
  $H^1(X,\uE)\ra H^1(\hat{Y},f^*\uE)$ and $H^1(X,\uE)\ra H^1(Y,i^*\uE)$ have the same kernel.

  In particular, if $\pi_1(Y)\ra\pi_1(X)$ is surjective, then the map $H^1(X,\uE)\ra H^1(\hat{Y},f^*\uE)$ is injective.
\end{lemma}

\begin{remark}
  When the local system $\uE$ is constant, or more generally when it is unitary,
  \Cref{lemmainjectivity} and its generalizations to higher cohomology groups is a well known consequence of mixed Hodge theory,
  see \cite[Prop. 8.2.7]{deligneTheorieHodgeIII1974}.
  Here, we provide a different proof that works for any semisimple local system, but is limited to degree $1$.
\end{remark}

\begin{remark}\label{remarkpullbacksemisimple}
  By non-Abelian Hodge theory,
  the pullback of a semisimple local system along an algebraic map of smooth proper DM stacks is semisimple
  (the pullback of a pluri-harmonic metric is a pluri-harmonic metric, see \cite[p. 18]{simpsonHiggsBundlesLocal1992} and \cite[Th. 9.7]{simpsonLocalSystemsProper2011}).
  Hence, in the statement of the Lemma, the local system $f^*\uE$ is semisimple. The proof will use this fact.
\end{remark}

\begin{proof}[Proof of \Cref{lemmainjectivity}]
  We will make use of the functorial sections $s_{\uG}:H^1(Z,\uG)\ra Z^1_{dR}(Z,\uG)$ of \Cref{theoremuniversalsectionofH1}.
  These are defined for any Kähler proper smooth algebraic DM stack $Z$ and semisimple complex local system $\uG$.
  They provide the following diagram, where both squares commute.
  \[\begin{tikzcd}
    {Z^1_{dR}(X,\uE)} & {Z^1_{dR}(\hat{Y},f^*\uE)} \\
    {H^1(X,\uE)} & {H^1(\hat{Y},f^*\uE).}
    \arrow["{f^*}", from=1-1, to=1-2]
    \arrow["{p_{\uE}}", from=1-1, to=2-1]
    \arrow["{p_{f^*\uE}}", from=1-2, to=2-2]
    \arrow["{s_{\uE}}", shift left=3, from=2-1, to=1-1]
    \arrow["{f^*}", from=2-1, to=2-2]
    \arrow["{s_{f^*\uE}}", shift left=3, from=2-2, to=1-2]
  \end{tikzcd}\]
  Let $w\in H^1(X,\uE)$ such that $f^*w=0$. Then by the diagram above, $f^*s_{\uE}(w)=0$.
  So the $1$-form $s_{\uE}(w)$ vanishes on $U\subset Y$ (regular) dense open subset.
  Let $\gamma$ be a loop in $Y$ which meets $Y\setminus U$ at finitely many points (i.e. a collection of paths on local étale charts).
  Then $\int_{\tilde{\gamma}}s_{\uE}(w)=0$.
  
  However, for $y\in U$, such loops generate $\pi_1(Y,y)$.
  Indeed, $Y$ is étale locally an affine scheme $V$ over $\C$. Such schemes have Whitney stratification,
  and it can be arranged that the stratification extends to a compactification $\overline{V}$ of $V$
  such that $V$ and $V\cap (Y\setminus U)$ are unions of strata \cite[Th. 2.2]{verdierStratificationWhitneyTheoreme1976}.
  A compact Whitney stratified space is also a stratified space \cite[Th. 2.B.1]{thomEnsemblesMorphismesStratifies1969},
  so $\overline{V}$ is a stratified space. Hence $V$ is a stratified space, and any path in a stratified space is homotopic
  to a path meeting non-maximal strata at only finitely many points.

  So, for any $\gamma$ in $\pi_1(Y,y)$, $\int_{\tilde{\gamma}}s_{\uE}(w)=0$, i.e., $s_{\uE}(w)$ is killed by the map
  $Z^1_{dR}(X,\uE)\ra Z^1(\pi_1(X,y),\rho) \ra Z^1(\pi_1(Y,y),\rho)$, where $\rho$ is the monodromy representation of $\uE$ at $y$.
  By \Cref{propositionnonabHurwitz}, this implies that $i^*w=i^*[s_{\uE}(w)]=0$. Hence the desired coincidence of kernels.

  If $\pi_1(Y)\ra\pi_1(X)$ is surjective, then by inflation-restriction \cite[VII, Prop. 4]{serreLocalFields2013}
  $H^1(\pi_1(X,y),\rho) \ra H^1(\pi_1(Y,y),\rho)$ is injective. Hence, by \Cref{propositionnonabHurwitz},
  $H^1(X,\uE)\ra H^1(Y,i^*\uE)$ is injective and, by the above, so is $H^1(X,\uE)\ra H^1(\hat{Y},f^*\uE)$.
\end{proof}


\subsection{The case of some subvarieties of moduli spaces and application to mapping class groups}\label{subcorollaryinjectivity}


In this section we apply \Cref{lemmainjectivity} to some maps between twisted moduli spaces.

\begin{corollary}\label{corollaryinjectivity}
  Let $r\geq 1$ and let $f:\hat{Y}\ra X$ be a map in one of the following collections.
  \begin{enumerate}
    \item $\Mgrb{g-h}{n+1}{r}\lra \Mgrb{g}{n}{r}$, $h\geq 1$, $g\geq 2h+1$, $n\geq 0$.\\
    The map is induced by a choice of point in $\Mgrb{h}{1}{r}$ and the gluing map $\Mgrb{g-h}{n+1}{r}\times\Mgrb{h}{1}{r} \lra \Mgrb{g}{n}{r}$.
    \item $\Mgrb{g}{n+1}{r}\sqcup\Mgrb{g}{n+1}{r}\lra \Mgrb{g}{n+2}{r}$, $g\geq 0$, $n\geq 2$, $2g-2+n>0$.\\
    The first map $\Mgrb{g}{n+1}{r}\ra \Mgrb{g}{n+2}{r}$ has image the divisor where the point marked $1$ meets the point marked $n+1$,
    while the second has image the divisor where the point marked $2$ meets the point marked $n+2$.
    \item $\hat{\partial}\Mgrb{g}{n}{r}\lra\Mgrb{g}{n}{r}$, $2g-2+n>0$ and $\dim \Mgrb{g}{n}{r}\geq 2$.\\
    Here, $\hat{\partial}\Mgrb{g}{n}{r}=\bigsqcup_i\hat{D_i}$ is the normalization of the boundary normal crossing divisor $\partial\Mgrb{g}{n}{r}=\bigcup_iD_i$.
  \end{enumerate}

  Then for any semisimple local system $\uE$ on $X$, $H^1(X,\uE)\ra H^1(\hat{Y},f^*\uE)$ is injective.
\end{corollary}

\begin{proof}
  To apply \Cref{lemmainjectivity} to $f:\hat{Y}\ra X$, we need to check that $\hat{Y}$ and $X$ are Kähler,
  that $f(\hat{Y})$ is connected, and that $\pi_1(f(\hat{Y}))\ra\pi_1(X)$ is surjective. The twisted moduli spaces are Kähler
  (see the discussion in the proof of \Cref{theoremlifting}).
    
    \textbf{(1)} $\hat{Y}$ is connected, hence $f(\hat{Y})$ is also. Let $*\in \Mgrb{h}{1}{r}$ be the point chosen to define $f$. Then the image of 
    $\{*\}\times \Mgrb{g-2h}{n+2}{r}\times \{*\}\ra \Mgrb{g}{n}{r}$ is covered twice by $f$, in the sense that $f(\hat{Y})$
    is both the image of $\{*\}\times \Mgrb{g-h}{n+2}{r}\ra \Mgrb{g}{n}{r}$ and the image of $\Mgrb{g-2h}{n+2}{r}\times \{*\}\ra \Mgrb{g}{n}{r}$.
    The fundamental group $\pi_1(\Mgrb{g}{n}{r})$ is a quotient of $\Modu{S_g^n}$ and from the above we see that the image of $\pi_1(f(\hat{Y}))$
    in $\pi_1(\Mgrb{g}{n}{r})$ contains the images of both
    $\Modu{S'}$ and $\Modu{S''}$, with $S'$ and $S''$ as below.
    
    \begin{center}
    \def\svgwidth{0.9\linewidth}
\begingroup%
  \makeatletter%
  \providecommand\color[2][]{%
    \errmessage{(Inkscape) Color is used for the text in Inkscape, but the package 'color.sty' is not loaded}%
    \renewcommand\color[2][]{}%
  }%
  \providecommand\transparent[1]{%
    \errmessage{(Inkscape) Transparency is used (non-zero) for the text in Inkscape, but the package 'transparent.sty' is not loaded}%
    \renewcommand\transparent[1]{}%
  }%
  \providecommand\rotatebox[2]{#2}%
  \newcommand*\fsize{\dimexpr\f@size pt\relax}%
  \newcommand*\lineheight[1]{\fontsize{\fsize}{#1\fsize}\selectfont}%
  \ifx\svgwidth\undefined%
    \setlength{\unitlength}{566.51000977bp}%
    \ifx\svgscale\undefined%
      \relax%
    \else%
      \setlength{\unitlength}{\unitlength * \real{\svgscale}}%
    \fi%
  \else%
    \setlength{\unitlength}{\svgwidth}%
  \fi%
  \global\let\svgwidth\undefined%
  \global\let\svgscale\undefined%
  \makeatother%
  \begin{picture}(1,0.3873027)%
    \lineheight{1}%
    \setlength\tabcolsep{0pt}%
    \put(0,0){\includegraphics[width=\unitlength,page=1]{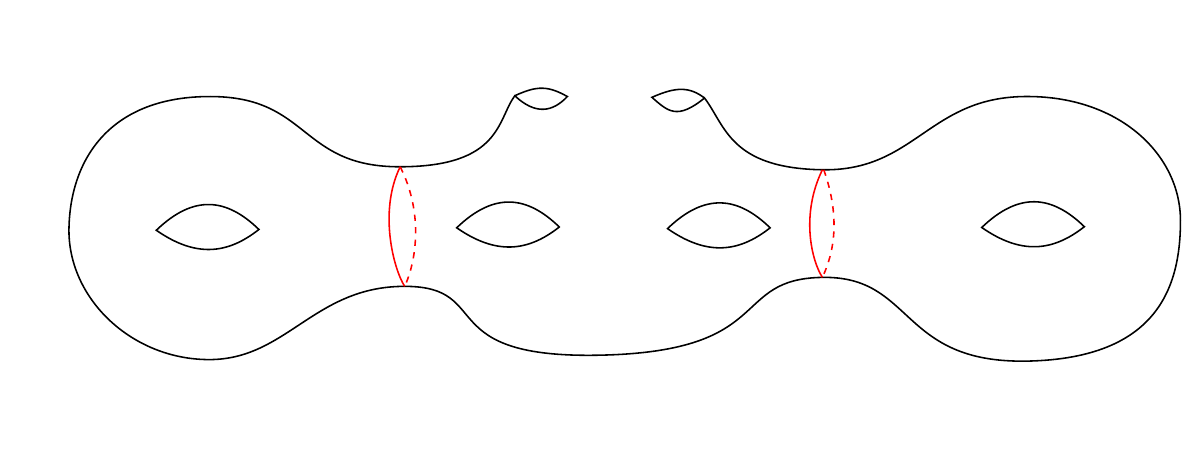}}%
    \put(-0.00120237,0.19594561){\color[rgb]{0,0,0}\makebox(0,0)[lt]{\lineheight{1.25}\smash{\begin{tabular}[t]{l}$S_g^n$\end{tabular}}}}%
    \put(0,0){\includegraphics[width=\unitlength,page=2]{pi1surj_genus.pdf}}%
    \put(0.37059893,0.37507735){\color[rgb]{0,0,0}\makebox(0,0)[lt]{\lineheight{1.25}\smash{\begin{tabular}[t]{l}$S'$\end{tabular}}}}%
    \put(0.65590667,0.00280033){\color[rgb]{0,0,0}\makebox(0,0)[lt]{\lineheight{1.25}\smash{\begin{tabular}[t]{l}$S''$\end{tabular}}}}%
    \put(0,0){\includegraphics[width=\unitlength,page=3]{pi1surj_genus.pdf}}%
  \end{picture}%
\endgroup%

    \end{center}
    
    As $g\geq 3$, we see using Likorish generators \cite[4.13]{farbPrimerMappingClass2011}
    that $\Modu{S'}$ and $\Modu{S''}$ together generate $\Modu{S_g^n}$. Hence $\pi_1(f(\hat{Y}))\ra\pi_1(X)$ is surjective.
    
    \textbf{(2)} The images of the $2$ maps are the divisors where $1$ meets $n+1$ and $2$ meets $n+2$. These $2$ divisors intersect,
    hence $f(\hat{Y})$ is connected. As for the statement about fundamental groups, by an argument similar to that of case (1),
    we are reduced to show that the images of $\Modu{S'}$ and $\Modu{S''}$ together generate $\Modu{S_g^n}$, where $S'$ and $S''$ are as below.

    \begin{center}
    \def\svgwidth{0.8\linewidth}
\begingroup%
  \makeatletter%
  \providecommand\color[2][]{%
    \errmessage{(Inkscape) Color is used for the text in Inkscape, but the package 'color.sty' is not loaded}%
    \renewcommand\color[2][]{}%
  }%
  \providecommand\transparent[1]{%
    \errmessage{(Inkscape) Transparency is used (non-zero) for the text in Inkscape, but the package 'transparent.sty' is not loaded}%
    \renewcommand\transparent[1]{}%
  }%
  \providecommand\rotatebox[2]{#2}%
  \newcommand*\fsize{\dimexpr\f@size pt\relax}%
  \newcommand*\lineheight[1]{\fontsize{\fsize}{#1\fsize}\selectfont}%
  \ifx\svgwidth\undefined%
    \setlength{\unitlength}{345.69376788bp}%
    \ifx\svgscale\undefined%
      \relax%
    \else%
      \setlength{\unitlength}{\unitlength * \real{\svgscale}}%
    \fi%
  \else%
    \setlength{\unitlength}{\svgwidth}%
  \fi%
  \global\let\svgwidth\undefined%
  \global\let\svgscale\undefined%
  \makeatother%
  \begin{picture}(1,0.50452413)%
    \lineheight{1}%
    \setlength\tabcolsep{0pt}%
    \put(0,0){\includegraphics[width=\unitlength,page=1]{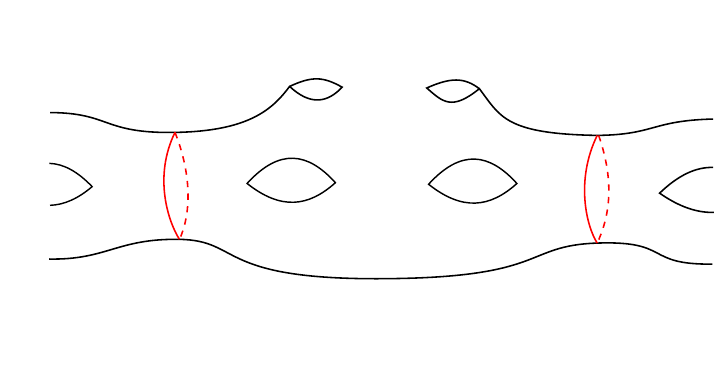}}%
    \put(-0.00197039,0.2386654){\color[rgb]{0,0,0}\makebox(0,0)[lt]{\lineheight{1.25}\smash{\begin{tabular}[t]{l}$S_g^n$\end{tabular}}}}%
    \put(0,0){\includegraphics[width=\unitlength,page=2]{pi1surj_boundary.pdf}}%
    \put(0.4294205,0.48448969){\color[rgb]{0,0,0}\makebox(0,0)[lt]{\lineheight{1.25}\smash{\begin{tabular}[t]{l}$S'$\end{tabular}}}}%
    \put(0.6062532,0.00458913){\color[rgb]{0,0,0}\makebox(0,0)[lt]{\lineheight{1.25}\smash{\begin{tabular}[t]{l}$S''$\end{tabular}}}}%
    \put(0,0){\includegraphics[width=\unitlength,page=3]{pi1surj_boundary.pdf}}%
  \end{picture}%
\endgroup%

    \end{center}

    Again, this can be checked using generators (see \cite[4.4.4]{farbPrimerMappingClass2011}).
    
    \textbf{(3)} For this case, we will use connectivity of the Harer complex of curves \cite{harerVirtualCohomologicalDimension1986},
    following Boggi's paper \cite[§3]{boggiProfiniteTeichmuellerTheory2006}.\footnote{Although \cite{boggiProfiniteTeichmuellerTheory2006}
    is known to contain a mistake in the later sections, the content of section 3 is correct.}
    By the root stack construction and excision, we have
    \begin{align*}
      \pi_1(\Mgrp{g}{n}{r},\partial\Mgrp{g}{n}{r})&\simeq\pi_1(\oMg{g}{n},\partial\oMg{g}{n})
                                                  \simeq\pi_1(\hatMg{g}{n},\partial\hatMg{g}{n})\\
                                                  &\simeq\pi_1(\widehat{T}_{g,n},\partial \widehat{T}_{g,n})
                                                  \simeq\pi_{0}(\partial \widehat{T}_{g,n})
    \end{align*}
    where $\hatMg{g}{n}$ is the real blow up of $\oMg{g}{n}$ along the components of the boundary,
    $\widehat{T}_{g,n}$ is the universal cover of $\hatMg{g}{n}$, which is homotopy equivalent to Teichmüller space, hence contractible,
    and $\partial \widehat{T}_{g,n}$ is the pullback of $\partial\hatMg{g}{n}$ to $\widehat{T}_{g,n}$
    and has the homotopy type of the Harer curve complex (see \cite[§3]{boggiProfiniteTeichmuellerTheory2006}).
    Now, the Harer complex of curves is connected whenever $\dim \Mgrb{g}{n}{r}\geq 2$
    (see \cite[Th. 4.1]{harerVirtualCohomologicalDimension1986} or \cite[Th. 4.3]{farbPrimerMappingClass2011}).
    Hence in these cases, $\pi_1(\Mgrp{g}{n}{r},\partial\Mgrp{g}{n}{r})$ is trivial, so that by the long exact sequence of homotopy groups
    $\partial\Mgrp{g}{n}{r}$ is connected and $\pi_1(\partial\Mgrp{g}{n}{r})\ra \pi_1(\Mgrp{g}{n}{r})$ is surjective.
    Now $\Mgrb{g}{n}{r}$ is a $\mu_r^n$-gerbe over $\Mgrp{g}{n}{r}$, so the same results hold for $\Mgrb{g}{n}{r}$.
\end{proof}


\subsection{Digression: application to the Ivanov conjecture and other rigidity conjectures about mapping class groups}\label{subdigressionIvanov}


In this subsection, we explain how \Cref{corollaryinjectivity} can be applied to relate the Ivanov conjectures about mapping class groups
$\Modu{S_g^n}$, $g\geq 3$, and some of its generalizations, for different values of $g$ and $n$.
In particular, this reduces, for each of these conjectures, the cases $g\geq 3$, $n\geq 0$ to the single case $g=3$, $n=3$.

\begin{conjecture}[{Ivanov conjecture, \cite[7.]{ivanovFifteenProblemsMapping2006}}]\label{conjectureIvanov}
  Let $g\geq 3$, $n\geq 0$, and let $\Gamma\subset\Modu{S_g^n}$ be a finite index subgroup.
  Then $\Gamma^\mathrm{ab}=H^1(\Gamma,\Z)$ is finite.
\end{conjecture}

\begin{notation}
  We will denote by $(\mathrm{I})_g^n$ the statement of the Ivanov conjecture for $\Modu{S_g^n}$.
\end{notation}

This conjecture has been extensively studied.
For example, in 2017, Ershov-He proved that the finite-index subgroups of $\Modu{S_g^1}$, $g\geq 12$,
that contain a term of the Johnson filtration have a finite abelianization \cite{ershovFinitenessPropertiesJohnson2018}.
Their result extends those of McCarthy \cite{mccarthyFirstCohomologyGroup2001}, Hain \cite{hainTorelliGroupsGeometry1996},
and Putman \cite{putmanNoteAbelianizationsFiniteindex2009} on this question. In 2011, Putman and Wieland \cite{putmanAbelianQuotientsSubgroups2013} showed that Ivanov's conjecture
is essentially equivalent to another conjecture, now called the \emph{Putman-Wieland conjecture}, which concerns the higher Prym representations
of mapping class groups.

Given $\Gamma$ as in the conjecture, $H^1(\Gamma,\Z)$ is finite if and only if $H^1(\Gamma,\C)=0$,
which is in turn equivalent to $H^1(\Modu{S_g^n},\C[\Modu{S_g^n}/\Gamma])=0$. Now notice that the representations
that are direct summands of some $\C[\Modu{S_g^n}/\Gamma]^{\oplus m}$ for some $\Gamma$ of finite index and some $m\geq 0$
are exactly the finite image representations of $\Modu{S_g^n}$. Hence the Ivanov conjecture for $\Modu{S_g^n}$
is equivalent to the following.

\begin{conjecture*}[{\ref{conjectureIvanov}}, Ivanov conjecture, reformulation]
  Let $g\geq 3$, $n\geq 0$, then for any \emph{finite image} complex linear representation $\rho$ of $\Modu{S_g^n}$, $H^1(\Modu{S_g^n},\rho)=0$.
\end{conjecture*}

This form of the Ivanov conjecture makes clear its relation to the property (T) conjecture for $\Modu{S_g^n}$, $g\geq 3$, $n\geq 0$
(question 8 in \cite{ivanovFifteenProblemsMapping2006}).
This conjecture says that for any unitary representation $\rho$ of $\Modu{S_g^n}$ in a Hilbert space, $H^1(\Modu{S_g^n},\rho)=0$.
As any finite image representation is unitary, so the property (T) conjecture would imply the Ivanov conjecture.
It would also imply rigidity of all quantum representations in genus at least $3$ which are unitary up to Galois conjugation.
This includes $\SO$ and $\SU$ quantum representations (see \cite[5.1.2, 5.1.3]{deroinToledoInvariantsTopological2022} for example),
but also quantum representations coming from the modular functors associated to simple Lie algebras
(see \cite[3.3]{bakalovLecturesTensorCategories2000} and \cite{kirillovInnerProductModular1996,kirillovInnerProductModular1998,wenzlTensorCategoriesQuantum1998}).

This motivates the statement of an intermediate ``finite dimensional property (T)'' conjecture that would cover
these cases. 

\begin{conjecture}\label{conjecturefT}
  Let $g\geq 3$, $n\geq 0$, then for any \emph{finite dimensional} unitary representation $\rho$ of $\Modu{S_g^n}$, $H^1(\Modu{S_g^n},\rho)=0$.
\end{conjecture}

\begin{notation}
  We will denote by $(\mathrm{fT})_g^n$ the above statement of this ``finite dimensional property (T)'' conjecture for $\Modu{S_g^n}$.
\end{notation}

This last conjecture is related, in the case of adjoints of irreducible representations, to the general conjecture \cite[Conj. 6.1.2]{littMotivesMappingClass2024}
of Daniel Litt, which says that irreducible representations of $\Modu{S_{g,n}}$ for $g\geq 3$ should be rigid.
Motivated by this conjecture of Litt and \Cref{propositionrestrictions} below, we now state a last conjecture,
about representations sending Dehn twists to finite order elements, which we will call \emph{finite Dehn twist representation}.
Note that in genus at least $3$, this is equivalent to the Dehn twists being sent to diagonalizable elements,
as then linear representations of mapping class groups send Dehn twists to quasi-unipotent elements \cite[2.5]{aramayonaRigidityPhenomenaMapping2016}.
In particular, for $g\geq 3$ unitary representations are finite Dehn twist representations.

\begin{conjecture}\label{conjecturefD}
  For any $g\geq 3$, $n\geq 0$ and semisimple finite Dehn twist representation $\rho$ of $\Modu{S_g^n}$, $H^1(\Modu{S_{g,n}},\rho)=0$.
\end{conjecture}

\begin{notation}
  We will denote by $(\mathrm{fD})_g^n$ the above statement of this conjecture for $\Modu{S_g^n}$.
\end{notation}

\begin{remark}
  As mentioned above, for any $g\geq 3$, $n\geq 0$, $(\mathrm{fD})_g^n\Rightarrow (\mathrm{fT})_g^n \Rightarrow (\mathrm{I})_g^n$.
  Note also that $(\mathrm{fD})_g^n$ would imply that any finite Dehn twist representation of $\Modu{S_g^n}$ is semisimple,
  as extensions of semisimple representations are classified by $H^1$ groups which would then vanish.
\end{remark}

To be able to apply \Cref{corollaryinjectivity} to the $3$ conjectures above, we need to compare, for $\rho$ a representation of $\Modl{r}{S_g^n}$,
the cohomology groups $H^1(\Modl{r}{S_g^n},\rho)$ and $H^1(\Modu{S_g^n},\rho)$. This is done in the Proposition below.
For $\rho$ of finite image, the Proposition was observed by Putman in \cite[Theorem A]{putmanNoteAbelianizationsFiniteindex2009}.
The intuition behind the equality of the $2$ cohomology groups is more clear when one considers coefficients in $\ad\rho$.
Then the groups classify deformations, and representations of $\Modu{S_g^n}$ for $g\geq 3$ always send Dehn twists to quasi-unipotent matrices,
see \cite[2.5]{aramayonaRigidityPhenomenaMapping2016}.

\begin{proposition}\label{propositionrirreleventinH1}
  Let $g\geq 3$, $n\geq 0$, $r\geq 1$, and let $\rho:\Modl{r}{S_g^n}\ra \GLn{d}{\C}$ be a representation.
  Then $H^1(\Modl{r}{S_g^n},\rho)=H^1(\Modu{S_g^n},\rho)$.
\end{proposition}
\begin{proof}
  By inflation-restriction \cite[VII, Prop. 4]{serreLocalFields2013}, we
  need only show that for any $\gamma$ simple closed curve on $S_g^n$ with corresponding Dehn twist $T_\gamma$,
  the map $H^1(\Modu{S_g^n},\rho)\ra H^1(T_\gamma^\Z,\rho)$ vanishes, where $T_\gamma^\Z$ is the subgroup generated by $T_\gamma$.
  Let $S_\gamma$ be a connected component of the compact surface obtained
  by cutting $S_g^n$ along $\gamma$. When $\gamma$ is separating, we may assume that we chose the component
  $S_\gamma$ so that it has genus at least $2$.
  
  The map $\Modl{r}{S_\gamma}\ra\Modl{r}{S_g^n}$ is represented by an algebraic map between $r$-twisted moduli space,
  and applying \Cref{remarkpullbacksemisimple} to this map and $\rho$, we see that the restriction of $\rho$ to $\Modu{S_\gamma}$ is semisimple.
  Hence the subspace $\rho^\mathrm{triv}\subset \rho$ of $\Modu{S_\gamma}$-fixed points
  admits a direct complement $\rho'\subset\rho$ for the action of $\Modu{S_\gamma}$.
  Then, as $S_\gamma$ has genus at least $2$, $\Modu{S_\gamma}$ has finite abelianization \cite[5.1.2, 5.1.3]{farbPrimerMappingClass2011},
  and hence $H^1(\Modu{S_\gamma},\rho^\mathrm{triv})=0$.
  Now $H^1(\Modu{S_g^n},\rho)\ra H^1(T_\gamma^\Z,\rho^\mathrm{triv})$ factors as
  \begin{equation*}
    H^1(\Modu{S_g^n},\rho)\ra H^1(\Modu{S_\gamma},\rho^\mathrm{triv})\ra H^1(T_\gamma^\Z,\rho^\mathrm{triv})
  \end{equation*}
  and thus vanishes. Now let $\rho'_1\subset \rho'$ be the subspace on which $T_\gamma$ acts trivially.
  As $T_\gamma$ is central in the image of $\Modu{S_\gamma}\ra\Modu{S_g^n}$ and acts semisimply,
  $\rho'_1$ is stabilized by $\Modu{S_\gamma}$ and there exists a retraction $\rho'\ra \rho'_1$ of $\Modu{S_\gamma}$-representations.
  Then $H^1(T_\gamma^\Z,\rho')=H^1(T_\gamma^\Z,\rho'_1)$,
  and $H^1(\Modu{S_g^n},\rho)\ra H^1(T_\gamma^\Z,\rho'_1)$ factors as
  \begin{multline*}
    H^1(\Modu{S_g^n},\rho)\ra H^1(\Modu{S_\gamma}\times T_\gamma^\Z,\rho'_1)\\\ra H^0(\Modu{S_\gamma},\rho'_1)\otimes H^1(T_\gamma^\Z,\C)\ra \rho_1'\otimes H^1(T_\gamma^\Z,\C)
  \end{multline*}
  which vanishes as $H^0(\Modu{S_\gamma},\rho')=0$. Hence $H^1(\Modu{S_g^n},\rho)\ra H^1(T_\gamma^\Z,\rho)$ vanishes.
\end{proof}

We now state the consequence of \Cref{corollaryinjectivity} which will be applied to the conjectures.

\begin{proposition}\label{propositionrestrictions}
  Let $g,n\geq 0$ with $2g-2+n>0$ and let $r\geq 1$. Let $\rho$ be a semisimple representation of $\Modl{r}{S_g^n}$.
  If for every embedding $i$ of $S=S_0^4$ or of $S=S_1^1$ into $S_g^n$, $H^1(\Modl{r}{S},i^*\rho)=0$,
  then $H^1(\Modl{r}{S_g^n},\rho)=0$.

  Assume $g,n\geq 3$ and choose $g\geq g'\geq 3$, $n+g-g'\geq n'\geq 3$. If for every embedding $i$ of $S_{g'}^{n'}$ into $S_g^n$,
  $H^1(\Modl{r}{S_{g'}^{n'}},i^*\rho)=0$, then $H^1(\Modl{r}{S_g^n},\rho)=0$.
\end{proposition}

\begin{proof}
  Let us prove the first statement. Note that by \Cref{propositionnonabHurwitz}, $H^1(\Modl{r}{S_{g,n}},\rho)=H^1(\Mgrb{g}{n}{r},\uE)$
  for some semisimple local system $\uE$. As $\Mgrb{1}{1}{r}$ and $\Mgrb{0}{4}{r}$ are the only DM moduli spaces of dimension $1$,
  by repeated applications of case (3) of \Cref{corollaryinjectivity},
  we see that $H^1(\Mgrb{g}{n}{r},\uE)$ injects into a direct sum of terms of the form $H^1(\Mgrb{1}{1}{r},j^*\uE)$
  for some embedding $j:\Mgrb{1}{1}{r}\ra \Mgrb{g}{n}{r}$ or of the form $H^1(\Mgrb{0}{4}{r},j^*\uE)$
  for some embedding $j:\Mgrb{0}{4}{r}\ra \Mgrb{g}{n}{r}$. At the level of fundamental groups,
  these embeddings $j$ correspond to embeddings $i: S\ra S_g^n$ for $S=S_1^1$ or $S=S_0^4$.

  The second case is obtained similarly by repeated applications of case (1) of \Cref{corollaryinjectivity} ($g-g'$ times, to reduce the genus to $g'$)
  and then of case (2) of \Cref{corollaryinjectivity} ($n+g-g'-n'$ times, to reduce the number of markings to $n'$).
\end{proof}

\begin{remark}
  Note that in the second case of the proof, not all embeddings of $S_{g'}^{n'}$ into $S_g^n$ are used. We will not need this fact.
\end{remark}

\begin{corollary}\label{corollaryconjecturesreducetothreethree}
  Let $(\mathrm{X})$ denote one of $(\mathrm{I})$ (Ivanov conjecture \Cref{conjectureIvanov}),
  $(\mathrm{fT})$ (finite dimensional property (T) conjecture \ref{conjecturefT})
  or $(\mathrm{fD})$ (conjecture \ref{conjecturefD} on finite Dehn twist representations). Then
  \begin{itemize}
    \item $\forall g\geq 3,\forall n\geq 0$, $(\mathrm{X})_{g,n+1}\Rightarrow (\mathrm{X})_{g+1,n}$;
    \item $\forall g\geq 3,\forall n\geq 0$, $(\mathrm{X})_{g,n+1}\Rightarrow (\mathrm{X})_{g,n}$;
    \item $\forall g\geq 3,\forall n\geq 3$, $(\mathrm{X})_{g,n}\Rightarrow (\mathrm{X})_{g,n+1}$.
  \end{itemize}
  In particular, if $(\mathrm{X})_{3,3}$ holds, then $(\mathrm{X})_{g,n}$ holds for all $g\geq 3$, $n\geq 0$.
\end{corollary}

\begin{proof}
  The implications $(\mathrm{X})_{g,n+1}\Rightarrow (\mathrm{X})_{g+1,n}$ and $(\mathrm{X})_{g,n}\Rightarrow (\mathrm{X})_{g,n+1}$ follow from the second statement
  of \Cref{propositionrestrictions} and the fact that the finite image/ unitarity/ semisimple finite Dehn twist properties are stable
  by pullback under maps between mapping class groups induced by embeddings of surfaces (for semisimplicity, see \Cref{remarkpullbacksemisimple} and the proof
  of \Cref{propositionrirreleventinH1}).
  The implication $(\mathrm{X})_{g,n+1}\Rightarrow (\mathrm{X})_{g,n}$ follows from the Birman exact sequence and inflation restriction
  (see the proof of \cite[Lem. 7.1]{godfardRigidityFibonacciRepresentations2025}).
\end{proof}

\begin{remark}\label{remarkPWpaper}
  This result is not surprising to experts, although, to our knowledge, it has not been written down before.
  In the case where $(\mathrm{X})=(\mathrm{I})$,
  implications\footnote{The $2$ is not a typographic mistake.} $(\mathrm{I})_{g,n+2}\Rightarrow (\mathrm{I})_{g+1,n}$ and $(\mathrm{I})_{g,n+1}\Rightarrow (\mathrm{I})_{g,n}$
  are proven in \cite[step 2 of Lem. 3.2, Lem. 3.1]{putmanAbelianQuotientsSubgroups2013}.
  Their proof of $(\mathrm{I})_{g,n+2}\Rightarrow (\mathrm{I})_{g+1,n}$ goes through the work \cite[Lem. 2.6]{boggiProfiniteTeichmuellerTheory2006} of Boggi,
  in which the crucial argument is provided by mixed Hodge theory.
\end{remark}


\subsection{Rigidity of Fibonacci representations, revisited}


To conclude this section, and this article, we deduce directly from \Cref{propositionrestrictions} a quick proof of rigidity of Fibonacci representations
of mapping class groups, which was the subject of the article \cite[Th. 3.3]{godfardRigidityFibonacciRepresentations2025}.
The general idea of the proof remains the same: reduce rigidity for all $g,n\geq 0$ to a finite number of cases in genus $0$ and $1$.
However, the application of \Cref{corollaryinjectivity}, which generalizes \cite[Lem. 6.1]{godfardRigidityFibonacciRepresentations2025}
and its other variants used in that paper, substantially simplifies the proof. As a bonus, we extend the result to the $\SU$ case in level $10$.

As mentioned in \cite[Rmk. 1.2]{godfardRigidityFibonacciRepresentations2025}, there is little hope to generalize
this quick proof to some $\ell>5$, as in these cases, rigidity fails for $(g,n)=(0,4)$.

\begin{corollary}[{Rigidity of Fibonacci representations}]
  Let $\Nu=\Nu_5$ (level $10$ $\SU$ modular functor) or $\Nu=\Nu_5^\mathrm{ev}$ (level $5$ $\SO$ modular functor).
  Then for any $g,n\geq 0$ and $\ul\in\Lambda^n$, the $\Modl{5}{S_g^n}$ projective representation on $\Nu_g(\ul)$ is cohomologically rigid.
  For $g\geq 3$, it is also rigid as a representation of $\Modu{S_g^n}$.
\end{corollary}

\begin{proof}
  Fix $g,n\geq 0$ and $\ul\in\Lambda^n$.
  Applying the first statement of \Cref{propositionrestrictions}, we see that we need only prove that the restriction
  of $\ad\Nu_g(\ul)$ to embedded $S_0^4$ and $S_1^1$ inside $S_g^n$ has no $H^1$.
  However, by the gluing axioms and \Cref{corollarycenterkillsandKunneth},
  these pullbacks have direct sum decompositions, whose only summands with potentially non-zero $H^1$
  are of the form $\ad\Nu_0(\umu)$, $\umu\in\Lambda^4$ or $\ad\Nu_1(\mu)$, $\mu\in\Lambda$.
  Hence we are reduced to the cases $(g,n)=(0,4)$ and $(g,n)=(1,1)$.
  The group $\Modl{5}{S_{1,1}}$ is the fundamental group of $\Mgrp{1}{1}{5}$, hence isomorphic to the triangular group $T(2,3,5)$,
  which is spherical ($\frac 1 2 + \frac 1 3 +\frac  1 5>1$) and hence finite. Hence so is $\Modl{5}{S_1^1}$, and all of its representations are rigid.
  Let us turn to $\Nu_0(\umu)$, $\umu\in\Lambda^4$. The representations $\Nu_0(\mu_1,\mu_2,\mu_3,\mu_4)$ and $\Nu_0(3-\mu_1,3-\mu_2,\mu_3,\mu_4)$ are isomorphic
  (special case of a general fact about $\SO$ modular functors, see \cite[4.13-4.16]{blanchetTopologicalQuantumField1995}, in genus $0$,
  level $r$ $\SO$ and level $2r$ $\SU$ theories match, see \cite[Th. 1.5]{blanchetTopologicalQuantumField1995}).
  Hence we can assume that the $\mu_i$ are all $0$ or $2$.
  If one of the $\mu_i$ is $0$, then by \Cref{propositionvacuum}, the $H^1$ vanishes.
  We are left with the case of $\Nu_0(2,2,2,2)$, which is $2$-dimensional and can be treated by hand,
  see \cite[Lem. 7.3]{godfardRigidityFibonacciRepresentations2025}.
  Alternatively, one can see that $\Modl{5}{S_{0,4}}$ is isomorphic to the triangular group $T(5,5,5)$
  and that the representation $\Nu_0(2,2,2,2)$ is valued, up to Galois conjugation, in $\PU{1}{1}$,
  and conclude by using that hyperbolic representations of triangle groups are rigid.\footnote{This argument for $\Nu_0(2,2,2,2)$ is due to Julien Marché.}
  
  The second statement follows from the first and \Cref{propositionrirreleventinH1}.
\end{proof}


\bibliographystyle{plain}
\bibliography{biblio}

@article{abramovichCompactifyingSpaceStable2002,
  title = {Compactifying the Space of Stable Maps},
  author = {Abramovich, Dan and Vistoli, Angelo},
  year = 2002,
  month = jan,
  journal = {Journal of the American Mathematical Society},
  volume = {15},
  number = {1},
  pages = {27--75},
  issn = {0894-0347, 1088-6834},
  doi = {10.1090/S0894-0347-01-00380-0},
  urldate = {2025-01-27},
  langid = {english}
}

@incollection{aramayonaRigidityPhenomenaMapping2016,
  title = {Rigidity Phenomena in the Mapping Class Group},
  booktitle = {Handbook of {{Teichm\"uller Theory}}},
  author = {Aramayona, Javier and Souto, Juan},
  editor = {Papadopoulos, A.},
  year = 2016,
  month = may,
  volume = {VI},
  pages = {131--165},
  publisher = {European Mathematical Society Publishing House},
  address = {Berlin},
  issn = {2523-5133, 2523-5141},
  doi = {10.4171/161},
  urldate = {2023-01-05},
  isbn = {978-3-03719-161-3 978-3-03719-661-8},
  langid = {english}
}

@book{bakalovLecturesTensorCategories2000,
  title = {Lectures on {{Tensor Categories}} and {{Modular Functors}}},
  author = {Bakalov, Bojko and Kirillov, Alexander},
  year = 2000,
  month = nov,
  series = {University {{Lecture Series}}},
  volume = {21},
  publisher = {American Mathematical Society},
  address = {Providence, Rhode Island},
  doi = {10.1090/ulect/021},
  urldate = {2024-01-03},
  isbn = {978-0-8218-2686-7 978-1-4704-2168-7},
  langid = {english}
}

@article{blanchetThreemanifoldInvariantsDerived1992,
  title = {Three-Manifold Invariants Derived from the {{Kauffman}} Bracket},
  author = {Blanchet, C. and Habegger, N. and Masbaum, G. and Vogel, P.},
  year = 1992,
  month = oct,
  journal = {Topology},
  volume = {31},
  number = {4},
  pages = {685--699},
  issn = {0040-9383},
  doi = {10.1016/0040-9383(92)90002-Y},
  urldate = {2023-02-18},
  langid = {english}
}

@article{blanchetTopologicalQuantumField1995,
  title = {Topological {{Quantum Field Theories}} Derived from the {{Kauffman}} Bracket},
  author = {Blanchet, C. and Habegger, N. and Masbaum, G. and Vogel, P.},
  year = 1995,
  month = oct,
  journal = {Topology},
  volume = {34},
  number = {4},
  pages = {883--927},
  issn = {0040-9383},
  doi = {10.1016/0040-9383(94)00051-4},
  urldate = {2022-12-28},
  langid = {english}
}

@article{boggiProfiniteTeichmuellerTheory2006,
  title = {{Profinite Teichm\"uller theory}},
  author = {Boggi, Marco},
  year = 2006,
  journal = {Mathematische Nachrichten},
  volume = {279},
  number = {9-10},
  pages = {953--987},
  issn = {1522-2616},
  doi = {10.1002/mana.200510405},
  urldate = {2024-10-03},
  copyright = {Copyright \copyright{} 2006 WILEY-VCH Verlag GmbH \& Co. KGaA, Weinheim},
  langid = {ngerman}
}

@article{chiodoStableTwistedCurves2008,
  title = {Stable Twisted Curves and Their R-Spin Structures},
  author = {Chiodo, Alessandro},
  year = 2008,
  journal = {Annales de l'Institut Fourier},
  volume = {58},
  number = {5},
  pages = {1635--1689},
  issn = {1777-5310},
  doi = {10.5802/aif.2394},
  urldate = {2022-12-12},
  langid = {english}
}

@article{deligneTheorieHodgeIII1974,
  title = {Th\'eorie de {{Hodge}} : {{III}}},
  shorttitle = {Th\'eorie de {{Hodge}}},
  author = {Deligne, Pierre},
  year = 1974,
  journal = {Publications Math\'ematiques de l'IH\'ES},
  volume = {44},
  pages = {5--77},
  issn = {1618-1913},
  urldate = {2023-05-01},
  langid = {english}
}

@misc{deroinToledoInvariantsTopological2022,
  title = {Toledo Invariants of {{Topological Quantum Field Theories}}},
  author = {Deroin, Bertrand and March{\'e}, Julien},
  year = 2022,
  month = jul,
  number = {arXiv:2207.09952},
  eprint = {2207.09952},
  primaryclass = {math},
  publisher = {arXiv},
  doi = {10.48550/arXiv.2207.09952},
  urldate = {2023-01-04},
  archiveprefix = {arXiv}
}

@book{drutuGeometricGroupTheory2018,
  title = {Geometric {{Group Theory}}},
  author = {Dru{\c t}u, Cornelia and Kapovich, Michael},
  year = 2018,
  month = mar,
  series = {Colloquium {{Publications}}},
  volume = {63},
  publisher = {American Mathematical                     Society},
  address = {Providence, Rhode                     Island},
  doi = {10.1090/coll/063},
  urldate = {2025-11-04},
  isbn = {978-1-4704-1104-6 978-1-4704-4164-7},
  langid = {english}
}

@article{ershovFinitenessPropertiesJohnson2018,
  title = {On Finiteness Properties of the {{Johnson}} Filtrations},
  author = {Ershov, Mikhail and He, Sue},
  year = 2018,
  month = jun,
  journal = {Duke Mathematical Journal},
  volume = {167},
  number = {9},
  pages = {1713--1759},
  publisher = {Duke University Press},
  issn = {0012-7094, 1547-7398},
  doi = {10.1215/00127094-2018-0005},
  urldate = {2025-02-27}
}

@article{etingofFusionCategories2005,
  title = {On Fusion Categories},
  author = {Etingof, Pavel and Nikshych, Dmitri and Ostrik, Viktor},
  year = 2005,
  month = sep,
  journal = {Annals of Mathematics},
  volume = {162},
  number = {2},
  pages = {581--642},
  issn = {0003-486X},
  doi = {10.4007/annals.2005.162.581},
  urldate = {2023-12-29},
  langid = {english}
}

@misc{etingofRigidityNonnegligibleObjects2024,
  title = {Rigidity of Non-Negligible Objects of Moderate Growth in Braided Categories},
  author = {Etingof, Pavel and Penneys, David},
  year = 2024,
  month = dec,
  number = {arXiv:2412.17681},
  eprint = {2412.17681},
  primaryclass = {math},
  publisher = {arXiv},
  doi = {10.48550/arXiv.2412.17681},
  urldate = {2025-01-05},
  archiveprefix = {arXiv},
  langid = {english}
}

@book{etingofTensorCategories2015,
  title = {Tensor {{Categories}}},
  author = {Etingof, Pavel and Gelaki, Shlomo and Nikshych, Dmitri and Ostrik, Victor},
  year = 2015,
  month = jul,
  series = {Mathematical {{Surveys}} and {{Monographs}}},
  volume = {205},
  publisher = {American Mathematical Society},
  address = {Providence, Rhode Island},
  doi = {10.1090/surv/205},
  urldate = {2023-12-29},
  isbn = {978-1-4704-2024-6 978-1-4704-2349-0},
  langid = {english}
}

@misc{eyssidieuxOrbifoldKahlerGroups2021,
  title = {Orbifold {{K\"ahler Groups}} Related to {{Mapping Class}} Groups},
  author = {Eyssidieux, Philippe and Funar, Louis},
  year = 2021,
  month = dec,
  number = {arXiv:2112.06726},
  eprint = {2112.06726},
  primaryclass = {math},
  publisher = {arXiv},
  doi = {10.48550/arXiv.2112.06726},
  urldate = {2023-11-28},
  archiveprefix = {arXiv}
}

@book{farbPrimerMappingClass2011,
  title = {A {{Primer}} on {{Mapping Class Groups}} ({{PMS-49}})},
  author = {Farb, Benson and Margalit, Dan},
  year = 2011,
  month = sep,
  journal = {A Primer on Mapping Class Groups (PMS-49)},
  publisher = {Princeton University Press},
  doi = {10.1515/9781400839049},
  urldate = {2022-12-30},
  isbn = {978-1-4008-3904-9},
  langid = {english}
}

@misc{godfardHodgeStructuresConformal2025,
  title = {Hodge Structures on Conformal Blocks},
  author = {Godfard, Pierre},
  year = 2025,
  month = jul,
  number = {arXiv:2406.07459},
  eprint = {2406.07459},
  primaryclass = {math},
  publisher = {arXiv},
  doi = {10.48550/arXiv.2406.07459},
  urldate = {2025-08-19},
  archiveprefix = {arXiv},
  langid = {english}
}

@article{godfardRigidityFibonacciRepresentations2025,
  title = {Rigidity of {{Fibonacci}} Representations of Mapping Class Groups},
  author = {Godfard, Pierre},
  year = 2025,
  month = aug,
  journal = {Annales de l'Institut Fourier},
  volume = {75},
  number = {6},
  pages = {2529--2563},
  issn = {1777-5310},
  doi = {10.5802/aif.3676},
  urldate = {2025-08-19},
  langid = {english}
}

@misc{godfardSemisimplicityConformalBlocks2025a,
  title = {Semisimplicity of Conformal Blocks},
  author = {Godfard, Pierre},
  year = 2025,
  month = jul,
  number = {arXiv:2507.06318},
  eprint = {2507.06318},
  primaryclass = {math},
  publisher = {arXiv},
  doi = {10.48550/arXiv.2507.06318},
  urldate = {2025-08-22},
  archiveprefix = {arXiv}
}

@article{grossmanResidualFinitenessCertain1974,
  title = {On the {{Residual Finiteness}} of {{Certain Mapping Class Groups}}},
  author = {Grossman, Edna K.},
  year = 1974,
  month = nov,
  journal = {Journal of the London Mathematical Society},
  volume = {s2-9},
  number = {1},
  pages = {160--164},
  issn = {0024-6107},
  doi = {10.1112/jlms/s2-9.1.160},
  urldate = {2025-11-03}
}

@incollection{hainTorelliGroupsGeometry1996,
  title = {Torelli {{Groups}} and {{Geometry}} of {{Moduli Spaces}} of {{Curves}}},
  booktitle = {Current {{Topics}} in {{Complex Algebraic Geometry}}},
  author = {Hain, Richard M.},
  editor = {Clemens, Herbert and Koll{\'a}r, Janos},
  year = 1996,
  series = {Mathematical {{Sciences Research Institute Publications}}},
  pages = {97--144},
  publisher = {Cambridge University Press},
  address = {Cambridge},
  doi = {10.1017/9781009701877.008},
  urldate = {2025-10-01},
  isbn = {978-0-521-56244-7}
}

@article{harerVirtualCohomologicalDimension1986,
  title = {The Virtual Cohomological Dimension of the Mapping Class Group of an Orientable Surface},
  author = {Harer, John L.},
  year = 1986,
  month = feb,
  journal = {Inventiones mathematicae},
  volume = {84},
  number = {1},
  pages = {157--176},
  issn = {1432-1297},
  doi = {10.1007/BF01388737},
  urldate = {2024-11-27},
  langid = {english}
}

@misc{ivanovFifteenProblemsMapping2006,
  title = {Fifteen Problems about the Mapping Class Groups},
  author = {Ivanov, Nikolai V.},
  year = 2006,
  month = aug,
  number = {arXiv:math/0608325},
  eprint = {math/0608325},
  publisher = {arXiv},
  doi = {10.48550/arXiv.math/0608325},
  urldate = {2024-12-31},
  archiveprefix = {arXiv},
  langid = {english}
}

@article{kaluba$textAutmathbbF_5$2019,
  title = {{${\text {Aut}}({\mathbb {F}}_5)$} Has Property ({{T}})},
  author = {Kaluba, Marek and Nowak, Piotr W. and Ozawa, Narutaka},
  year = 2019,
  month = dec,
  journal = {Mathematische Annalen},
  volume = {375},
  number = {3-4},
  pages = {1169--1191},
  issn = {0025-5831, 1432-1807},
  doi = {10.1007/s00208-019-01874-9},
  urldate = {2025-02-27},
  langid = {english}
}

@article{kalubaProperty$mathrmAutF_n$2021,
  title = {On Property ({{T}}) for {$\mathrm {Aut}(F_n)$} and {$\mathrm {SL}_n(\mathbb {Z})$}},
  author = {Kaluba, Marek and Kielak, Dawid and Nowak, Piotr W.},
  year = 2021,
  month = mar,
  journal = {Annals of Mathematics},
  volume = {193},
  number = {2},
  pages = {539--562},
  publisher = {Department of Mathematics of Princeton University},
  issn = {0003-486X, 1939-8980},
  doi = {10.4007/annals.2021.193.2.3},
  urldate = {2025-11-04}
}

@article{kazhdanConnectionDualSpace1967,
  title = {Connection of the Dual Space of a Group with the Structure of Its Close Subgroups},
  author = {Kazhdan, D. A.},
  year = 1967,
  month = jan,
  journal = {Functional Analysis and Its Applications},
  volume = {1},
  number = {1},
  pages = {63--65},
  issn = {1573-8485},
  doi = {10.1007/BF01075866},
  urldate = {2025-02-27},
  langid = {english}
}

@article{kirillovInnerProductModular1996,
  title = {On an Inner Product in Modular Tensor Categories},
  author = {Kirillov, Alexander},
  year = 1996,
  journal = {Journal of the American Mathematical Society},
  volume = {9},
  number = {4},
  pages = {1135--1169},
  issn = {0894-0347, 1088-6834},
  doi = {10.1090/S0894-0347-96-00210-X},
  urldate = {2025-02-17},
  langid = {english}
}

@article{kirillovInnerProductModular1998,
  title = {On Inner Product in Modular Tensor Categories. {{II}}: {{Inner}} Product on Conformal Blocks and Affine Inner Product Identities},
  shorttitle = {On Inner Product in Modular Tensor Categories. {{II}}},
  author = {Kirillov, A. A.},
  year = 1998,
  journal = {Adv. Theor. Math. Phys.},
  volume = {2},
  pages = {155--180},
  doi = {10.4310/ATMP.1998.v2.n1.a6}
}

@article{klukowskiTangleFreePermutations2024,
  title = {Tangle {{Free Permutations}} and the {{Putman}}--{{Wieland Property}} of {{Random Covers}}},
  author = {Klukowski, Adam and Markovi{\'c}, Vladimir},
  year = 2024,
  month = oct,
  journal = {International Mathematics Research Notices},
  volume = {2024},
  number = {20},
  pages = {13400--13416},
  issn = {1073-7928, 1687-0247},
  doi = {10.1093/imrn/rnae206},
  urldate = {2025-02-27},
  copyright = {https://creativecommons.org/licenses/by-nc-nd/4.0/},
  langid = {english}
}

@article{koberdaIrreducibilityQuantumRepresentations2018,
  title = {Irreducibility of Quantum Representations of Mapping Class Groups with Boundary},
  author = {Koberda, Thomas and Santharoubane, Ramanujan},
  year = 2018,
  month = jul,
  journal = {Quantum Topology},
  volume = {9},
  number = {4},
  pages = {633--641},
  issn = {1663-487X},
  doi = {10.4171/qt/116},
  urldate = {2023-02-06},
  langid = {english}
}

@article{landesmanCanonicalRepresentationsSurface2024,
  title = {Canonical Representations of Surface Groups},
  author = {Landesman, Aaron and Litt, Daniel},
  year = 2024,
  month = mar,
  journal = {Annals of Mathematics},
  volume = {199},
  number = {2},
  pages = {823--897},
  publisher = {Department of Mathematics of Princeton University},
  issn = {0003-486X, 1939-8980},
  doi = {10.4007/annals.2024.199.2.6},
  urldate = {2024-10-13}
}

@misc{littMotivesMappingClass2024,
  title = {Motives, Mapping Class Groups, and Monodromy},
  author = {Litt, Daniel},
  year = 2024,
  month = sep,
  number = {arXiv:2409.02234},
  eprint = {2409.02234},
  primaryclass = {math},
  publisher = {arXiv},
  doi = {10.48550/arXiv.2409.02234},
  urldate = {2024-10-13},
  archiveprefix = {arXiv}
}

@misc{looijengaStableCohomologyMapping1995,
  title = {Stable Cohomology of the Mapping Class Group with Symplectic Coefficients and of the Universal {{Abel-Jacobi}} Map},
  author = {Looijenga, Eduard},
  year = 1995,
  month = mar,
  number = {arXiv:alg-geom/9401005},
  eprint = {alg-geom/9401005},
  publisher = {arXiv},
  doi = {10.48550/arXiv.alg-geom/9401005},
  urldate = {2025-01-03},
  archiveprefix = {arXiv},
  langid = {english}
}

@incollection{marcheIntroductionQuantumRepresentations2021,
  title = {Introduction to Quantum Representations of Mapping Class Groups},
  booktitle = {{{IRMA Lectures}} in {{Mathematics}} and {{Theoretical Physics}}},
  author = {March{\'e}, Julien},
  editor = {Papadopoulos, Athanase},
  year = 2021,
  month = jul,
  edition = {1},
  volume = {33},
  pages = {109--130},
  publisher = {EMS Press},
  doi = {10.4171/irma/33-1/7},
  urldate = {2023-11-18},
  isbn = {978-3-98547-001-3 978-3-98547-501-8},
  langid = {english}
}

@article{markovicUnramifiedCorrespondencesVirtual2022,
  title = {Unramified Correspondences and Virtual Properties of Mapping Class Groups},
  author = {Markovi{\'c}, Vladimir},
  year = 2022,
  month = dec,
  journal = {Bulletin of the London Mathematical Society},
  volume = {54},
  number = {6},
  pages = {2324--2337},
  issn = {0024-6093, 1469-2120},
  doi = {10.1112/blms.12696},
  urldate = {2025-01-31},
  langid = {english}
}

@article{mccarthyFirstCohomologyGroup2001,
  title = {On the First Cohomology Group of Cofinite Subgroups in Surface Mapping Class Groups},
  author = {McCarthy, John D.},
  year = 2001,
  month = mar,
  journal = {Topology},
  volume = {40},
  number = {2},
  pages = {401--418},
  issn = {0040-9383},
  doi = {10.1016/S0040-9383(99)00066-X},
  urldate = {2025-02-27}
}

@article{putmanAbelianQuotientsSubgroups2013,
  title = {Abelian Quotients of Subgroups of the Mapping Class Group and Higher {{Prym}} Representations},
  author = {Putman, Andrew and Wieland, Ben},
  year = 2013,
  month = aug,
  journal = {Journal of the London Mathematical Society},
  volume = {88},
  number = {1},
  pages = {79--96},
  issn = {00246107},
  doi = {10.1112/jlms/jdt001},
  urldate = {2024-10-03},
  copyright = {http://doi.wiley.com/10.1002/tdm\_license\_1.1},
  langid = {english}
}

@article{putmanNoteAbelianizationsFiniteindex2009,
  title = {A Note on the Abelianizations of Finite-Index Subgroups of the Mapping Class Group},
  author = {Putman, Andrew},
  year = 2009,
  month = sep,
  journal = {Proceedings of the American Mathematical Society},
  volume = {138},
  number = {02},
  pages = {753--758},
  issn = {0002-9939},
  doi = {10.1090/S0002-9939-09-10124-7},
  urldate = {2025-02-27},
  langid = {english}
}

@article{reshetikhinInvariants3manifoldsLink1991,
  title = {Invariants of 3-Manifolds via Link Polynomials and Quantum Groups},
  author = {Reshetikhin, N. and Turaev, V. G.},
  year = 1991,
  month = dec,
  journal = {Inventiones mathematicae},
  volume = {103},
  number = {1},
  pages = {547--597},
  issn = {1432-1297},
  doi = {10.1007/BF01239527},
  urldate = {2022-12-21},
  langid = {english}
}

@article{robertsIrreducibilityQuantumRepresentations2001,
  title = {Irreducibility of Some Quantum Representations of Mapping Class Groups},
  author = {Roberts, Justin},
  year = 2001,
  month = aug,
  journal = {Journal of Knot Theory and Its Ramifications},
  volume = {10},
  number = {05},
  pages = {763--767},
  publisher = {World Scientific Publishing Co.},
  issn = {0218-2165},
  doi = {10.1142/S021821650100113X},
  urldate = {2023-02-06}
}

@book{serreLocalFields2013,
  title = {Local {{Fields}}},
  author = {Serre, Jean-Pierre},
  year = 2013,
  month = jun,
  publisher = {Springer Science \& Business Media},
  googlebooks = {3LAJCAAAQBAJ},
  isbn = {978-1-4757-5673-9},
  langid = {english}
}

@article{simpsonHiggsBundlesLocal1992,
  title = {Higgs Bundles and Local Systems},
  author = {Simpson, Carlos T.},
  year = 1992,
  journal = {Publications Math\'ematiques de l'IH\'ES},
  volume = {75},
  pages = {5--95},
  issn = {1618-1913},
  urldate = {2023-02-03},
  langid = {english}
}

@article{simpsonLocalSystemsProper2011,
  title = {Local {{Systems}} on {{Proper Algebraic V}} -Manifolds},
  author = {Simpson, Carlos},
  year = 2011,
  journal = {Pure and Applied Mathematics Quarterly},
  volume = {7},
  number = {4},
  pages = {1675--1760},
  issn = {15588599, 15588602},
  doi = {10.4310/PAMQ.2011.v7.n4.a27},
  urldate = {2024-04-02},
  langid = {english}
}

@article{thomEnsemblesMorphismesStratifies1969,
  title = {Ensembles et Morphismes Stratifi\'es},
  author = {Thom, R.},
  year = 1969,
  month = mar,
  journal = {Bulletin of the American Mathematical Society},
  volume = {75},
  number = {2},
  pages = {240--284},
  publisher = {American Mathematical Society},
  issn = {0002-9904, 1936-881X},
  urldate = {2025-09-30}
}

@book{turaevQuantumInvariantsKnots2016,
  title = {Quantum {{Invariants}} of {{Knots}} and 3-{{Manifolds}}},
  author = {Turaev, Vladimir G.},
  year = 2016,
  month = jul,
  journal = {Quantum Invariants of Knots and 3-Manifolds},
  publisher = {De Gruyter},
  doi = {10.1515/9783110435221},
  urldate = {2022-12-21},
  isbn = {978-3-11-043522-1},
  langid = {english}
}

@article{verdierStratificationWhitneyTheoreme1976,
  title = {{Stratification de Whitney et th\'eor\`eme de Bertini-Sard.}},
  author = {Verdier, Jean-Louis},
  year = 1976,
  journal = {Inventiones mathematicae},
  volume = {36},
  pages = {295--312},
  issn = {0020-9910; 1432-1297/e},
  urldate = {2025-09-30},
  langid = {und}
}

@article{vogtmannAutomorphismsFreeGroups2002,
  title = {Automorphisms of {{Free Groups}} and {{Outer Space}}},
  author = {Vogtmann, Karen},
  year = 2002,
  month = oct,
  journal = {Geometriae Dedicata},
  volume = {94},
  number = {1},
  pages = {1--31},
  issn = {1572-9168},
  doi = {10.1023/A:1020973910646},
  urldate = {2025-11-04},
  langid = {english}
}

@article{wenzlTensorCategoriesQuantum1998,
  title = {{$C$}* Tensor Categories from Quantum Groups},
  author = {Wenzl, Hans},
  year = 1998,
  journal = {Journal of the American Mathematical Society},
  volume = {11},
  number = {2},
  pages = {261--282},
  issn = {0894-0347, 1088-6834},
  doi = {10.1090/S0894-0347-98-00253-7},
  urldate = {2025-07-08},
  langid = {english}
}

@misc{wuCompactKahlerOrbifold2023,
  title = {On Compact {{K\"ahler}} Orbifold},
  author = {Wu, Xiaojun},
  year = 2023,
  month = feb,
  number = {arXiv:2302.11914},
  eprint = {2302.11914},
  primaryclass = {math},
  publisher = {arXiv},
  doi = {10.48550/arXiv.2302.11914},
  urldate = {2025-10-01},
  archiveprefix = {arXiv},
  langid = {english}
}

@misc{zhongPrymRepresentationsTwisted2025,
  title = {Prym {{Representations}} and {{Twisted Cohomology}} of the {{Mapping Class Group}} with {{Level Structures}}},
  author = {Zhong, Xiyan},
  year = 2025,
  month = jul,
  number = {arXiv:2401.13869},
  eprint = {2401.13869},
  primaryclass = {math},
  publisher = {arXiv},
  doi = {10.48550/arXiv.2401.13869},
  urldate = {2025-11-07},
  archiveprefix = {arXiv},
  langid = {english}
}

\end{document}